\documentclass[a4paper]{amsart}                

\usepackage[latin1]{inputenc}                  
\usepackage[T1]{fontenc}                       
\usepackage[spanish,english]{babel}            
\usepackage{graphicx}                  
\usepackage{amsmath,amssymb,amsthm}            
\usepackage{latexsym}                          
\usepackage{delarray}                          
\usepackage{bbm}                               
\usepackage{hyperref}                          
\usepackage[pdftex,usenames,dvipsnames]{color} 
\usepackage{calrsfs,eufrak,mathrsfs,upgreek}   
\usepackage{bbding,trfsigns}                   
\usepackage{wasysym}                           
\usepackage{datetime}                          
\usepackage{color}                             

\DeclareMathAlphabet{\mathpzc}{OT1}{pzc}{m}{it} 

\newtheorem{teo}{Theorem}[section]
\newtheorem{propo}{Proposition}[section]
\newtheorem{coro}{Corollary}[section]
\newtheorem{lema}{Lemma}[section]
\newtheorem{remark}{Remark}[section]

\newcommand{\N}{\mathbb{N}}

\newcommand{\Z}{\mathbb{Z}}

\DeclareMathAlphabet{\mathpzc}{OT1}{pzc}{m}{it}

\DeclareFontFamily{U}{mathx}{\hyphenchar\font45}
\DeclareFontShape{U}{mathx}{m}{n}{
      <5> <6> <7> <8> <9> <10>
      <10.95> <12> <14.4> <17.28> <20.74> <24.88>
      mathx10
      }{}
\DeclareSymbolFont{mathx}{U}{mathx}{m}{n}
\DeclareFontSubstitution{U}{mathx}{m}{n}
\DeclareMathAccent{\widecheck}{0}{mathx}{"71}
\DeclareMathAccent{\wideparen}{0}{mathx}{"75}

\allowdisplaybreaks

\author[J.J. Betancor]{Jorge J. Betancor$^1$}
\author[A.J. Castro]{Alejandro J. Castro$^2$}
\author[J.C. Fari\~na]{Juan C. Fari\~na$^1$}
\author[L. Rodr\'iguez-Mesa]{L. Rodr\'iguez-Mesa$^1$}

\address{$^1$
        Departamento de An\'alisis Matem\'atico,
        Universidad de La Laguna, 
        Campus de Anchieta, Avda. Astrof\'{\i}sico Francisco S\'anchez, s/n, 
        38271, La Laguna (Sta. Cruz de Tenerife), Spain.}
\email{jbetanco@ull.es, jcfarina@ull.es, lrguez@ull.es}

\address{$^{2}$ Department of Mathematics, Uppsala University, S-751 06 Uppsala, Sweden.}
\email{alejandro.castro@math.uu.se}

\thanks{This paper is partially supported by MTM2013-44357-P. The second author is partially supported by Swedish Research Council Grant 621-2011-3629.}

\date{\today}

\begin{document}

\title[Discrete harmonic analysis associated with ultraspherical expansions]
{Discrete harmonic analysis associated with ultraspherical expansions}



\begin{abstract}
In this paper we study discrete harmonic analysis associated with ultraspherical orthogonal functions. We establish weighted $\ell ^p$-boundedness properties of maximal operators and Littlewood-Paley $g$-functions defined by Poisson and heat semigroups generated by the difference operator
$$
\Delta _\lambda f(n):=a_n^\lambda f(n+1)-2f(n)+a_{n-1}^\lambda f(n-1),\quad n\in \mathbb{N},\;\lambda >0,
$$
where $a_n^\lambda :=\{(2\lambda +n)(n+1)/[(n+\lambda )(n+1+\lambda )]\}^{1/2}$, $n\in \mathbb{N}$, and $a_{-1}^\lambda:=0$. We also prove weighted $\ell ^p$-boundedness properties of transplantation operators associated with the system $\{\varphi _n^\lambda \}_{n\in \mathbb{N}}$ of ultraspherical  functions, a family of eigenfunctions of $\Delta _\lambda$. In order to show our results we previously establish a vector-valued local Calder\'on-Zygmund theorem in our discrete setting.

\end{abstract}

\maketitle

\section{Introduction}\label{sec:intro}
The study of harmonic analysis in discrete settings has attracted considerable attention in the last years. For instance, harmonic analysis on graphs has been studied in  \cite{BM}, \cite{BR}, \cite{KM}, \cite {R1} and \cite{R2}, and it has been considered on discrete groups in \cite{DP}, \cite{D}, \cite {EOT}, \cite{GT}, \cite{K}, \cite{L2} and \cite{L1}. Also, celebrated mathematicians have investigated discrete analogues of Euclidean harmonic analysis problems where the underlying real field $\mathbb{R}$ is replaced by the ring of integers $\mathbb{Z}$. In this discrete context the exponential sums play the role of oscillatory integrals in the Euclidean setting. Some of these problems are studied in \cite{B1}, \cite{B2}, \cite{IMSW}, \cite{IW}, \cite{SW1}, \cite{SW2} and \cite{SW3}.

As far as we know Titchmarsh (\cite{Ti}) was the first one who investigated the $\ell ^p$-boundedness properties of discrete harmonic analysis operators (see also \cite{Ri}). For every $1\leq p\leq \infty $ we denote as usual by $\ell ^p(\mathbb{Z})$ the space that consists of all those complex sequences $(a_n)_{n\in \mathbb{Z}}$ such  that $\|(a_n)_{n\in \mathbb{Z}}\|_{\ell ^p(\mathbb{Z})}<\infty $, where
\begin{equation}\label{normalp}
\|(a_n)_{n\in \mathbb{Z}}\|_{\ell ^p(\mathbb{Z})}:=\left\{\begin{array}{ll}
\displaystyle \Big(\sum_{n\in \mathbb{Z}}|a_n|^p\Big)^{1/p},&1\leq p<\infty,\\
			&\\
			\displaystyle \sup_{n\in \mathbb{Z}}|a_n|,&p=\infty .
		\end{array}
\right.
\end{equation}
The convolution operation on $\mathbb{C}^\mathbb{N}$ associated to the usual group operation on $\mathbb{Z}$ is defined as follows: if $a=(a_n)_{n\in \mathbb{Z}}\in \mathbb{C}^\mathbb{Z}$, and $b=(b_n)_{n\in \mathbb{Z}}\in \mathbb{C}^\mathbb{Z}$, the convolution $a*b\in \mathbb{C}^\mathbb{Z}$ is given by
$$
(a*b)_n=\sum_{m\in \mathbb{Z}}a_mb_{n-m},\quad n\in \mathbb{Z},
$$
provided that the last sum converges for every $n\in \mathbb{Z}$. As it is wellknown,  the Young's inequality holds for $*$ in $\ell ^p(\mathbb{Z})$ spaces. The discrete Hilbert transform $\mathcal{H}_\mathbb {Z}$ was defined in \cite{Ti} as the $*$-convolution operator with the kernel $k=(k(n))_{n\in \mathbb{Z}}$, where $k(n):=(\pi (n+1/2))^{-1}$, $n\in \mathbb{Z}$, and the convergence of the series is understood as a principal value, i.e., as the limit of partial sums from -M to M. $\mathcal{H}_\mathbb{Z}$ is bounded from $\ell ^p(\mathbb{Z})$ into itself, for every $1<p<\infty$. As in the continuous case, when $p=1$ the situation is different. The discrete Hilbert transform is bounded from $\ell ^1(\mathbb{Z})$ into $\ell ^{1,\infty }(\mathbb{Z})$ where $\ell ^{1,\infty }(\mathbb{Z})$ denotes the weak-$\ell ^1$ space (see \cite[Proposition 5 and Corollary 2]{CGRTV}). As the transference results show the $L^p$-boundedness properties of continuous and their discrete analogues operators are closely connected  (see, for instance, \cite{AC} and \cite{CZ}).

The discrete Hilbert transform also adopts other forms. It is usual to consider the operator $\mathfrak{H}_\mathbb{Z}$ defined by
$$
\mathfrak{H}_\mathbb{Z}(f)(n):=\frac{1}{\pi }\sum_{m\in \mathbb{Z},m\not=n}\frac{f(m)}{n-m},\quad f\in \ell ^p(\mathbb{Z}),\;1\leq p<\infty ,
$$
(see \cite{An1}, \cite{An2} and \cite{HMW}).

The operators $\mathfrak{H}_\mathbb{Z}$ and $\mathcal{H}_\mathbb{Z}$ map $\ell ^2(\mathbb{Z})$ into itself with norm 1 (\cite{Gra} and \cite{Ti}). Laeng (\cite{La}) has investigated the norms of discrete Hilbert transforms as bounded operators in $\ell ^p(\mathbb{Z})$, $1<p<\infty $.

Recently, Ciaurri et al. (\cite{CGRTV}) studied discrete harmonic analysis operators related to the discrete Laplacian $\Delta _\mathbb{Z}$ defined by
$$
(\Delta _\mathbb{Z}f)(n):=f(n+1)-2f(n)+f(n-1),\quad n\in \Z,
$$
for every $f=(f(n))_{n\in \mathbb{Z}}\in \mathbb{C}^\mathbb{Z}$. By using heat and Poisson semigroups associated with $\Delta _\mathbb{Z}$ they defined maximal operators, fractional powers of $\Delta _\mathbb{Z}$, Littlewood-Paley $g$-functions and Riesz transforms. The discrete Hilbert transform $\mathcal{H}_\mathbb{Z}$ appears as a Riesz transform. $\ell ^p$-boundedness properties of those operators are studied in \cite{CGRTV} by employing scalar and vector-valued Calder\'on-Zygmund theory (\cite{RRT}).

Motivated by \cite{CGRTV}, in this paper we develope a discrete harmonic analysis associated with ultraspherical expansions.

Assume that $\lambda >0$. For every $n\in \mathbb{N}$ we consider the $n$-th ultraspherical polynomial of order $\lambda $ defined by (\cite[\S 4.7]{Sz})
\begin{equation}\label{ultraspherical}
\mathcal{P}_n^\lambda (x):=\frac{(-1)^n}{2^n(\lambda +1/2)_n}(1-x^2)^{1/2-\lambda }\frac{d^n}{dx^n}(1-x^2)^{n+\lambda -1/2}, \quad x\in [-1,1].
\end{equation}
Here $(a)_n=a(a+1)\cdots (a+n-1)$, for each $a>0$, $n\in \mathbb{N}$. We have that, for every $n,m\in \mathbb{N}$,
$$
\int_{-1}^1\mathcal{P}_n^\lambda (x)\mathcal{P}_m^\lambda (x)(1-x^2)^{\lambda -1/2}dx=\frac{\delta_{n,m}}{w_\lambda (n)},
$$
where $\delta _{n,m}$ represents the Kronecker's delta and
$$
w_\lambda (n):=\frac{\Gamma (\lambda )(2\lambda )_n(n+\lambda )}{\sqrt{\pi }\Gamma (\lambda +1/2)n!}.
$$
By taking into account that, for certain $C\geq 1$,
\begin{equation}\label{3.7}
\frac{1}{C}(m+1)^{\alpha -1}\leq \frac{(\alpha )_m}{m!}\leq C(m+1)^{\alpha -1},\quad \alpha >0\mbox{ and }m\in \mathbb{N},
\end{equation}
we get
\begin{equation}\label{3.8}
\frac{1}{C}(n+1)^{2\lambda}\leq w_\lambda (n)\leq C(n+1)^{2\lambda},\quad n\in \mathbb{N}.
\end{equation}

We consider for each $n\in \mathbb{N}$ the ultraspherical function $\varphi _n^\lambda $ defined by
$$
\varphi _n^\lambda (x):=\sqrt{w_\lambda(n)}\mathcal{P}_n^\lambda (x)(1-x^2)^{\lambda /2-1/4},\quad x\in (-1,1).
$$
The sequence $\{\varphi _n^\lambda \}_{n\in \mathbb{N}}$ is an orthonormal basis in $L^2(-1,1)$.

According to \cite[(4.7.17)]{Sz}, we have that
\begin{equation}\label{1.1} 2x\varphi _n^\lambda (x)=a_n^\lambda \varphi _{n+1}^\lambda (x)+a_{n-1}^\lambda \varphi _{n-1}^\lambda (x),\quad x\in (-1,1)\mbox{ and }n\in \mathbb{N},
\end{equation}
where
$$
a_n^\lambda:=\sqrt{\frac{(2\lambda +n)(n+1)}{(n+\lambda )(n+1+\lambda )}},\quad n\in \mathbb{N}.
$$
Here and in the sequel $\varphi _{-1}^\lambda :=0$ and $a_{-1}^\lambda :=0$.

We consider the $\lambda $- Laplacian operator given by 
$$
(\Delta _\lambda f)(n):=a_n^\lambda f(n+1)-2f(n)+a_{n-1}^\lambda f(n-1),\quad n\in \mathbb{N},\;f=(f(n))_{n\in \mathbb{N}}\in \mathbb{C}^\mathbb{N}.
$$
Note that $\Delta _\lambda $ reduces to the discrete Laplacian $\Delta$ on $\mathbb{N}$ in the end point $\lambda =0$.

For every $1\leq p\leq \infty $ we denote by $\ell ^p(\mathbb{N})$ the space constituted by all those complex sequences $(a_n)_{n\in \mathbb{N}}$ such that
$\|(a_n)_{n\in \mathbb{N}}\|_{\ell ^p(\mathbb{N})}<\infty $, where $\|\cdot\|_{\ell ^p(\mathbb{N})}$ is naturally defined as in \eqref{normalp} with $\mathbb{Z}$ replaced by $\mathbb{N}$.

The operator $\Delta _\lambda $ is selfadjoint in $\ell ^2(\mathbb{N})$ and bounded in $\ell ^p(\mathbb{N})$, for every $1\leq p\leq \infty$.

We define the $\lambda$-transform $\mathcal{F}_\lambda (f)$ of $f=(f(n))_{n\in \mathbb{N}}\in \ell ^2(\mathbb{N})$, by
$$
\mathcal{F}_\lambda (f)=\sum_{n=0}^\infty f(n)\varphi _n^\lambda .
$$
Parseval's equality says that $\mathcal{F}_\lambda$ is an isometry from $\ell^2(\mathbb{N})$ into $L^2(-1,1)$.

From (\ref{1.1}) we deduce that, for every $f\in \ell ^2(\mathbb{N})$,
\begin{equation}\label{1.19}
\mathcal{F}_\lambda(\Delta _\lambda f)(x)=2(x-1)\mathcal{F}_\lambda(f)(x),\quad x\in (-1,1).
\end{equation}
By using again Parseval's equality we obtain
$$
\sum_{n=0}^\infty (\Delta _\lambda f)(n)\overline{f(n)}=\int_{-1}^12(x-1)|\mathcal{F}_\lambda(f)(x)|^2dx\leq 0,\quad f=(f(n))_{n\in \mathbb{N}}\in \ell ^2(\mathbb{N}).
$$
Thus, we show that $-\Delta _\lambda $ is a positive operator in $\ell ^2(\mathbb{N})$.

Since $\Delta _\lambda $ is bounded in $\ell ^p (\mathbb{N})$, $1\leq p\leq \infty$,  we have that $\Delta _\lambda$ generates a semigroup of operators $\{W_t^\lambda :=e^{t\Delta _\lambda }\}_{t>0}$ in $\ell ^p(\mathbb{N})$, $1\leq p\leq \infty$, such that
$$
\lim_{t\rightarrow 0^+}e^{t\Delta _\lambda }f=f,\quad \mbox{ for every }f\in \ell ^p(\mathbb{N}).
$$

We can obtain an expression for $W_t^\lambda $, $t>0$, in terms of a convolution operation $\# _\lambda$ that is well adapted to our discrete ultraspherical setting. This $\# _\lambda$ convolution is a modification of the one introduced in \cite{Hi1}.

If $f=(f(m))_{m\in \mathbb{N}}\in \mathbb{C}^\mathbb{N}$ and $n\in \mathbb{N}$ the $\lambda$-translated $_\lambda \tau _nf$ is defined by
\begin{equation}\label{1.3}
(_\lambda \tau _nf)(m):=\sum_{k=0}^\infty c_\lambda (n,m,k)f(k),\quad m\in \mathbb{N},
\end{equation}
where, for every $m,k\in \mathbb{N}$,
$$
c_\lambda (n,m,k):=\int_{-1}^1\varphi _n^\lambda (x)\varphi _m^\lambda (x)\varphi _k^\lambda (x)(1-x^2)^{1/4-\lambda /2}dx.
$$
According to \cite[(1.1)]{Hsu} we have that
\begin{equation}\label{clambda}
c_\lambda (n,m,k)=\sqrt{w_\lambda (n)w_\lambda (m)w_\lambda (k)}\frac{n!m!k!}{(2\lambda )_n(2\lambda )_m(2\lambda )_k}\frac{(\lambda )_{\sigma -n}(\lambda )_{\sigma -m}(\lambda )_{\sigma -k}}{(\sigma -n)!(\sigma -m)!(\sigma -k)!}\frac{2^{1-2\lambda}\pi \Gamma (\sigma +2\lambda )}{\Gamma (\lambda )\Gamma (\sigma +\lambda +1)},
\end{equation}
when $n,m,k\in \mathbb{N}$, $|n-m|\leq k\leq n+m$ and $n+m+k=2\sigma $, for some $\sigma \in \mathbb{N}$. Otherwise $c_\lambda (n,m,k)=0$. Note that the series in (\ref{1.3}) is actually a finite sum.

Moreover, by using (\ref{3.7}) and (\ref{3.8}) we have that, there exists $C>0$ such that, for each $n,m,k\in \mathbb{N}$, such that $n+m+k=2\sigma$, with $\sigma \in \mathbb{N}$, and $|n-m|\leq k\leq n+m$, 
\begin{align}\label{acotclambda}
\frac{1}{C}\left(\frac{\sigma (\sigma -n+1)(\sigma -m+1)(\sigma -k+1)}{(n+1)(m+1)(k+1)}\right)^{\lambda -1}&\nonumber\\
&\hspace{-3cm}\leq c_\lambda (n,m,k)\leq C\left(\frac{\sigma (\sigma -n+1)(\sigma -m+1)(\sigma -k+1)}{(n+1)(m+1)(k+1)}\right)^{\lambda -1}.
\end{align}

We remark that there is not a group operation $\circ$ on $\mathbb{N}$ such that $(_\lambda \tau _nf)(m)=f(n\circ m^{-1})$, for every $n,m\in \mathbb{N}$, where $m^{-1}$ represents the inverse of $m$ with respect to $\circ$.

If $f=(f(n))_{n\in \mathbb{N}}\in \mathbb{C}^\mathbb{N}$ and $g=(g(n))_{n\in \mathbb{N}}\in \mathbb{C}^\mathbb{N}$, the $\# _\lambda $-convolution $f\#_\lambda g$ of $f$ and $g$ is defined by
$$(f\#_\lambda g)(n):=\sum_{m=0}^\infty f(m)(_\lambda \tau _ng)(m),\quad n\in \mathbb{N},
$$
provided that the last series converges. 

The triple $(\mathbb{N},\mu ,\#_\lambda )$, is an hypergroup (\cite{BH}), where $\mu$ is the measure $\sum_{n\in \mathbb{N}}\delta _n$, and $\delta _n$, $n\in \mathbb{N}$, is the point mass probability measure supported on $n$. From \cite[(11)]{Hi1} we deduce that
\begin{equation}\label{1.4}
\mathcal{F}_\lambda (f\#_\lambda g)(x)=(1-x^2)^{-\lambda /2+1/4}\mathcal{F}_\lambda (f)(x)\mathcal{F}_\lambda (g)(x),\quad x\in (-1,1),
\end{equation}
for every $f,g\in \mathbb{C}^\mathbb{N}$ such that $\sqrt{w_\lambda }f\in \ell ^1(\mathbb{N})$ and $\sqrt{w_\lambda }g\in \ell ^1(\mathbb{N})$.

Analytic continuation and \cite[10.9 (38)]{EMOT} allow us to write, for every $n\in \mathbb{N}$,
\begin{equation}\label{1.2}
\int_{-1}^1e^{tx}\varphi _n^\lambda (x)(1-x^2)^{\lambda /2-1/4}dx=2^\lambda \sqrt{\pi }\Gamma \Big(\lambda +\frac{1}{2}\Big)\sqrt{w_\lambda (n)}t^{-\lambda }I_{\lambda +n}(t),\quad t>0.
\end{equation}
Here, $I_\nu$ denotes the modified Bessel function of the first kind and order $\nu$. By taking into account (\ref{1.19}) and (\ref{1.2}) we obtain that, for every $t>0$,
\begin{equation}\label{5prime}
W_t^\lambda (f)=h_t^\lambda \#_\lambda f,\quad f\in \ell ^p(\mathbb{N}),\;1\leq p\leq \infty ,
\end{equation}
where
\begin{equation}\label{hlambda}
h_t^\lambda (n):=\sqrt{\pi }\Gamma \Big(\lambda +\frac{1}{2}\Big)\sqrt{w_\lambda (n)}e^{-2t}t^{-\lambda }I_{\lambda +n}(2t),\quad n\in \mathbb{N}.
\end{equation}

By using the subordination formula, the Poisson semigroup $\{P_t^\lambda \}_{t>0}$ associated with $\Delta _\lambda $ (gene\-rated by $-\sqrt{-\Delta _\lambda}$) is defined by
\begin{equation}\label{1.5}
P_t^\lambda (f)(n):=\frac{1}{\sqrt{\pi}}\int_0^\infty \frac{e^{-u}}{\sqrt{u}}W_{t^2/(4u)}^\lambda (f)(n)du,\quad n\in \mathbb{N},\;f\in \ell ^p(\mathbb{N}),\mbox{ and }1\leq p\leq \infty .
\end{equation}
The semigroups $\{W_t^\lambda \}_{t>0}$ and $\{P_t^\lambda \}_{t>0}$ are not Markovian, that is, they do not map  constants into constants (see Section \ref{S3}).

From now on, let $\{T_t^\lambda \}_{t>0}$ represents the heat or the Poisson semigroup associated with $\Delta _\lambda$. If $k\in \mathbb{N}\setminus\{0\}$, we define the Littlewood-Paley $g_{T^\lambda }^k$ of $k$-th order by
$$
g_{T^\lambda }^k(f)(n):=\left(\int_0^\infty |t^k\partial _t^k(T_t^\lambda f)(n)|^2\frac{dt}{t}\right)^{1/2},\quad n\in \mathbb{N},
$$
and the maximal operator $T_*^\lambda $ by
$$
T_*^\lambda f:=\sup_{t>0}|T_t^\lambda f|.
$$

Suppose that $w=(w(n))_{n\in \mathbb{N}}\in (0,\infty )^\mathbb{N}$. We say that $w\in A_p(\mathbb{N})$ provided that
$$
\sup_{\substack{0\leq n\leq m\\n,m\in \mathbb{N}}}\frac{1}{(m-n+1)^p}\left(\sum_{k=n}^mw(k)\right)\left (\sum_{k=n}^mw(k)^{-1/(p-1)}\right)^{p-1}<\infty ,\quad \mbox{ when }1<p<\infty ,
$$
and
$$
\sup _{\substack{0\leq n\leq m\\n,m\in \mathbb{N}}}\frac{1}{m-n+1}\left(\sum_{k=n}^mw(k)\right)\max _{n\leq k \leq m}\frac{1}{w(k)}<\infty ,\quad \mbox{ when }p=1.
$$

For every $1\leq p<\infty$ and $w\in A_p(\mathbb{N})$ we denote by $\ell^p(\mathbb{N},w)$ and $\ell ^{p,\infty }(\mathbb{N},w)$ the usual weighted and weak weighted $\ell ^p$ space, respectively.

By using vector-valued local Calder\'on-Zygmund theory, we establish the $\ell ^p$-boundedness pro\-per\-ties for our discrete $g$-functions and maximal operators.

\begin{teo}\label{Th1}
Let $\lambda >0$ and $k\in \mathbb{N}\setminus\{0\}$. Then,

(i) If $1<p<\infty $ and $w\in A_p(\mathbb{N})$, $g_{T^\lambda}^k$ and $T_*^\lambda $ are bounded from $\ell ^p(\mathbb{N},w)$ into itself.

(ii) If $w\in A_1(\mathbb{N})$, $g_{W^\lambda }^1$, $g_{P^\lambda }^k$ and $T_*^\lambda $ are bounded from $\ell ^1(\mathbb{N}, w)$ into $\ell ^{1,\infty }(\mathbb{N}, w)$.
\end{teo}

Let $1\leq p\leq \infty $. Since $\Delta _\lambda $ is bounded in $\ell ^p(\mathbb{N})$, for every $f\in \ell ^p(\mathbb{N})$, $f=\lim_{t\rightarrow 0^+}W_t^\lambda f$ in $\ell ^p(\mathbb{N})$. Hence, for every $f\in \ell ^p(\mathbb{N})$,
$$
f(n)=\lim_{t\rightarrow 0^+}(W_t^\lambda f)(n),\quad n\in \mathbb{N}.
$$
Subordination formula (\ref{1.4}) allows us to obtain the same convergence properties for the Poisson semigroup $\{P_t^\lambda \}_{t>0}$.

From Theorem \ref{Th1} and by using density arguments we get the following.

\begin{coro}\label{Cor1}
Let $\lambda >0$, $1\leq p< \infty $ and $w\in A_p(\mathbb{N})$. Then, for every $f\in \ell ^p(\mathbb{N},w)$,
$$
f(n)=\lim_{t\rightarrow 0^+}(T_t^\lambda f)(n),\quad n\in \mathbb{N}.
$$
Moreover, if $1<p<\infty$, $f=\lim_{t\rightarrow 0^+}T_t^\lambda f$, in $\ell ^p(\mathbb{N},w)$.
\end{coro}

As a consequence of Theorem \ref{Th1} we can see that the Littlewood-Paley functions $g_{W^\lambda }^k$ and $g_{P^\lambda }^k$, $k\in \mathbb{N}\setminus\{0\}$, define equivalent norms in $\ell ^p(\mathbb{N},w)$, $1<p<\infty$ and $w\in A_p(\mathbb{N})$.

\begin{propo}\label{Prop1}
Let $\lambda >0$, $k\in \mathbb{N}\setminus\{0\}$, $1< p< \infty $ and $w\in A_p(\mathbb{N})$.  Then, there exists $C>0$ such that, for every $f\in \ell ^p(\mathbb{N},w)$,
\begin{equation}\label{acotg}
\frac{1}{C}\|f\|_{\ell ^p(\mathbb{N},w)}\leq \|g_{T^\lambda }^k(f)\|_{\ell ^p(\mathbb{N},w)}\leq C\|f\|_{\ell ^p(\mathbb{N},w)}.
\end{equation}
\end{propo}

In \cite{AW} Askey and Wainger established a transplantation theorem for ultraspherical coefficients (see \cite[Theorem 1]{AW}). They obtained weighted $\ell ^p$ inequalities with power $A_p(\mathbb{N})$-weights for $1<p<\infty$. By using discrete local Calder\'on-Zygmund theory we extend \cite[Theorem 1]{AW} to general $A_p(\mathbb{N})$-weights for $1<p<\infty$ and also we get results for $p=1$.

Suppose that $\lambda , \mu >0$. We consider the operator $\mathcal{T}_{\lambda , \mu }$ defined by $\mathcal{T}_{\lambda ,\mu }:=\mathcal{F}_\mu ^{-1}\mathcal{F}_\lambda$  on $\ell ^2(\mathbb{N})$.

$\ell ^p$-boundedness properties of the operator $\mathcal{T}_{\lambda , \mu }$ are established in the following.

\begin{teo}\label{Th2}
Let $\lambda ,\mu >1$. If $1< p< \infty $ and $w\in A_p(\mathbb{N})$, then there exits $C>0$ such that
$$
\frac{1}{C}\|f\|_{\ell ^p(\mathbb{N},w)}\leq \|\mathcal{T}_{\lambda ,\mu }(f)\|_{\ell ^p(\mathbb{N},w)}\leq C\|f\|_{\ell ^p(\mathbb{N},w)},\quad f\in \ell ^p(\mathbb{N},w).
$$
If $(\mu -1)/2<\lambda <\mu$ and $w\in A_1(\mathbb{N})$, the operator $\mathcal{T}_{\lambda , \mu }$ is bounded from $\ell ^1(\mathbb{N},w)$ into $\ell ^{1,\infty }(\mathbb{N},w)$.
\end{teo}

$\mathcal{T}_{\lambda , \mu}$ is really a transplantation operator and Theorem \ref{Th2} extends \cite[Theorem 1]{AW}. Indeed, let $1<q<\infty$, $v$ a weight in the Muckenhoupt class $A_q(-1,1)$  and $F\in L^q((-1,1),v)$. According to \cite[Theorem 2]{Ke} (see also the proof of \cite[Proposition 2.2]{ABCSS}), we have that
$$
S_n^\lambda (F)=\sum_{k=0}^nc_k^\lambda (F)\varphi _k^\lambda \longrightarrow F,\quad \mbox{ as }n\rightarrow \infty ,
$$
in $L^q((-1,1),v)$. Here, for every $k\in \mathbb{N}$,
$$
c_k^\lambda (F):=\int_{-1}^1F(x)\varphi _k^\lambda (x)dx.
$$
Since $|\mathcal{P}_k^\gamma (x)|\leq 1$, $k\in \mathbb{N}$ and $\gamma >0$, (\cite[Theorem 7.33.1]{Sz}), 
 it follows that
$$
c_m^\mu (F)=\lim_{n\rightarrow \infty }\int_{-1}^1S_n^\lambda (F)(x)\varphi _m^\mu (x)dx=\lim_{n\rightarrow \infty }\sum_{k=0}^nc_k^\lambda (F)\int_{-1}^1\varphi _k^\lambda (x)\varphi _m^\mu (x)dx=\mathcal{T}_{\lambda , \mu }(f)(m),\quad m\in \mathbb{N},
$$
where $f(k)=c_k^\lambda (F)$, $k\in \mathbb{N}$.

From Theorem \ref{Th2} we deduce the following generalization of \cite[Theorem 1]{AW}.
\begin{coro}\label{Coro2}
Let $\lambda ,\mu >1$. Assume that $1< p,q< \infty $, $w\in A_p(\mathbb{N})$ and $v\in A_q(-1,1)$. Then, there exits $C>0$ such that, for every $F\in L^q((-1,1),v)$,
$$
\frac{1}{C}\|(c_k^\lambda (F))_{k\in \mathbb{N}}\|_{\ell ^p(\mathbb{N},w)}\leq \|(c_k^\mu (F))_{k\in \mathbb{N}}\|_{\ell ^p(\mathbb{N},w)}\leq C\|(c_k^\lambda (F))_{k\in \mathbb{N}}\|_{\ell ^p(\mathbb{N},w)}.
$$
\end{coro}

The transplantation operator can also be seen as an extension of Riesz transform operators (see, for instance, \cite{Stem}).

In Sections \ref{S3}, \ref{S4} and \ref{S5} we present proofs of our results. In Section \ref{S2} we establish a discrete vector-valued local Calder\'on-Zygmund theorem that will be very useful to prove Theorems \ref{Th1} and \ref{Th2}.

Throughout this paper by $C$ we always denote a positive constant that can change in each ocurrence.

\section{Discrete vector-valued local Calder\'on-Zygmund operators}\label{S2}

Nowak and Stempak \cite{NS} developed the so called local Calder\'on-Zygmund theory that allows us to treat singular integrals on $(0,\infty )$. They used it to obtain $L^p$-boundedness properties for transplantation operators in the Bessel settings (\cite[Proposition 4.2]{NS}). Banach valued singular integral operators were investigated by Rubio de Francia, Ruiz and Torrea \cite{RRT} (see also \cite{BCP}). Recently, Grafakos, Liu and Yang \cite{GLY} have established  a Banach valued version of Calder\'on-Zygmund theory for singular integrals on spaces of homogeneous type.

In order to show Theorem \ref{Th1} we need to establish the following result that is a local version of \cite[Theorem 1.1]{GLY} for Banach valued Calder\'on-Zygmund operators on the space $(\mathbb{N},\mu ,|\cdot |)$ of homogeneous type. Here, $\mu$ as above denotes the measure $\mu =\sum_{n\in \mathbb{N}}\delta _n$ on $\mathbb{N}$, and $|\cdot |$ represents the usual metric on $\mathbb{N}$. Next result can also be seen as a discrete version of \cite[Proposition 4.2]{NS}.

Suppose that $\mathbb{B}_1$ and $\mathbb{B}_2$ are Banach spaces. By $\mathcal{L}(\mathbb{B}_1,\mathbb{B}_2)$ we denote the space of bounded linear operators from $\mathbb{B}_1$ into $\mathbb{B}_2$. Assume also that the function
$$
K:(\mathbb{N}\times \mathbb{N})\setminus D\longrightarrow \mathcal{L}(\mathbb{B}_1,\mathbb{B}_2),
$$
where $D:=\{(n,n):n\in \mathbb{N}\}$, is measurable, and that for certain $C>0$ the following conditions are satisfied, for each $n,m,\ell \in \mathbb{N}$, $n\not =m$:

$(a)$ the size condition:
$$
\|K(n,m)\|_{\mathcal{L}(\mathbb{B}_1,\mathbb{B}_2)}\leq \frac{C}{|n-m|},
$$

$(b)$ the regularity properties:

$(b1)$ $\displaystyle \|K(n,m)-K(\ell ,m)\|_{\mathcal{L}(\mathbb{B}_1,\mathbb{B}_2)}\leq C\frac{|n-\ell|}{|n-m|^2},\quad |n-m|>2|n-\ell |,\;\frac{m}{2}<n,\ell <\frac{3m}{2}$.

$(b2)$ $\displaystyle \|K(m,n)-K(m,\ell )\|_{\mathcal{L}(\mathbb{B}_1,\mathbb{B}_2)}\leq C\frac{|n-\ell|}{|n-m|^2},\quad |n-m|>2|n-\ell |,\;\frac{m}{2}<n,\ell <\frac{3m}{2}$.

We say that $K$ is a local $\mathcal{L}(\mathbb{B}_1,\mathbb{B}_2)$-standard kernel when the above conditions are satisfied.

Here and in the sequel, if $\mathbb{B}$ is a Banach space we denote by $\ell _\mathbb{B}^p(\mathbb{N})$ the $\mathbb{B}$-valued $\ell ^p(\mathbb{N})$ space, for every $1\leq p\leq \infty$. Also, by $\mathbb{B}_0^\mathbb{N}$ we represent the space of $\mathbb{B}$-valued sequences $f=(f(n))_{n\in \mathbb{N}}$ such that $f(n)=0$, $n>k$, for some $k\in \mathbb{N}$.

\begin{teo}\label{Th2.1}
Let $\mathbb{B}_1$ and $\mathbb{B}_2$ be Banach spaces. Suppose that $T$ is a linear and bounded operator from $\ell ^r_{\mathbb{B}_1}(\mathbb{N})$ into $\ell ^r_{\mathbb{B}_2}(\mathbb{N})$, for some $1<r<\infty$, and such that there exists a local $\mathcal{L}(\mathbb{B}_1,\mathbb{B}_2)$-standard kernel $K$ such that, for every sequence $f\in (\mathbb{B}_1)_0^\mathbb{N}$,
$$
T(f)(n)=\sum_{m\in \N}K(n,m)f(m),
$$
for every $n\in \mathbb{N}$ such that $f(n)=0$. Then,

(i) for every $1<p<\infty$ and $w\in A_p(\mathbb{N})$ the operator $T$ can be extended from $\ell ^r_{\mathbb{B}_1}(\mathbb{N})\bigcap \ell ^p_{\mathbb{B}_1}(\mathbb{N},w)$ to $\ell ^p_{\mathbb{B}_1}(\mathbb{N},w)$ as a bounded operator from $\ell ^p_{\mathbb{B}_1}(\mathbb{N},w)$ into $\ell ^p_{\mathbb{B}_2}(\mathbb{N},w)$.

(ii) for every $w\in A_1(\mathbb{N})$, the operator $T$ can be extended from $\ell ^r_{\mathbb{B}_1}(\mathbb{N})\bigcap \ell ^1_{\mathbb{B}_1}(\mathbb{N},w)$ to $\ell ^1_{\mathbb{B}_1}(\mathbb{N},w)$ as a bounded operator from $\ell ^1_{\mathbb{B}_1}(\mathbb{N},w)$ into $\ell ^{1,\infty }_{\mathbb{B}_2}(\mathbb{N},w)$.
\end{teo}

\begin{proof}
For every $n\in \mathbb{N}$, we define $W_n:=\{m\in \mathbb{N}:n/2\leq m\leq 3n/2\}$ and the operators
$$
T_{\rm glob}(f)(n):=T(\chi _{\mathbb{N}\setminus W_n}f)(n),\quad n\in \N,
$$
and
$$
T_{\rm loc}(f):=T(f)-T_{\rm glob}(f),
$$
for every $f=(f(n))_{n\in \mathbb{N}}\in (\mathbb{B}_1)_0^\mathbb{N}$.

Let $f\in (\mathbb{B}_1)_0^\mathbb{N}$. Since $\chi _{\mathbb{N}\setminus W_n}(n)=0$, $n\in \mathbb{N}$, we can write
$$
T_{\rm glob}(f)(n)=\sum_{m\in \mathbb{N}\setminus W_n}K(n,m)f(m),\quad n\in \mathbb{N}.
$$
According to condition $(a)$ for $K$ we get
\begin{align*}
\|T_{\rm glob}(f)(n)\|_{\mathbb{B}_2}&\leq C\sum _{m\in \mathbb{N}\setminus W_n}\frac{\|f(m)\|_{\mathbb{B}_1}}{|n-m|}\leq C\left(\frac{1}{n}\sum_{m\in \mathbb{N}, m<n/2}
\|f(m)\|_{\mathbb{B}_1}+\sum _{m\in \mathbb{N}, m>3n/2}\frac{\|f(m)\|_{\mathbb{B}_1}}{m}\right)\\
&\leq C\Big(H_0(\|f\|_{\mathbb{B}_1})(n)+H_\infty (\|f\|_{\mathbb{B}_1})(n)\Big),\quad n\in \mathbb{N},
\end{align*}
where $\|f\|_{\mathbb{B}_1}:=(\|f(m)\|_{\mathbb{B}_1})_{m\in \mathbb{N}}$, and $H_0,$ and $H_\infty$ are the discrete Hardy operators given by
$$
H_0(g)(n):=\frac{1}{n}\sum_{m=0}^ng(m),\quad n\in \mathbb{N}\setminus\{0\},$$
and
$$
H_\infty (g)(n):=\sum_{m=n}^\infty \frac{g(m)}{m},\quad n\in \mathbb{N}\setminus\{0\},
$$
with $g=(g(m))_{m\in \mathbb{N}}\in \mathbb{C}^\mathbb{N}$.
Wellknown $\ell ^p$-boundedness properties for discrete Hardy operators (\cite{AM}) allow us to conclude that if $1<p<\infty $ and $w\in A_p(\mathbb{N})$, $T_{\rm glob}$ can be extended to  $\ell _{\mathbb{B}_1}^p(\mathbb{N},w)$ as a bounded operator from $\ell _{\mathbb{B}_1}^p(\mathbb{N},w)$ into $\ell _{\mathbb{B}_2}^p(\mathbb{N},w)$ and if $w\in A_1(\mathbb{N})$, $T_{\rm glob}$ can be extended to $\ell _{\mathbb{B}_1}^1(\mathbb{N},w)$ as a bounded operator from $\ell _{\mathbb{B}_1}^1(\mathbb{N},w)$ into $\ell _{\mathbb{B}_2}^{1,\infty }(\mathbb{N},w)$.

We now study the operator $T_{\rm loc}$. If $f=(f(n))_{n\in \mathbb{N}}\in (\mathbb{B}_1)_0^\mathbb{N}$, we can write
$$
T_{\rm loc}(f)(n)=\sum_{m\in W_n}K(n,m)f(m),\quad n\in \mathbb{N}\mbox{ such that }f(n)=0.
$$
We define the function $\widetilde{K}$ by
$$
\widetilde{K}(n,m):=\chi _{W_n}(m)K(n,m),\quad n,m\in \mathbb{N},\;n\not =m.
$$
We are going to see that $\widetilde{K}$ satisfies certain H\"ormander type conditions that can be seen as discrete analogues of \cite[(4.4) and (4.5)]{NS}. 

If $a,b\in \mathbb{N}$, $a<b$, and $I$ is the interval in $\mathbb{N}$ given by $I:=[a,b]\cap \mathbb{N}$, we denote by $2I$ the set 
$$
2I:=\Big[a-\frac{b-a}{2}, b+\frac{b-a}{2}\Big]\cap \mathbb{N}.
$$ 
By $\mathcal{M}$ we represent the noncentered Hardy-Littlewood maximal function on $\mathbb{N}$ defined as follows: for every $g=(g(m))_{m\in \mathbb{N}}\in \mathbb{C}^\mathbb{N}$,
$$
\mathcal{M}(g)(n):=\sup_{\substack{I{\rm interval}\\n\in I}}\frac{1}{\# (I)}\sum _{m\in I} g(m),\quad n\in \mathbb{N},
$$
where $\#(I)$ denotes the cardinal of $I$. 

We assert that the function $\widetilde{K}$ satisfies the following H\"ormander conditions: there exists $C>0$ such that, for every interval $I$ in $\mathbb{N}$ and $f=(f(m))_{m\in \mathbb{N}}\in\mathbb{B}_1^\mathbb{N}$,
\begin{equation}\label{2.1}
\sum_{m\in \mathbb{N}\setminus 2I}\|\widetilde{K}(n,m)-\widetilde{K}(\ell,m)\|_{\mathcal{L}(\mathbb{B}_1,\mathbb{B}_2)}{\|f(m)\|_{\mathbb{B}_1}}\leq C\mathcal{M}(\|f\|_{\mathbb{B}_1})(n),\quad n,\ell\in I,
\end{equation}
and
\begin{equation}\label{2.2}
\sum_{m\in \mathbb{N}\setminus 2I}\|\widetilde{K}(m,n)-\widetilde{K}(m,\ell)\|_{\mathcal{L}(\mathbb{B}_1,\mathbb{B}_2)}{\|f(m)\|_{\mathbb{B}_1}}\leq C\mathcal{M}(\|f\|_{\mathbb{B}_1})(n),\quad n,\ell\in I.
\end{equation}

Since (\ref{2.1}) and (\ref{2.2}) can be proved similarly, we only show (\ref{2.1}). Let $a,b\in \mathbb{N}$, $a<b$, $I:=[a,b]\cap \mathbb{N}$ and $f=(f(m))_{m\in \mathbb{N}}\in \mathbb{B}_1^\mathbb{N}$. Suppose that $n,\ell\in I$, and $n<\ell$. 

First we observe that when $m\in \mathbb{N}\setminus 2I$, then, 
\begin{equation}\label{nml}
\frac{|m-n|}{3}\leq |m-\ell|\leq 3|m-n|.
\end{equation}
To see this estimate, let us write $2I=[A,B]\cap \mathbb{N}$ and $L=(b-a)/2$, and take $m\in \mathbb{N}\setminus 2I$. In the case that $m>B$, we get
$$
m-n>m-\ell >m-b=m-B+L=m-B+\frac{B-a}{3}\geq \frac{m-a}{3}\geq\frac{m-n}{3}.
$$
Similarly, when $m<A$, we obtain that
$$
\ell-m>n-m>L+A-m=\frac{b-A}{3}+A-m\geq \frac{b-m}{3}\geq \frac{\ell -m}{3},
$$ 
and \eqref{nml} is established.

We can decompose the left hand side in (\ref{2.1}) as follows:
\begin{align}\label{D1}
\sum_{m\in \mathbb{N}\setminus 2I}\|\widetilde{K}(n,m)-\widetilde{K}(\ell,m)\|_{\mathcal{L}(\mathbb{B}_1,\mathbb{B}_2)}{\|f(m)\|_{\mathbb{B}_1}}\nonumber&
=\sum_{\substack{m\in \mathbb{N}\setminus 2I\\m\in W_n\cap W_\ell}}\|K(n,m)-K(\ell,m)\|_{\mathcal{L}(\mathbb{B}_1,\mathbb{B}_2)}{\|f(m)\|_{\mathbb{B}_1}}\\
&\hspace{-4cm}+\sum_{\substack{m\in \mathbb{N}\setminus 2I\\m\in W_n\setminus W_\ell}}\|K(n,m)\|_{\mathcal{L}(\mathbb{B}_1,\mathbb{B}_2)}{\|f(m)\|_{\mathbb{B}_1}}
+\sum_{\substack{m\in \mathbb{N}\setminus 2I\\m\in W_\ell\setminus W_n}}\|K(\ell,m)\|_{\mathcal{L}(\mathbb{B}_1,\mathbb{B}_2)}{\|f(m)\|_{\mathbb{B}_1}}\nonumber\\
&\hspace{-4.3cm}:=S_1(n,\ell)+S_2(n,\ell)+S_3(n,\ell).
\end{align}
By taking into account the size condition of $K$ and (\ref{nml}) we have that
$$
S_2(n,\ell)+S_3(n,\ell)\leq C\Big(\sum_{\substack{m\in \mathbb{N}\setminus 2I\\m\in W_n\setminus W_\ell}}+
\sum_{\substack{m\in \mathbb{N}\setminus 2I\\m\in W_\ell\setminus W_n}}\Big)\frac{\|f(m)\|_{\mathbb{B}_1}}{|n-m|}.
$$

Assume now that $3n<\ell$. In this case we have that $W_n\cap W_\ell=\emptyset$ and then $S_1(n,\ell )=0$. Also, since for $m\in \mathbb{N}\setminus 2I$, $|m-n|>(b-a)/2\geq (\ell -n)/2\geq \ell /3>n$, we get
\begin{align*}
S_2(n,\ell)+S_3(n,\ell)&\leq C\left(\frac{1}{n}\sum_{m\in W_n}\|f(m)\|_{\mathbb{B}_1}+\frac{1}{\ell }\sum_{m\in W_\ell}\|f(m)\|_{\mathbb{B}_1}\right)\\
&\leq C\left(\frac{1}{n}\sum_{m\in W_n}\|f(m)\|_{\mathbb{B}_1}+\frac{1}{\ell }\sum_{m\in J}\|f(m)\|_{\mathbb{B}_1}\right),
\end{align*}
where $J:=[n,3\ell /2]\cap \mathbb{N}$. By considering that $3\ell /2\geq 3\ell /2-n>3\ell /2-\ell /3=7\ell /6$,  we conclude that
\begin{equation}\label{D2}
S_2(n,\ell )+S_3(n,\ell)\leq C\mathcal{M}(\|f\|_{\mathbb{B}_1})(n). 
\end{equation}

Next we deal with the condition $\ell \leq 3n$. Now we have that
\begin{align*}
&W_n\cap W_\ell=\Big[\frac{\ell }{2},\frac{3n}{2}\Big]\cap \mathbb{N}=\left(\Big[\frac{\ell }{2},\frac{2\ell}{3}\Big]\cap \mathbb{N}\right)\bigcup \left(\Big[\frac{2\ell }{3},\frac{3n}{2}\Big]\cap \mathbb{N}\right)=:J_1\cup J_2,\\
&W_n\setminus W_\ell =\Big[\frac{n}{2},\frac{\ell}{2}\Big)\cap \mathbb{N},\\
&W_\ell\setminus W_n =\Big(\frac{3n}{2},\frac{3\ell}{2}\Big]\cap \mathbb{N},
\end{align*}
and we can write
\begin{align}\label{D3}
S_1(n,\ell)+S_2(n,\ell)+S_3(n,\ell)&\leq \sum_{m\in J_2\cap \mathbb{N}\setminus 2I}\|K(n,m)-K(\ell,m)\|_{\mathcal{L}(\mathbb{B}_1,\mathbb{B}_2)}{\|f(m)\|_{\mathbb{B}_1}}\nonumber\\
&\hspace{-3cm}+\Big(\sum_{\substack{m\in \mathbb{N}\setminus 2I\\n/2\leq m\leq 2\ell /3}}+\sum_{\substack{m\in \mathbb{N}\setminus 2I\\3n/2\leq m\leq 3\ell /2}}\Big)\frac{\|f(m)\|_{\mathbb{B}_1}}{|n-m|}=:T_1(n,\ell)+T_2(n,\ell ).
\end{align}

In order to estimate $T_1(n,\ell)$, we decompose it into two terms as follows:
\begin{align*}
T_1(n,\ell )&=\Big(\sum_{\substack{m\in J_2\cap \mathbb{N}\setminus 2I\\|n-m|\leq 2|n-\ell|}}+\sum_{\substack{m\in J_2\cap \mathbb{N}\setminus 2I\\|n-m|>2|n-\ell|}}\Big)\|K(n,m)-K(\ell,m)\|_{\mathcal{L}(\mathbb{B}_1,\mathbb{B}_2)}{\|f(m)\|_{\mathbb{B}_1}}\\
&=:T_{1,1}(n,\ell)+T_{1,2}(n,\ell).
\end{align*}

Then, according to the size condition $(a)$ and (\ref{nml}) we get
$$
T_{1,1}(n,\ell)\leq C\sum_{\substack{m\in\mathbb{N}\setminus 2I\\|n-m|\leq 2|n-\ell|}}\frac{\|f(m)\|_{\mathbb{B}_1}}{|n-m|}\leq C\sum_{m\in \mathbb{N}\setminus 2I}\frac{|n-\ell |}{|n-m|^2}\|f(m)\|_{\mathbb{B}_1},
$$
and by taking into account the regularity property $(b1)$ for $K$  we deduce the same estimate for $T_{1,2}(n,\ell )$. Hence, 
\begin{align}\label{D4}
T_1(n,\ell )&\leq C\sum_{m\in \mathbb{N}\setminus 2I}\frac{|n-\ell |}{|n-m|^2}\|f(m)\|_{\mathbb{B}_1}\leq C\sum_{k=1}^\infty \sum_{m\in 2^{k+1}I\setminus 2^kI}\frac{|n-\ell|}{|n-m|^2}\|f(m)\|_{\mathbb{B}_1}\nonumber\\
&\leq C\sum_{k=1}^\infty\frac{\#(I)}{2^{2k}(\#(I))^2}\sum_{m\in 2^{k+1}I}\|f(m)\|_{\mathbb{B}_1}\leq C\mathcal{M}(\|f\|_{\mathbb{B}_1})(n).
\end{align}

We analyze now $T_2(n,\ell)$. It is clear that $m-n>n/2$, when $m>3n/2$. We can also establish that if $m\in \mathbb{N}\setminus 2I$ and $n/2\leq m\leq 2\ell/3$, then $|m-n|>n/8$. Let $m\in \mathbb{N}\setminus 2I$, such that $n/2\leq m\leq 2\ell/3$. 
If $\ell <5n/4$, then, $m\leq 5\ell /6$, and $n-m>n/6$. When $\ell \geq 5n/4$, we get that $|n-m|>(b-a)/2>(5n/4-n)/2=n/8$ and the result is establihed.

Then, since $\ell \leq 3n$, we can write
\begin{equation}\label{D5}
T_2(n,\ell )\leq \frac{C}{n}\Big(\sum_{m\in [n/2,2n]\cap \mathbb{N}}+\sum_{m\in [n,9n/2]\cap \mathbb{N}}\Big)\|f(m)\|_{\mathbb{B}_1}\leq C\mathcal{M}(\|f\|_{\mathbb{B}_1}).
\end{equation}

By combining (\ref{D1})-(\ref{D5}) we establish (\ref{2.1}).

Since $T$ and $T_{\rm glob}$ are bounded from $\ell _{\mathbb{B}_1}^r(\mathbb{N})$ into $\ell _{\mathbb{B}_2}^r(\mathbb{N})$, $T_{\rm loc}$ is also bounded from $\ell _{\mathbb{B}_1}^r(\mathbb{N})$ into $\ell _{\mathbb{B}_2}^r(\mathbb{N})$.
Hence, according to \cite[Theorem 1.1]{GLY}, for every $1<p<\infty$, $T_{\rm loc}$ can be extended from $\ell _{\mathbb{B}_1}^p(\mathbb{N})\cap \ell _{\mathbb{B}_1}^r(\mathbb{N})$ to $\ell _{\mathbb{B}_1}^p(\mathbb{N})$ as a bounded operator from $\ell _{\mathbb{B}_1}^p(\mathbb{N})$ into $\ell _{\mathbb{B}_2}^p(\mathbb{N})$ and $T_{\rm loc}$ can be extended from $\ell _{\mathbb{B}_1}^1(\mathbb{N})\cap \ell _{\mathbb{B}_1}^r(\mathbb{N})$ to $\ell _{\mathbb{B}_1}^1(\mathbb{N})$ as a bounded operator from $\ell _{\mathbb{B}_1}^1(\mathbb{N})$ into $\ell _{\mathbb{B}_2}^{1,\infty }(\mathbb{N})$. The same properties are satisfied by $T$ because $T_{\rm glob}$ also verifies them.

Finally, by adapting the arguments in \cite[Lemmas 5.15, 7.9 and 7.10 and Theorems 7.11 and 7.12]{Duo} to vector-valued homogeneous settings, we conclude that $T_{\rm loc}$, and then $T$, can be extended from $\ell _{\mathbb{B}_1}^p(\mathbb{N},w)\cap \ell _{\mathbb{B}_1}^r(\mathbb{N})$ to $\ell _{\mathbb{B}_1}^p(\mathbb{N},w)$ as a bounded operator from $\ell _{\mathbb{B}_1}^p(\mathbb{N},w)$ into $\ell _{\mathbb{B}_2}^p(\mathbb{N},w)$, for every $1<p<\infty $ and $w\in A_p(\mathbb{N})$, and from $\ell _{\mathbb{B}_1}^1(\mathbb{N},w)\cap \ell _{\mathbb{B}_1}^r(\mathbb{N})$ to $\ell _{\mathbb{B}_1}^1(\mathbb{N},w)$ as a bounded operator from $\ell _{\mathbb{B}_1}^1(\mathbb{N},w)$ into $\ell _{\mathbb{B}_2}^{1,\infty }(\mathbb{N},w)$, for every $w\in A_1(\mathbb{N})$.

Thus the proof of this theorem is completed.
\end{proof}

\section{Proof of Theorem \ref{Th1} for maximal operators}\label{S3}

In this section we prove the boundedness properties for the maximal operators associated with heat and Poisson semigroups stated in Theorem \ref{Th1}.
From (\ref{1.5}) we deduce that $P_*^\lambda (f)\leq W_*^\lambda (f)$. Hence, the properties in Theorem \ref{Th1} for $P_*^\lambda $ can be infered from the corresponding properties for $W_*^\lambda$. Hence, it is sufficient to prove Theorem \ref{Th1} for $W_*^\lambda$.

Let $1\leq p\leq \infty$. Since $\Delta _\lambda $ is a bounded operator from $\ell ^p(\mathbb{N})$ into itself, $\Delta _\lambda$ generates the $C_0$-semigroup $\{W_t^\lambda :=e^{t\Delta _\lambda}\}_{t>0}$ of operators in $\ell ^p(\mathbb{N})$. We have that
\begin{equation}\label{3.1}
\partial _tW_t^\lambda (f)=\Delta _\lambda W_t^\lambda (f), \quad f\in \ell ^p(\mathbb{N})\mbox{ and }t>0.
\end{equation}
Moreover, $\{W_t^\lambda \}_{t>0}$ is not Markovian. Indeed, let $g=(g(n)=1)_{n\in \mathbb{N}}$. If $W_t^\lambda (g)=g$, $t>0$, then, by (\ref{3.1}) we have that $\Delta _\lambda g=0$, and that is clearly impossible.

According to  (\ref{1.4}), (\ref{1.2}) and (\ref{5prime}) we can write, for every $f=(f(n))_{n\in \mathbb{N}}\in \ell ^2(\mathbb{N})$, 
\begin{align}\label{eq:9}
\sum_{n\in \mathbb{N}}|W_t^\lambda (f)(n)|^2& =\int_{-1}^1|\mathcal{F}_\lambda (W_t^\lambda f)(x)|^2dx=\int_{-1}^1|e^{-2t(1-x)}\mathcal{F}_\lambda (f)(x)|^2dx\nonumber\\
& \leq \int_{-1}^1|\mathcal{F}_\lambda (f)(x)|^2dx=\sum_{n\in \mathbb{N}}|f(n)|^2,\quad t>0.
\end{align}
Hence, for every $t>0$, $W_t^\lambda$ is contractive in $\ell ^2(\mathbb{N})$. 
We now prove that the maximal operator $W_*^\lambda$ is bounded in $\ell ^2(\mathbb{N})$. By proceeding as in \cite[p. 75]{StLP} we obtain
$$
W_*^\lambda (f)\leq M^\lambda (f)+g_{W^\lambda }^1(f),\quad f\in \ell ^2(\mathbb{N}),
$$
where
$$
g_{W^\lambda }^1(f)(n):=\left(\int_0^\infty |t\partial _t(W_t^\lambda f)(n)|^2\frac{dt}{t}\right)^{1/2},\quad n\in \mathbb{N},
$$
and
$$
M^\lambda (f)(n):=\sup_{s>0}\frac{1}{s}\left|\int_0^sW_t^\lambda (f)(n)dt\right|,\quad n\in \mathbb{N}.
$$

The $g$-function $g_{W^\lambda }^1$ is bounded from $\ell ^2(\mathbb{N})$ into itself. This can be seen by using spectral arguments (see \cite[p. 74]{StLP}). Since $W_t^\lambda $ is a contraction in $\ell ^2(\mathbb{N})$, for every $t>0$, the Hopf-Dunford-Schwartz ergodic theorem (see \cite{Stein}) allows us to show that the maximal operator $M^\lambda$ is bounded from $\ell ^2(\mathbb{N})$ into itself. Hence, $W_*^\lambda$ is bounded from $\ell ^2(\mathbb{N})$ into itself.

In order to show $\ell^p$-properties of $W_*^\lambda $ for $1\leq p<\infty $, $p\not=2$, we use Theorem \ref{Th2.1}.

We consider the operator
$$
T:\ell ^2( \mathbb{N})\longrightarrow \ell _\mathbb{B}^2(\mathbb{N})
$$
$$
\hspace{3cm}f\rightarrow T(f)(n;t):=W_t^\lambda (f)(n),
$$
where $\mathbb{B}=L^\infty (0,\infty )$. The operator $T$ is bounded from $\ell ^2(\mathbb{N})$ into $\ell _\mathbb{B}^2(\mathbb{N})$, because $W_*^\lambda $ is bounded from $\ell ^2(\mathbb{N})$ into itself.

According to (\ref{5prime}), for every $f\in \mathbb{C}_0^\mathbb{N}$  we can write
$$
T(f)(n;t)=(h_t^\lambda \#_\lambda f)(n)=\sum_{m=0}^Nf(m)\;_\lambda \tau _n(h_t^\lambda )(m),\quad n\in \mathbb{N}\mbox{ and }t>0,
$$
where $N\in \mathbb{N}$ is such that $f(n)=0$, $n\geq N$.

To simplify we define $K_t^\lambda (n,m):=\,_\lambda \tau _n(h_t^\lambda )(m)$, $n,m\in \mathbb{N}$ and $t>0$. By (\ref{1.3}) we have that
$$
K_t^\lambda (n,m)=\sum_{k=|n-m|}^{n+m}c_\lambda (n,m,k)h_t^\lambda (k),\quad m,n\in \mathbb{N},
$$
where $c_\lambda (n,m,k)$, $n,m,k\in \mathbb{N}$, and $h_t^\lambda $, $t>0$, are given by (\ref{clambda}) and (\ref{hlambda}), respectively.

We are going to see that there exists $C>0$ such that, for every $n,m,\ell \in \mathbb{N}$, $n\not=m$,
\begin{equation}\label{3.2}
\|K_t^\lambda (n,m)\|_\mathbb{B}\leq \frac{C}{|n-m|};
\end{equation}
\begin{equation}\label{3.3}
\|K_t^\lambda (n,m)-K_t^\lambda (\ell , m)\|_\mathbb{B}\leq C\frac{|n-\ell |}{|n-m|^2},\quad |n-m|>2|n-\ell |,\;\frac{m}{2}<n,\ell <\frac{3m}{2};
\end{equation}
\begin{equation}\label{3.4}
\|K_t^\lambda (m,n)-K_t^\lambda (m, \ell)\|_\mathbb{B}\leq C\frac{|n-\ell |}{|n-m|^2},\quad |n-m|>2|n-\ell |,\;\frac{m}{2}<n,\ell <\frac{3m}{2}.
\end{equation}
When properties (\ref{3.2}), (\ref{3.3}) and (\ref{3.4}) are established then, from Theorem \ref{Th2.1} we deduce the $\ell^p$-boundedness properties stated in Theorem \ref{Th1} for $W_*^\lambda$. Observe that, since $K_t^\lambda (n,m)=K_t^\lambda (m,n)$, $n,m\in \mathbb{N}$, $t>0$, we only have to establish (\ref{3.2}) and (\ref{3.3}).

\begin{proof}[Proof of (\ref{3.2})] 
We will use the following integral representations for the  modified Bessel  function $I_\nu$ (\cite[(5.10.22)]{Leb}
\begin{equation}\label{3.5}
I_\nu (z)=\frac{z^\nu}{\sqrt{\pi }2^\nu \Gamma (\nu +1/2)}\int_{-1}^1e^{-zs}(1-s^2)^{\nu -1/2}ds,\quad z>0\mbox{ and }\nu >-\frac{1}{2},
\end{equation}
and

\begin{equation}\label{3.11}
I_\nu (z)=-\frac{z^{\nu -1}}{\sqrt{\pi }2^{\nu -1}\Gamma (\nu -1/2)}\int_{-1}^1e^{-zs}s(1-s^2)^{\nu -3/2}ds,\quad z>0\mbox{ and }\nu >\frac{1}{2}.
\end{equation}
Note that (\ref{3.11}) is obtained by partial integration in (\ref{3.5}) (see \cite[(34)]{CGRTV}).

We have that
\begin{equation}\label{3.60}
0\leq e^{-2t}t^{-\lambda}I_{\lambda +k}(2t)\leq \frac{C}{(k+1)^{2\lambda +1}},\quad k\in \N \mbox{ and }t>0,
\end{equation}
where $C>0$ does not depend on $k\in \mathbb{N}$ nor on $t>0$.

To see this estimation we proceed as in \cite[Proposition 3]{CGRTV}. For every $k\in \mathbb{N}$ and $t>0$, we can write
\begin{align*}
0\leq e^{-2t}t^{- \lambda }I_{\lambda +k}(2t)&=\frac{t^k}{\sqrt{\pi }\Gamma (\lambda +k+1/2)}\int_{-1}^1e^{-2t(s+1)}(1-s^2)^{\lambda +k-1/2}ds\\
&=\frac{2t^{k-1}}{\sqrt{\pi} \Gamma (\lambda +k+1/2)}\int_0^te^{-4w}\Big(\frac{2w}{t}\Big)^{\lambda +k-1/2}\Big[2\Big(1-\frac{w}{t}\Big)\Big]^{\lambda +k-1/2}dw\\
&\leq \frac{2^{2(\lambda +k)}}{\sqrt{\pi }\Gamma (\lambda +k+1/2)}\int_0^te^{-4w}\Big(\frac{w}{t}\Big)^{\lambda +1/2}\Big(1-\frac{w}{t}\Big)^{\lambda +k-1/2}w^{k-1}dw\\
&\leq \frac{2^{2(\lambda +k)}}{\sqrt{\pi }\Gamma (\lambda +k+1/2)}\Big(\frac{\lambda +1/2}{2\lambda +k}\Big)^{\lambda +1/2}\int_0^te^{-4w}w^{k-1}dw\\
&\leq\frac{2^{2\lambda}(\lambda +1/2)^{\lambda +1/2}}{\sqrt{\pi}}\frac{\Gamma(k)}{\Gamma (\lambda +k+1/2)(2\lambda +k)^{\lambda +1/2}}\leq \frac{C}{(k+1)^{2\lambda +1}}.\end{align*}
We have taken into account that
\begin{equation}\label{3.6}
(1-r)^\eta r^\gamma \leq \Big(\frac{\gamma}{\gamma +\eta}\Big)^\gamma,\quad 0<r<1,\;\eta >0\mbox{ and }\gamma >0.
\end{equation}
By using (\ref{3.8}) and (\ref{3.60}) we obtain
$$
0\leq K_t^\lambda (n,m)\leq C\sum_{k=0}^\infty \frac{c_\lambda (n,m,k)}{(k+1)^{\lambda +1}},\quad n,m\in \mathbb{N},n\not=m\mbox{ and }t>0.
$$
Since $c_\lambda (n,m,k)=c_\lambda (m,n,k)$, $n,m,k\in \mathbb{N}$, \cite[Lemma 3b]{AH1} says that, if $0<\alpha <1/2$, then
\begin{equation}\label{3.9}
\sum_{k=0}^\infty \frac{c_\lambda (n,m,k)}{(k+1)^{\lambda +2\alpha +1}}\leq \frac{C}{|n-m|^{2\alpha +1}},\quad n,m\in \mathbb{N},n\not=m.
\end{equation}
The same proof of \cite[Lemma 3b]{AH1} allows us to see that (\ref{3.9}) also holds for $\alpha =0$ and $\alpha =1/2$.

Hence,
$$
\|K_t^\lambda (n,m)\|_\mathbb{B}\leq \frac{C}{|n-m|},\quad n,m\in \mathbb{N},\;n\not=m.
$$
\end{proof}

\begin{proof}[ Proof of (\ref{3.3})] In a first step we show the following estimations:
\begin{align*}
(A1) &\hspace{1cm}\|K_t^\lambda (n+1,m)-K_t^\lambda (n,m)\|_\mathbb{B}\leq \frac{C}{|n-m|^2},\quad n,m\in \mathbb{N},n>m;\\
(A2) &\hspace{1cm} \|K_t^\lambda (n,m)-K_t^\lambda (n-1,m)\|_\mathbb{B}\leq \frac{C}{|n-m|^2},\quad n,m\in \mathbb{N},\,1\leq n<m\leq 2n.
\end{align*}

Let $n,m\in \mathbb{N}$ and $n>m$. We can write
\begin{align}\label{A1d}
K_t^\lambda (n+1,m)-K_t^\lambda (n,m)&=\sum_{k=n-m+1}^{n+m+1}c_\lambda (n+1,m,k)h_t^\lambda (k)-\sum_{k=n-m}^{n+m}c_\lambda (n,m,k)h_t^\lambda (k)\nonumber\\
&=\sum_{k=n-m}^{n+m}[c_\lambda (n+1,m,k+1)h_t^\lambda (k+1)-c_\lambda (n,m,k)h_t^\lambda (k)]\nonumber\\
&=\sum_{k=n-m}^{n+m}c_\lambda (n+1,m,k+1)\Big[h_t^\lambda (k+1)-\frac{\sqrt{w_\lambda (k+1)}}{\sqrt{w_\lambda (k)}}h_t^\lambda (k)\Big]\nonumber\\
&\quad +\sum_{k=n-m}^{n+m}\Big[\frac{\sqrt{w_\lambda (k+1)}}{\sqrt{w_\lambda (k)}}c_\lambda (n+1,m,k+1)-c_\lambda (n,m,k)\Big]h_t^\lambda (k)\nonumber\\
&=:I_1(n,m,t)+I_2(n,m,t),\quad t>0.
\end{align}

By combining (\ref{3.5}) and (\ref{3.11}) (see \cite[p. 9]{CGRTV}) we get, for every $t>0$ and $k\in \mathbb{N}$,
\begin{align*}
e^{-2t}t^{-\lambda }|I_{\lambda +k+1}(2t)-I_{\lambda +k}(2t)|&=\frac{t^k}{\sqrt{\pi }\Gamma (\lambda +k+1/2)}\int_{-1}^1e^{-2t(1+s)}(1+s)(1-s^2)^{\lambda +k-1/2}ds\\
&=\frac{2^{2(\lambda +k)+1}t^{k-1}}{\sqrt{\pi }\Gamma (\lambda +k+1/2)}\int_0^te^{-4w}\Big(\frac{w}{t}\Big)^{\lambda +k+1/2}\Big(1-\frac{w}{t}\Big)^{\lambda +k-1/2}dw\\
&=\frac{2^{2(\lambda +k)+1}}{\sqrt{\pi }\Gamma (\lambda +k+1/2)}\int_0^te^{-4w}w^{k-1}\Big(\frac{w}{t}\Big)^{\lambda +3/2}\Big(1-\frac{w}{t}\Big)^{\lambda +k-1/2}dw.
\end{align*}

\noindent By using (\ref{3.6}) we conclude
\begin{align}\label{3.110}
e^{-2t}t^{-\lambda }|I_{\lambda +k+1}(2t)-I_{\lambda +k}(2t)|&\leq \frac{2^{2(\lambda +k)+1}}{\sqrt{\pi }\Gamma (\lambda +k+1/2)}\Big(\frac{\lambda +3/2}{2\lambda +k+1}\Big)^{\lambda +3/2}\int_0^\infty e^{-4w}w^{k-1}dw\nonumber\\
&\leq C\frac{\Gamma (k)}{(2\lambda +k+1)^{\lambda +3/2}\Gamma (\lambda +k+1/2)}\leq \frac{C}{(k+1)^{2\lambda +2}},\quad k\in \mathbb{N}\mbox{ and }t>0.
\end{align}
According to (\ref{3.8}) and \eqref{3.9} for $\alpha =1/2$, it follows that
\begin{equation}\label{3.12}
|I_1(n,m,t)|\leq C\sum_{k=|n-m|}^{n+m}\frac{c_\lambda (n,m,k)}{(k+1)^{\lambda +2}}\leq C\frac{C}{|n-m|^2},\quad t>0.
\end{equation}

We now analyze $I_2(n,m,t)$, $t>0$. Note that, if $k\in \mathbb{N}$, $c_\lambda (n,m,k)\not=0$ if and only if $c_\lambda(n+1,m,k+1)\not=0$. 
By using (\ref{3.8}) and (\ref{3.60}), we can write 
\begin{align*}
|I_2(n,m,t)|&\leq \sum_{\substack{|n-m|\leq k\leq n+m\\n+m+k\;{ \rm even }}}c_\lambda (n,m,k)\left|\frac{\sqrt{w_\lambda (k+1)}c_\lambda (n+1,m,k+1)}{\sqrt{w_\lambda (k)}c_\lambda (n,m,k)}-1\right|h_t^\lambda (k)\\
&\leq \sum_{\substack{|n-m|\leq k\leq n+m\\n+m+k\;{ \rm even }}}\frac{c_\lambda (n,m,k)}{(k+1)^{\lambda +1}}\left|\frac{\sqrt{w_\lambda (k+1)}c_\lambda (n+1,m,k+1)}{\sqrt{w_\lambda (k)}c_\lambda (n,m,k)}-1\right|
\end{align*}

Let $k\in \mathbb{N}$, $n-m\leq k\leq n+m$. Assume that $n+m+k=2\sigma$, with $\sigma \in \mathbb{N}$. Straightforward manipulations lead to 
$$
\frac{\sqrt{w_\lambda (k+1)}c_\lambda (n+1,m,k+1)}{\sqrt{w_\lambda (k)}c_\lambda (n,m,k)}
=\Big(\frac{n+1}{n+2\lambda}\Big)^{1/2}\Big(\frac{n+\lambda+1 }{n+\lambda }\Big)^{1/2}\frac{k+\lambda +1}{k+\lambda }\frac{\sigma -m+\lambda }{\sigma -m+1}\frac{\sigma+ 2\lambda }{\sigma +\lambda +1}.
$$
To simplify we define
$$
\alpha _1:=\sqrt{\frac{n+1}{ n+2\lambda}},\quad \alpha _2:=\sqrt{\frac{n+\lambda +1}{n+\lambda}},\quad \alpha _3:=\frac{k+\lambda +1}{k+\lambda },\quad
\alpha _4:=\frac{\sigma -m+\lambda}{\sigma -m+1},\quad \alpha _5:=\frac{\sigma+2\lambda}{\sigma +\lambda+1}.
$$
There exists $C>0$ independent on $n,m$ and $k$, such that $0\leq \alpha _j\leq C$, $j=1,...,5$. Also we get
$$
\Big(\prod_{j=1}^5\alpha _j\Big)-1=\sum_{j=1}^5(\alpha _j-1)\prod_{i=j+1}^5\alpha _i .
$$
We deduce that
\begin{equation}\label{calphaj}
\left|\frac{\sqrt{w_\lambda (k+1)}c_\lambda (n+1,m,k+1)}{\sqrt{w_\lambda (k)}c_\lambda (n,m,k)}-1\right|\leq C\sum_{j=1}^5|\alpha _j-1|.
\end{equation}

Since
$$ |\alpha_j-1|\leq \frac{C}{n+1}, \;\;j=1,2,\quad |\alpha_3-1|\leq \frac{C}{k+1}, \quad |\alpha_4-1|\leq \frac{C}{\sigma -m},\quad \mbox{ and }
|\alpha_5-1|\leq \frac{C}{\sigma},
$$
we have that
$$
|\alpha _j-1|\leq \frac{C}{n-m},\quad j=1,...,5.
$$
By using \eqref{3.8}, \eqref{3.60} and \eqref{3.9} (for $\alpha =0$) we obtain
\begin{equation}\label{I2}
|I_2(n,m,t)|\leq \frac{C}{n-m}\sum_{k=n-m}^{n+m}\frac{c_\lambda(n,m,k)}{(k+1)^{\lambda +1}}\leq \frac{C}{|n-m|^2},\quad t>0,
\end{equation}
According to (\ref{A1d}), (\ref{3.12}) and (\ref{I2}) we deduce property (A1).

Next we justify (A2). Suppose that $1\leq n<m\leq 2n$. We have that 
$$
K_t^\lambda (n,m)-K_t^\lambda (n-1,m)=
\sum_{k=m-n}^{m+n}c_\lambda (n,m,k)h_t ^\lambda (k)-\sum_{k=m-n+1}^{n+m-1}c_\lambda (n-1,m,k)h_t^\lambda (k),\quad t>0.
$$
In order to make the estimations easier, we assemble the terms in the sums in a suitable form. We observe that the first sum has two terms more than the sencond one, so we proceed as follows:
\begin{align}\label{A2d}
K_t^\lambda (n,m)-K_t^\lambda (n-1,m)&=\left(\sum_{k=m-n}^{m-1}+\sum_{k=m+2}^{n+m}\right)c_\lambda (n,m,k)h_t ^\lambda (k)\nonumber\\
&\hspace{-3cm}+c_\lambda (n,m,m)h_t^ \lambda (m)+c_\lambda (n,m,m+1)h_t^\lambda (m+1)-\sum_{k=m-n+1}^{n+m-1}c_\lambda (n-1,m,k)h_t^\lambda (k)\nonumber\\
&\hspace{-3cm}=\sum_{k=m-n+1}^{m}c_\lambda (n,m,k-1)h_t ^\lambda (k-1)+\sum_{k=m+1}^{n+m-1}c_\lambda (n,m,k+1)h_t^\lambda (k+1)\nonumber\\
&\hspace{-3cm}\quad +c_\lambda (n,m,m)h_t^ \lambda (m)+c_\lambda (n,m,m+1)h_t^\lambda (m+1)-\sum_{k=m-n+1}^{n+m-1}c_\lambda (n-1,m,k)h_t^\lambda (k)\nonumber\\
&\hspace{-3cm}=\sum_{k=m-n+1}^{m}[c_\lambda (n,m,k-1)h_t^\lambda (k-1)-c_\lambda (n-1,m,k)h_t^\lambda (k)]\nonumber\\
&\hspace{-3cm}\quad +\sum_{k=m+1}^{m+n-1}[c_\lambda (n,m,k+1)h_t^\lambda (k+1)-c_\lambda (n-1,m,k)h_t^\lambda (k)]\nonumber\\
&\hspace{-3cm}\quad +c_\lambda (n,m,m)h_t^ \lambda (m)+c_\lambda (n,m,m+1)h_t^\lambda (m+1)\nonumber\\
&\hspace{-3cm}=:J_1(n,m,t)+J_2(n,m,t)+J_3(n,m,t),\quad t>0.
\end{align}

We first  study $J_3(n,m,t)$, $t>0$. Note that $c_\lambda (n,m,m)\not=0$ if and only if $c_\lambda (n,m,m+1)=0$. 

Suppose that $n+2m=2\sigma$, with $\sigma \in \mathbb{N}$. By using \eqref{acotclambda} and \eqref{3.60} and we obtain
\begin{align*}
c_\lambda (n,m,m)h_t^\lambda (m)&\leq C\left(\frac{\sigma (\sigma -n+1)(\sigma -m+1)^2}{(n+1)(m+1)^2}\right)^{\lambda -1}\frac{1}{(m+1)^{\lambda +1}}\nonumber\\
&\leq C\frac{[(n+2m)(2m-n)(n+1)]^{\lambda -1}}{(m+1)^{3\lambda -1}},\quad t>0.
\end{align*}
Since $n<m\leq 2n$, we have that $n<2m-n<3n$, and we get
$$
J_3(n,m,t)\leq \frac{C}{(m+1)^2},\quad t>0.
$$
In a similar way if $n+1+2m=2\sigma $, with $\sigma \in \mathbb{N}$ then
$$
c_\lambda (n,m,m+1)h_t^\lambda (m+1)\leq  \frac{C}{(m+1)^2} ,\quad t>0.
$$ 
Thus, we have obtained that
\begin{equation}\label{J3}
J_3(n,m,t)\leq \frac{C}{(m+1)^2}\leq \frac{C}{|n-m|^2} ,\quad t>0.
\end{equation}

To analyze $J_1$ we decompose it as follows:
\begin{align*}
J_1(n,m,t)&=\sum_{k=m-n+1}^{m}c_\lambda (n,m,k-1)\Big[h_t^\lambda (k-1)-\frac{\sqrt{w_\lambda (k-1)}}{\sqrt{w_\lambda (k)}}h_t^\lambda (k)\Big]\\
&\quad +\sum_{k=m-n+1}^{m}\Big[\frac{\sqrt{w_\lambda (k-1)}}{\sqrt{w_\lambda (k)}}c_\lambda (n,m,k-1)-c_\lambda (n-1,m,k)\Big]h_t^\lambda (k)\\
&=:J_{1,1}(n,m,t)+J_{1,2}(n,m,t),\quad t>0.
\end{align*}

By (\ref{3.8}) and (\ref{3.110}) and taking into account (\ref{3.9}) (for $\alpha =1/2$) we deduce
$$
|J_{1,1}(n,m,t)|\leq \sum_{k=m-n+1}^{m}c_\lambda (n,m,k-1)\frac{\sqrt{w_\lambda (k-1)}}{(k+1)^{2\lambda +2}}\leq C\sum_{k=0}^\infty \frac{c_\lambda (n,m,k)}{(k+1)^{\lambda +2}}\leq \frac{C}{|n-m|^2}
,\quad t>0.
$$

On the other hand, by (\ref{3.7}) and (\ref{3.60}) we can write
\begin{align*}
|J_{1,2}(n,m,t)|&\leq \sum_{\substack{m-n+1\leq k\leq m\\n+m+k-1\;{\rm even}}}c_\lambda (n-1,m,k)\left|\frac{\sqrt{w_\lambda (k-1)}c_\lambda(n,m,k-1)}{\sqrt{w_\lambda (k)}c_\lambda (n-1,m,k)}-1\right|h_t^\lambda (k) \\
&\leq \sum_{\substack{m-n+1\leq k\leq m\\n+m+k-1\;{\rm even}}}\frac{c_\lambda (n-1,m,k)}{(k+1)^{\lambda +1}}\left|\frac{\sqrt{w_\lambda (k-1)}c_\lambda(n,m,k-1)}{\sqrt{w_\lambda (k)}c_\lambda (n-1,m,k)}-1\right|\quad t>0.
\end{align*}

Let $k\in \mathbb{N}$, $m-n+1\leq k\leq m$ such that $n+m+k-1=2\sigma $, with $\sigma\in \mathbb{N}$. As before, we can see that 
$$
\frac{\sqrt{w_\lambda (k-1)}c_\lambda(n,m,k-1)}{\sqrt{w_\lambda (k)}c_\lambda (n-1,m,k)}=\sqrt{\frac{n}{n+2\lambda-1}}\sqrt{\frac{n+\lambda}{n+\lambda-1}}\frac{k+\lambda -1}{k+\lambda }\frac{\sigma -n+1}{\sigma-n+\lambda+1}
\frac{\sigma -k+\lambda+1}{ \sigma -k+1}=:\prod_{j=1}^5\beta _j.
$$
There exists $C>0$ such that $|\beta_j|\leq C$, $j=1,...,5$, and
$$ |\beta_j-1|\leq \frac{C}{n+1}, \;\;j=1,2,\quad |\beta_3-1|\leq \frac{C}{k+1}, \quad |\beta_4-1|\leq \frac{C}{\sigma -n+1},\quad \mbox{ and }
|\beta_5-1|\leq \frac{C}{\sigma-k+1}.
$$
Then,
\begin{align*}
\left|\frac{\sqrt{w_\lambda (k-1)}c_\lambda(n,m,k-1)}{\sqrt{w_\lambda (k)}c_\lambda (n-1,m,k)}-1 \right|&\leq C\sum_{j=1}^5|\beta _j-1|\leq C\left(\frac{1}{n+1}+\frac{1}{k+1}+\frac{1}{m-n+k}+\frac{1}{m+n-k}\right)\\
&\leq C\frac{n+k}{(n+1)(k+1)}\leq \frac{m+1}{(n+1)(k+1)},
\end{align*}
and since $m\leq 2n$ we deduce, by considering (\ref{3.9}) (for $\alpha =1/2$), that
\begin{align*}
|J_{1,2}(n,m,t)|&\leq C\sum_{k=0}^\infty \frac{c_\lambda (n-1,m,k)}{(k+1)^{\lambda +2}}\leq \frac{C}{|n-m|^2},\quad t>0.
\end{align*}
Hence, we have established that
\begin{equation}\label{J1}
J_1(n,m,t)\leq \frac{C}{|n-m|^2} ,\quad t>0.
\end{equation}

Finally we deal with $J_2$. We write
\begin{align*}
J_2(n,m,t)&=\sum_{k=m+1}^{m+n-1}c_\lambda (n,m,k+1)\Big[h_t^\lambda (k+1)-\frac{\sqrt{w_\lambda (k+1)}}{\sqrt{w_\lambda (k)}}h_t^\lambda (k)\Big]\\
&\quad +\sum_{k=m+1}^{m+n-1}\Big[\frac{\sqrt{w_\lambda (k+1)}}{\sqrt{w_\lambda (k)}}c_\lambda (n,m,k+1)-c_\lambda (n-1,m,k)\Big]h_t^\lambda (k)\\
&=:J_{2,1}(n,m,t)+J_{2,2}(n,m,t),\quad t>0.
\end{align*}
Again, by taking into account (\ref{3.8}), (\ref{3.9}) (for $\alpha =1/2$) and (\ref{3.110}) we deduce
$$
|J_{2,1}(n,m,t)|\leq \sum_{k=m+1}^{m+n-1}c_\lambda (n,m,k+1)\frac{\sqrt{w_\lambda (k+1)}}{(k+1)^{2\lambda +2}}\leq C\sum_{k=0}^\infty \frac{c_\lambda (n,m,k)}{(k+1)^{\lambda +2}}\leq \frac{C}{|n-m|^2}
,\quad t>0.
$$
 Also, according to by (\ref{3.8}) and (\ref{3.60}), we have that
$$
|J_{2,2}(n,m,t)|\leq \sum_{\substack{m+1\leq k\leq m+n-1\\n+m+k+1\;{\rm even}}}\frac{c_\lambda (n-1,m,k)}{(k+1)^{\lambda +1}}\left|\frac{\sqrt{w_\lambda (k+1)}c_\lambda(n,m,k+1)}{\sqrt{w_\lambda (k)}c_\lambda (n-1,m,k)}-1\right|, 
\quad t>0.
$$

Consider $k\in \mathbb{N}$, $m+1\leq k\leq m+n$, such that $n+m+k+1=2\sigma $, with $\sigma \in \mathbb{N}$. We can write
$$
\frac{\sqrt{w_\lambda (k+1)}c_\lambda(n,m,k+1)}{\sqrt{w_\lambda (k)}c_\lambda (n-1,m,k)}=
\sqrt{\frac{n}{n+2\lambda-1}}\sqrt{\frac{n+\lambda}{n+\lambda-1}}\frac{k+\lambda +1}{k+\lambda }\frac{\sigma -m+2\lambda -1}{\sigma-m}
\frac{\sigma +2\lambda-1}{ \sigma +\lambda}=:\prod_{j=1}^5\gamma _j.
$$
There exists $C>0$ such that $|\gamma_j|\leq C$, $j=1,...,5$, and
$$
 |\gamma_j-1|\leq \frac{C}{n+1}, \;\;j=1,2,\quad |\gamma_3-1|\leq \frac{C}{k+1}, \quad |\gamma_4-1|\leq \frac{C}{\sigma -m},\quad \mbox{ and }
|\gamma_5-1|\leq \frac{C}{\sigma}.
$$

Since $m\leq 2n$, by proceeding as in the case of $J_{1,2}$, and using again (\ref{3.9}) (for $\alpha =1/2$) it follows that
\begin{align*}
|J_{2,2}(n,m,t)|&\leq C\sum_{k=m+1}^{m+n-1}\frac{c_\lambda (n-1,m,k)}{(k+1)^{\lambda +1}}\left(\frac{1}{n+1}+\frac{1}{k+1}+\frac{1}{n-m+k+1}+\frac{1}{n+m+k+1}\right)\\
&\leq \sum_{k=0}^\infty \frac{c_\lambda (n-1,m,k)}{(k+1)^{\lambda +2}}\leq \frac{C}{|n-m|^2},\quad t>0.
\end{align*}

Thus, 
\begin{equation}\label{J2}
J_2(n,m,t)\leq \frac{C}{|n-m|^2} ,\quad t>0.
\end{equation}
Estimations \eqref{A2d}, \eqref{J3}, \eqref{J1} and \eqref{J2} allow us to conclude property (A2). 

We are going now to establish \eqref{3.3}.  
First observe that when $|n-m|>2|n-\ell |$, $n,m,\ell \in \mathbb{N}$, then 
\begin{equation}\label{nm2l}
\max\{n,m\}>\min\{n,m\}+|n-\ell|\quad \mbox{ and }\quad |n-m|-|n-\ell |>\frac{|n-m|}{2}.
\end{equation}

Let $n,m,\ell \in \mathbb{N}$, $n\not =m$, $|n-m|>2|n-\ell|$ and $m/2\leq n,\ell\leq 3m/2$.  Suppose that $n\leq \ell$. In this case, we have that
$$
\|K_t^\lambda (n,m)-K_t^\lambda (\ell , m)\|_\mathbb{B}\leq \sum_{j=0}^{\ell -n-1}\|K_t^\lambda (n+j,m)-K_t^\lambda (n+j+1, m)\|_\mathbb{B}.
$$
If $n>m$, we can apply $(A1)$ and obtain
$$
\|K_t^\lambda (n,m)-K_t^\lambda (\ell , m)\|_\mathbb{B}\leq C\sum_{j=0}^{\ell -n-1}\frac{1}{(n+j-m)^2}\leq C\frac{\ell -n}{|n-m|^2}.
$$
When $n<m$, \eqref{nm2l} leads to $n+j+1\leq n+ (\ell -n)<m$, $j=0,...,n-\ell -1$, and we can apply $(A2)$ to get the estimate.

In a similar way if $\ell \leq n$, we write
$$
\|K_t^\lambda (n,m)-K_t^\lambda (\ell , m)\|_\mathbb{B}\leq \sum_{j=0}^{n-\ell -1}\|K_t^\lambda (n-j,m)-K_t^\lambda (n-j-1, m)\|_\mathbb{B},
$$
and use $(A2)$ when $n<m$, and $(A1)$ if $n>m$, since in this case by \eqref{nm2l} it follows that
$$
n-j-1\geq \ell=n-(n-\ell )>m,\quad j=0,...,n-\ell -1.
$$
\end{proof}

By invoking Theorem \ref{Th2.1} we deduce that, for every $1\leq p<\infty $ and $w\in A_p(\mathbb{N})$ the operator $T$ can be extended from $\ell ^2(\mathbb{N})\cap \ell ^p(\mathbb{N},w)$ to $\ell ^p(\mathbb{N}, w)$ as a bounded operator from $\ell ^p(\mathbb{N},w)$ into $\ell ^p_\mathbb{B}(\mathbb{N},w)$, when $p\in (1,\infty)$ and from $\ell^1(\mathbb{N},w)$ into $\ell ^{1,\infty }_\mathbb{B}(\mathbb{N},w)$.

Let $p\in [1,\infty)$, $w\in A_p(\mathbb{N})$ and denote by $\mathbb{T}$ the extension obtained. We are going to see that, for every $f \in \ell ^p(\mathbb{N},w)$,  
\begin{equation}\label{extension}
 [\mathbb{T}(f)(n)](t)=( h_t ^\lambda \#_\lambda f)(n),\quad n\in \mathbb{N}\mbox{ and }t>0,
\end {equation}
and, thus, we prove that the maximal operator $W_*^\lambda $ is bounded from $\ell ^p(\mathbb{N},w)$ into itself.

Let $n\in \mathbb{N}$ and $t>0$. We consider the operator
$$
\begin{array}{l}
P_{t,n}:\ell ^p(\mathbb{N}, w)\longrightarrow \mathbb{C}\\
\hspace{1.5cm}f\longrightarrow P_{t,n}(f):=( h_t^\lambda \#_\lambda f)(n).
\end{array}
$$

We show that $P_{t,n}$ is a bounded operator. Assume first that $p\in (1,\infty)$. For every $f\in \ell ^2(\mathbb{N})\cap \ell ^p(\mathbb{N},w)$, we have that 
$$
|w(n)^{1/p}W_t^\lambda (f)(n)|\leq \|W_*^\lambda (f)\|_{\ell ^p(\mathbb{N},w)}=\|\mathbb{T}(f)\|_{\ell ^p(\mathbb{N},w)}\leq C\|f\|_{\ell ^p(\mathbb{N},w)}.
$$
Then, for each $f\in \ell ^2(\mathbb{N})\cap \ell ^p(\mathbb{N},w)$,
$$
\left|\sum_{m\in \mathbb{N}}f(m)_\lambda\tau_n(h_t^\lambda )(m)\right|=|(h_t^\lambda \#f)(n)|=|W_t^\lambda (f)(n)|\leq C\frac{\|f\|_{\ell ^p(\mathbb{N},w)}}{w(n)^{1/p}}.
$$
Hence, $\;_\lambda  \tau _n h_t^\lambda \in (\ell ^p(\mathbb{N},w))'=\ell ^{p'}(\mathbb{N}, w^{-1/(p-1)})$ and we obtain that $P_{t,n}$ is bounded. 

If $p=1$, it is sufficient to establish that the sequence
$$
\left(\frac{\;_\lambda \tau _n(h_t^\lambda)(m)}{w_\lambda (m)}\right)_{m\in \mathbb{N}},
$$
is in $\ell ^\infty (\mathbb{N})$. 

By taking into account \cite[(9)]{Hi1} and estimations \eqref{3.8} and (\ref{3.60}), we have 
$$
0\leq \;_\lambda \tau _n(h_t^\lambda )(n)\leq C\sum_{k=0}^\infty c_\lambda (n,n,k)\frac{\sqrt{w_\lambda (k)}}{(k+1)^{2\lambda +1}}
\leq C\sum_{k=0}^\infty c_\lambda(n,n,k)w_\lambda (k)=Cw_\lambda (n).
$$

On the other hand, size condition (\ref{3.2}) says that
$$
\;_\lambda \tau _n(h_t^\lambda )(m)\leq \frac{C}{|n-m|},\quad m\not=n,
$$
and, since $w\in A_1(\mathbb{N})$, it follows that, there exists $C>0$ for which
$$
\frac{1}{w(m)}\leq \max_{0\leq k\leq m}\frac{1}{w(k)}\sum _{k=0}^m\frac{w(k)}{m+1}\left(\sum_{k=0}^m\frac{w(k)}{m+1}\right)^{-1}
\leq C(m+1)\left(\sum_{k=0}^mw(k)\right)^{-1}\leq C\frac{m+1}{w(0)},\quad m\in \mathbb{N}.
$$
Then, we can find $C>0$ such that
$$
\frac{\;_\lambda \tau _n(h_t^\lambda)(m)}{w_\lambda (m)}\leq C\frac{m+1}{|n-m|}\leq C,\quad m\not=n,
$$
and we conclude that $P_{t,n}$ is bounded in $\ell ^1(\mathbb{N},w)$.

Let us now consider $f\in \ell ^p(\mathbb{N},w)$ and choose a sequence $(f_k)_{k\in \mathbb{N}}\in \ell ^2(\mathbb{N})\cap \ell ^p(\mathbb{N},w)$ such that $f_k\longrightarrow f$, as $k\rightarrow \infty$, in $\ell ^p(\mathbb{N},w)$.
The boundedness of $\mathbb{T}$ implies that 
$$
\mathbb{T}(f_k)=T(f_k)\longrightarrow \mathbb{T}(f),\mbox{ as }k\rightarrow  \infty , 
$$
in $\ell ^p_\mathbb{B}(\mathbb{N},w)$, when $p\in (1,\infty)$ and in $\ell ^{1,\infty }_\mathbb{B}(\mathbb{N},w)$, when $p=1$.

Suppose $p\in (1,\infty)$. Hence, 
$$
T(f_k)(n;t)=P_{t,n}(f_k)\longrightarrow [\mathbb{T}(f)(n)](t),\mbox{ as }k\rightarrow  \infty, \mbox{ in }\mathbb{C},
$$
and, according to the boundedness of $P_{t,n}$, we can conclude that
$$
 [\mathbb{T}(f)(n)](t)=P_{t,n}(f)=(f\#_\lambda h_t ^\lambda )(n).
$$

When $p=1$, we consider $F_k=(F_k(m))_{m\in \mathbb{N}}$, $k\in \mathbb{N}$, and $F=(F(m))_{m\in \mathbb{N}}$, where $F_k(m)=P_{t,m}(f_k)$ and $F(m)=P_{t,m}(f)$, $m,k\in \mathbb{N}$. 

Since $P_{t,m}$, $m\in  \mathbb{N}$, is a bounded operator, we have that
$$
F_k(m)\longrightarrow F(m), \mbox{ as }k\rightarrow \infty,\quad m\in \mathbb{N}.
$$
Also, we have that $F_k=T(f_k)\longrightarrow \mathbb{T}(f)$, as $k\rightarrow \infty$, in $\ell ^{1,\infty }_\mathbb{B}(\mathbb{N},w)$. Then we can conclude that $\mathbb{T}(f)=F$, that is, \eqref{extension} is verified when $p=1$.

\section{Proof of Theorem \ref{Th1} for Littlewood-Paley functions}\label{S4}

Let us consider along this section the Banach space $\mathbb{B}=L^2((0,\infty ),dt/t)$. The following technical lemma will be useful.

\begin{lema}\label{tecnico}
Let $k\in \mathbb{N}$, $a,\alpha, \beta \in \mathbb{R}$ such that $k+2a>\beta +1>a$, $k+\beta +1>0$ and $k+\alpha +1>0$. We have that
\begin{equation}\label{integral}
I_k^{a,\alpha ,\beta}=\left\|t^{k+a}\int_{-1}^1e^{-2t(1+s)}(1-s)^{k+\alpha }(1+s)^{k+\beta}ds\right\|_\mathbb{B}\leq C\frac{\Gamma (k+1)}{(k+1)^{\beta -2a+2}},
\end{equation}
for certain $C>0$ which does not depend on $k$.
\end{lema}
\begin{proof}
We write
\begin{align*}
(I_k^{a,\alpha ,\beta})^2&=\int_0^\infty t^{2k+2a-1}\left(\int_{-1}^1e^{-2t(1+s)}(1-s)^{k+\alpha}(1+s)^{k+\beta}ds\right)^2dt\\
&=\int_{-1}^1\int_{-1}^1(1-s)^{k+\alpha}(1+s)^{k+\beta}(1-z)^{k+\alpha}(1+z)^{k+\beta} \int_0^\infty t^{2k+2a-1}e^{-2t(2+s+z)}dtdsdz\\
&=\frac{\Gamma (2k+2a)}{2^{2k+2a}}\int_{-1}^1\int_{-1}^1\frac{(1-s)^{k+\alpha}(1+s)^{k+\beta}(1-z)^{k+\alpha}(1+z)^{k+\beta}}{(2+s+z)^{2k+2a}}dsdz\\
&=\frac{\Gamma (2k+2a)}{2^{4a-2\alpha -2\beta-2}}\int_0^1\int_0^1\frac{(1-u)^{k+\alpha}u^{k+\beta}(1-v)^{k+\alpha}v^{k+\beta}}{(u+v)^{2k+2a}}dudv\\
&\leq\frac{\Gamma (2k+2a)}{2^{4a-2\alpha -2\beta-2}}\int_0^1(1-u)^{k+\alpha}u^{k+\beta}\int_0^\infty \frac{v^{k+\beta}}{(u+v)^{2k+2a}}dvdu\\
&=\frac{\Gamma (k+\beta +1)\Gamma (k+2a-\beta-1)}{2^{4a-2\alpha -2\beta-2}}\int_0^1(1-u)^{k+\alpha}u^{2\beta-2a+1}du\\
&=\frac{\Gamma (k+\beta +1)\Gamma (k+2a-\beta-1)\Gamma (k+\alpha +1)\Gamma (2\beta -2a+2)}{2^{4a-2\alpha -2\beta-2}\Gamma (k-2a+\alpha+2\beta  +3)}\\
&\leq C_{a,\alpha, \beta}\frac{(\Gamma(k+1))^2}{(k+1)^{2\beta -4a+4}},
\end{align*}
for certain $C_{a,\alpha ,\beta }>0$, and then \eqref{integral} is obtained.
\end{proof}

\subsection{Proof of Theorem \ref{Th1} for $g_{W^\lambda }^1$} \label{s1}

Since $\Delta _\lambda $ is a bounded operator from $\ell ^p( \mathbb{N})$ into itself, $1\leq p\leq \infty$, it follows that, for every $f \in \ell ^p(\mathbb{N})$ and $k\in \mathbb{N}$, $\partial _t^kW_t^\lambda (f)=W_t^\lambda (\Delta _ \lambda ^kf)$, in the sense of derivative in $\ell ^p(\mathbb{N})$. Then, for every $f\in \ell ^p(\mathbb{N})$ and $n\in \mathbb{N}$, $W_t^\lambda (f)(n)$ is smooth in $(0,\infty )$ and $\partial_tW_t^\lambda (f)(n)=W_t^\lambda (\Delta _\lambda ^kf)(n)$, $k\in \mathbb{N}$.

To prove the $\ell ^p$-boundedness properties of the Littlewood-Paley $g_{W^\lambda}^1$ we will apply Theorem \ref{Th2.1}.

Let $G_\lambda$ the operator given by
$$
\begin{array}{c}
G_\lambda :\ell ^2(\mathbb{N}, w)\longrightarrow \ell ^2_\mathbb{B}(\mathbb{N})\\
\hspace{4.5cm}f\longrightarrow G_\lambda (f)(n;t):=t\partial _tW_t^\lambda (f)(n).
\end{array}
$$
As it was commented in Section \ref{S3}, $g_{W^\lambda  }^1$ is a bounded (sublinear) operator from $\ell ^2(\mathbb{N})$ into itself. Then, the operator $G_\lambda$ is bounded. Moreover, by defining $K_t^\lambda (n,m):=_\lambda \!\!\tau _n(t\partial _th_t^\lambda )(m)$, $n,m\in \mathbb{N}$ and $t>0$, we can write, for every $f\in \mathbb{C}_0^\mathbb{N}$,
$$
G_\lambda (f)(n;t)=\sum_{m\in \mathbb{N}}f(m)K_t^\lambda (n,m),\quad n\in \mathbb{N}\mbox{ and }t>0.
$$
Indeed,  for every $f\in \mathbb{C}_0^\mathbb{N}$, we have that
\begin{align*}
t\partial _tW_t^\lambda (f)(n)&=t\partial _t(f\#h_t^\lambda)(n)=\sum_{m\in \mathbb{N}}f(m)t\partial _t[_\lambda \tau _n(h_t^\lambda )(m)]
=\sum_{m\in \mathbb{N}}f(m)\sum_{k=|n-m|}^{n+m}c_\lambda (n,m,k)t\partial _t[h_t^\lambda (k)]\\
&=\sum_{m\in \mathbb{N}}f(m)_\lambda\tau _n(t\partial _th_t^\lambda )(m),\quad n,m\in \mathbb{N}\mbox{ and }t>0.
\end{align*}

Next, we establish that the kernel function $K_t(n,m)$, $n,m\in \mathbb{N}$, $t>0$, satisfies the following properties for a certain $C>0$ and each $n,m,\ell \in \mathbb{N}$, $n\not=m$:
\begin{align}
\label{4.1}&\|K_t^\lambda (n,m)\|_\mathbb{B}\leq \frac{C}{|n-m|};\\
\label{4.2}&\|K_t^\lambda (n,m)-K_t^\lambda (\ell ,m)\|_\mathbb{B}\leq C\frac{|n-\ell|}{|n-m|^2},\quad|n-m|>2|n- \ell|\mbox{ and }\frac{m}{2}\leq n,\ell \leq \frac{3m}{2};\\
\label{4.3}&\|K_t^\lambda (m,n)-K_t^\lambda (m,\ell )\|_\mathbb{B}\leq C\frac{|n-\ell|}{|n-m|^2},\quad|n-m|>2|n- \ell|\mbox{ and }\frac{m}{2}\leq n,\ell \leq \frac{3m}{2}.
\end{align}

We note that, since $K_t^\lambda (n,m)=K_t^\lambda (m,n)$, $n,m\in \mathbb{N}$, $t>0$, we only have to prove (\ref{4.1}) and (\ref{4.2}). For that, we use, by making suitable modifications some of the ideas in \cite[Section 5]{CGRTV}.

\begin{proof}[Proof of \eqref{4.1}] 
We denote by $\psi _t^\lambda (k):=t\partial _th_t^\lambda (k)$, $k\in \mathbb{N}$, $t>0$. Firstly, we show that, there exists $C>0$ such that, 
\begin{equation}\label{4.4}
\|\psi_t^\lambda (k)\|_\mathbb{B}\leq \frac{C}{(k+1)^{\lambda +1}},\quad k\in \mathbb{N}.
\end{equation}
Since (\cite[(5.7.9)]{Leb})
$$
2\frac{d}{dz}I_\nu (z)=I_{\nu -1}(z)+I_{\nu +1}(z),\quad z>0 \mbox{ and }\nu >0,
$$
and
$$
I_{\nu -1}(z)-I_{\nu +1}(z)=\frac{2\nu }{z}I_\nu (z),\quad z>0\mbox{ and }\nu >0,
$$
we deduce that
\begin{align}\label{descomht}
\psi_t^\lambda (k)&=\sqrt{\pi}\Gamma \Big(\lambda +\frac{1}{2}\Big)\sqrt{w_\lambda (k)}t\partial _t[e^{-2t}t^{-\lambda }I_{\lambda +k}(2t)]\nonumber\\
&=\sqrt{\pi}\Gamma \Big(\lambda +\frac{1}{2}\Big)\sqrt{w_\lambda (k)}te^{-2t}t^{-\lambda}
\Big[-\lambda t^{-1}I_{\lambda +k}(2t)-2I_{\lambda +k}(2t)+2\Big(\frac{d}{dt}I_{\lambda +k}\Big)(2t)\Big]\nonumber\\
&=\sqrt{\pi}\Gamma \Big(\lambda +\frac{1}{2}\Big)\sqrt{w_\lambda (k)}t^{1-\lambda }e^{-2t}\Big[\frac{k}{\lambda +k}I_{\lambda +k-1}(2t)-2I_{\lambda +k}(2t)+\frac{2\lambda +k}{\lambda +k}I_{\lambda +k+1}(2t)\Big]\\
&=\frac{k}{\lambda +k}\frac{\sqrt{w_\lambda (k)}}{\sqrt{w_\lambda (k-1)}}th_t^\lambda (k-1)-2th_t^\lambda (k)+\frac{2\lambda +k}{\lambda +k}\frac{\sqrt{w_\lambda (k)}}{\sqrt{w_\lambda (k+1)}}th_t^\lambda (k+1)\nonumber\\
&=\frac{k}{\lambda +k}\left[\frac{\sqrt{w_\lambda (k)}}{\sqrt{w_\lambda (k-1)}}th_t^\lambda (k-1)-th_t^\lambda (k)\right]\nonumber\\
&\quad +\frac{2\lambda +k}{\lambda +k}\left[\frac{\sqrt{w_\lambda (k)}}{\sqrt{w_\lambda (k+1)}}th_t^\lambda (k+1)-th_t^\lambda (k)\right],\quad k\in \mathbb{N},t>0,\label{descomht2}
\end{align}
where we assume $h_t^\lambda (-1):=0$, $t>0$.
Then, we can write
\begin{align}\label{4.5}
\|\psi_t^\lambda (k)\|_\mathbb{B}&\leq \frac{k}{\lambda +k}\frac{\sqrt{w_\lambda (k)}}{\sqrt{w_\lambda (k-1)}}\left\|th_t^\lambda (k-1)-\frac{\sqrt{w_\lambda (k-1)}}{\sqrt{w_\lambda (k)}}th_t^\lambda (k)\right\|_\mathbb{B}\nonumber\\
&\quad +\frac{2\lambda +k}{\lambda +k}\left\|\frac{\sqrt{w_\lambda (k)}}{\sqrt{w_\lambda (k+1)}}th_t^\lambda (k+1)-th_t^\lambda (k)\right\|_\mathbb{B}, \quad k\in \mathbb{N}.
\end{align}
Now, we are going to establish that
\begin{equation}\label{F}
L(k):=\left\|\frac{\sqrt{w_\lambda (k)}}{\sqrt{w_\lambda (k+1)}}th_t^\lambda (k+1)-th_t^\lambda (k)\right\|_\mathbb{B}\leq \frac{C}{(k+1)^{\lambda +1}}, \quad k>\lambda +\frac{5}{2}.
\end{equation}

By partial integration in (\ref{3.11}) we can obtain (\cite[(35)]{CGRTV})
\begin{equation}\label{4.6}
I_\nu (z)=\frac{z^{\nu -2}}{\sqrt{\pi}2^{\nu -2}\Gamma (\nu -3/2)}\int_{-1}^1\frac{e^{-zs}}{z}(1+zs)s(1-s^2)^{\nu -5/2}ds,\quad z>0\mbox{ and }\nu >\frac{3}{2}.
\end{equation}
By using (\ref{3.11}) and (\ref{4.6}) we have that
\begin{align*}
L(k)&=\sqrt{\pi }\Gamma \Big(\lambda +\frac{1}{2}\Big)\sqrt{w_\lambda (k)}\|e^{-2t}t^{1-\lambda }[I_{\lambda +k+1}(2t)-I_{\lambda +k}(2t)]\|_\mathbb{B}\\
&\leq C\frac{\sqrt{w_\lambda (k)}}{\Gamma (\lambda +k-1/2)}\left\|t^{k-1}\int_{-1}^1e^{-2t(1+s)}(1-s^2)^{\lambda +k-3/2}s(1+2t(1+s))ds\right\|_\mathbb{B}\\
&\leq C\frac{\sqrt{w_\lambda (k)}}{\Gamma (\lambda +k-1/2)}\left(\left\|t^{k-1}\int_{-1}^1e^{-2t(1+s)}(1-s^2)^{\lambda +k-3/2}(1+s)ds\right\|_\mathbb{B}\right.\\
&\quad \left.+\left\|t^k\int_{-1}^1e^{-2t(1+s)}(1-s^2)^{\lambda +k-3/2}(1+s)^2ds\right\|_\mathbb{B}\right).
\end{align*}

By taking into account \eqref{3.8} and Lemma \ref{tecnico} we can conclude that
$$
L(k)\leq \frac{\sqrt{w_\lambda (k)}k!}{\Gamma (\lambda +k-1/2)}\left(\frac{1}{(k+1)^{\lambda +7/2}}+\frac{1}{(k+1)^{\lambda +5/2}}\right)\leq \frac{C}{(k+1)^{\lambda +1}},\quad k>\lambda +\frac{5}{2},
$$
and (\ref{F}) is proved.

Then, by using (\ref{3.8}) and (\ref{F}), we obtain
\begin{equation}\label{c1}
\|\psi_t^\lambda (k)\|_\mathbb{B}\leq \frac{k}{\lambda +k}\frac{\sqrt{w_\lambda (k)}}{\sqrt{w_\lambda (k-1)}}L(k-1)+\frac{2\lambda +k}{\lambda +k}L(k)\leq \frac{C}{(k+1)^{\lambda +1}},\quad \mbox{ when }k>\lambda +\frac{7}{2}.
\end{equation}
On the other hand, \cite[(5.16.4) and (5.11.10)]{Leb} say that, for every $\nu >0$,
\begin{equation}\label{4.7}
I_\nu (z)\sim \frac{z^\nu }{2^\nu \Gamma (\nu +1)},\quad \mbox{ as }z\rightarrow 0^+,
\end{equation}
and, for every $n\in \mathbb{N}$,
\begin{equation}\label{4.8}
e^{-z}\sqrt{z}I_\nu (z)=\frac{1}{\sqrt{2\pi }}\left(\sum_{r=0}^n(-1)^r[\nu ,r](2z)^{-r}+\mathcal{O}(z^{-n-1})\right),
\end{equation}
where $[\nu ,0]=1$, and
$$
[\nu ,r]=\frac{(4\nu ^2-1)(4\nu ^2-3^2)\cdots(4\nu ^2-(2r-1)^2)}{2^{2r}\Gamma (r+1)},\quad r\in \mathbb{N}\setminus\{0\}.
$$

From (\ref{4.7}) we deduce that, for every $k\in \mathbb{N}$, there exists $C_k>0$ such that
\begin{equation}\label{F1}
|th_t^\lambda (k)|\leq C_kt^{k+1},\quad t\in (0,1).
\end{equation}
By (\ref{descomht}) and (\ref{4.8}) it follows that
\begin{align}\label{F2}
\psi _t^\lambda (k)&=\sqrt{\pi }\Gamma \Big(\lambda +\frac{1}{2}\Big)\sqrt{w_\lambda (k)}t^{1-\lambda }e^{-2t}\left[\frac{k}{\lambda +k}I_{\lambda +k-1}(2t)-2I_{\lambda+k}(2t)+\frac{2\lambda +k}{\lambda +k}I_{\lambda +k+1}(2t)\right]\nonumber\\
&=\frac{1}{2}\Gamma \Big(\lambda +\frac{1}{2}\Big)\sqrt{w_\lambda (k)}t^{1/2-\lambda }\left[\frac{k}{\lambda +k}-2+\frac{2\lambda +k}{\lambda +k}+O\Big(\frac{1}{t}\Big)\right]\nonumber\\
&=O(t^{-1/2-\lambda }),\quad t>0\mbox{ and }k\in \mathbb{N}.
\end{align}

By combining \eqref{descomht2}, \eqref{F1} and \eqref{F2} we obtain, for every $k\in \mathbb{N}$, a constant $C_k>0$ such that
$$
\|\psi _t^\lambda (k)\|_\mathbb{B}\leq C_k.
$$

This estimate, jointly (\ref{c1}), gives (\ref{4.4}) with a constant $C>0$ that does not depend on $k$.

By considering (\ref{3.9}) (for $\alpha =0$) and (\ref{4.4}) we can write
$$
\|K_t^\lambda (n,m)\|_\mathbb{B}\leq \sum_{k=|n-m|}^{n+m}c_\lambda (n,m,k)\|\psi_t^\lambda (k)\|_\mathbb{B}\leq \sum_{k=0}^\infty \frac{c_\lambda (n,m,k)}{(k+1)^{\lambda +1}}\leq \frac{C}{|n-m|},
$$
and (\ref{4.1}) is established.
\end{proof}

\begin{proof}[Proof of \eqref{4.2}]
In order to show \eqref{4.2}, we are going to proceed as in the proof of (\ref{3.3}). Thus we need to justify the following estimations:
\begin{align*}
(B1) &\hspace{1cm}\|K_t ^\lambda (n+1,m)-K_t^\lambda (n,m)\|_\mathbb{B}\leq \frac{C}{|n-m|^2},\quad n,m\in \mathbb{N},n>m;\\
(B2) &\hspace{1cm} \|K_t^\lambda (n,m)-K_t^\lambda (n-1,m)\|_\mathbb{B}\leq \frac{C}{|n-m|^2},\quad n,m\in \mathbb{N},\,1\leq n<m\leq 2n.
\end{align*}

Let $n,m\in \mathbb{N}$ and $n>m$. As in the proof of (\ref{3.3}) we can write
\begin{equation}\label{KI1I2}
K_t^\lambda (n+1,m)-K_t^\lambda (n,m)= I_1(n,m,t)+I_2(n,m,t), \quad t>0,
\end{equation}
where
$$
I_1(n,m,t):=\sum_{k=n-m}^{n+m}c_\lambda (n+1,m,k+1)\Big[\psi_t^\lambda (k+1)-\frac{\sqrt{w_\lambda (k+1)}}{\sqrt{w_\lambda (k)}}\psi_t^\lambda (k)\Big],\quad t>0,
$$
and 
$$
I_2(n,m,t):=\sum_{k=n-m}^{n+m}\Big[\frac{\sqrt{w_\lambda (k+1)}}{\sqrt{w_\lambda (k)}}c_\lambda (n+1,m,k+1)-c_\lambda (n,m,k)\Big]\psi_t^\lambda (k)
,\quad t>0.
$$

By using the estimations (\ref{calphaj}), (\ref{4.4})  and \eqref{3.9} (for $\alpha =0$) we deduce that
\begin{equation}\label{I2g}
\|I_2(n,m,\cdot)\|_\mathbb{B}\leq \frac{C}{n-m}\sum_{k=n-m}^{n+m}\frac{c_\lambda(n,m,k)}{(k+1)^{\lambda +1}}\leq \frac{C}{|n-m|^2}.
\end{equation}

On the other hand, if we show that
\begin{equation}\label{psidif}
\left\|\psi_t^\lambda (k+1)-\frac{\sqrt{w_\lambda (k+1)}}{\sqrt{w_\lambda (k)}}\psi_t^\lambda (k)\right\|_\mathbb{B}\leq \frac{C}{(k+1)^{\lambda +2}},\quad k\in \mathbb{N},
\end{equation}
we can write, by using also (\ref{3.9}) (for $\alpha =1/2$), 
$$
\|I_1(n,m,\cdot)\|_\mathbb{B}\leq C\sum_{k=n-m}^{n+m}\frac{c_\lambda(n+1,m,k+1)}{(k+1)^{\lambda +2}}\leq \frac{C}{|n-m|^2},
$$
which, jointly (\ref{KI1I2}) and \eqref{I2g}), gives property $(B1)$.

Let us justify (\ref{psidif}). 
By using (\ref{3.11}) and (\ref{4.6}) we get
\begin{align*}
\left\|\psi_t^\lambda (k+1)-\frac{\sqrt{w_\lambda (k+1)}}{\sqrt{w_\lambda (k)}}\psi_t^\lambda (k)\right\|_\mathbb{B}&\\
&\hspace{-5cm}=\sqrt{\pi }\Gamma \Big(\lambda +\frac{1}{2}\Big)\sqrt{w_\lambda (k+1)}\left\|t\partial _t[e^{-2t}t^{-\lambda}(I_{\lambda +k+1}(2t)-I_{\lambda +k}(2t)]\right\|_\mathbb{B}\\
&\hspace{-5cm}\leq C\frac{\sqrt{w_\lambda (k+1)}}{\Gamma (\lambda +k-1/2)}\left\|t\partial _t\Big[t^{k-2}\int_{-1}^{1}e^{-2t(1+s)}(1-s^2)^{\lambda +k-3/2}s(1+2t(1+s))ds\Big]\right\|_\mathbb{B}\\
&\hspace{-5cm}\leq C\frac{\sqrt{w_\lambda (k+1)}}{\Gamma (\lambda +k-1/2)}\left(\left\|(k-2)t^{k-2}\int_{-1}^{1}e^{-2t(1+s)}(1-s^2)^{\lambda +k-3/2}s(1+2t(1+s))ds\Big]\right\|_\mathbb{B}\right.\\
&\hspace{-5cm}\quad +\left\|t^{k-1}\int_{-1}^{1}e^{-2t(1+s)}(1-s^2)^{\lambda +k-3/2}s(1+s)(1+2t(1+s))ds\right\|_\mathbb{B}\\
&\hspace{-5cm}\quad \left.+\left\|t^{k-1}\int_{-1}^{1}e^{-2t(1+s)}(1-s^2)^{\lambda +k-3/2}s(1+s)ds\right\|_\mathbb{B}\right)\\
&\hspace{-5cm}\leq C\frac{\sqrt{w_\lambda (k+1)}}{\Gamma (\lambda +k-1/2)}\left(\left\|(k+1)t^{k-2}\int_{-1}^{1}e^{-2t(1+s)}(1-s^2)^{\lambda +k-3/2}(1+s)ds\right\|_\mathbb{B}\right.\\
&\hspace{-5cm}\quad +\left\|(k+1)t^{k-1}\int_{-1}^{1}e^{-2t(1+s)}(1-s^2)^{\lambda +k-3/2}(1+s)^2ds\right\|_\mathbb{B}\\
&\hspace{-5cm}\quad +\left.\left\|t^k\int_{-1}^{1}e^{-2t(1+s)}(1-s^2)^{\lambda +k-3/2}(1+s)^3ds\right\|_\mathbb{B}\right).
\end{align*}

Now, by using estimate \eqref{3.8} and Lemma \ref{tecnico} we obtain
\begin{align*}
\left\|\psi_t^\lambda (k+1)-\frac{\sqrt{w_\lambda (k+1)}}{\sqrt{w_\lambda (k)}}\psi_t^\lambda (k)\right\|_\mathbb{B}&\leq C\frac{(k+1)^\lambda k!}{\Gamma (\lambda +k-1/2)}\left(\frac{1}{(k+1)^{\lambda +9/2}}+\frac{1}{(k+1)^{\lambda +7/2}}\right)\\
&\hspace{-3cm}\leq C\frac{k!}{\Gamma (\lambda +k-1/2)(k+1)^{7/2}}\leq \frac{C}{(k+1)^{\lambda +2}},\quad \mbox{ when }k >\lambda +9/2. 
\end{align*}

Also, by taking into account (\ref{4.4}), if $k\in \mathbb{N}$, $k\leq\lambda +9/2$,
\begin{align*}
\left\|\psi_t^\lambda (k+1)-\frac{\sqrt{w_\lambda (k+1)}}{\sqrt{w_\lambda (k)}}\psi_t^\lambda (k)\right\|_\mathbb{B}&\leq C(\|\psi_t^\lambda (k+1)\|_\mathbb{B}+\|\psi _t^\lambda (k)\|_\mathbb{B})\\
&\leq \frac{C}{(k+1)^\lambda }\leq C\leq \frac{C}{(k+1)^{\lambda +2}},
\end{align*}
and thus, \eqref{psidif} is established.

Property $(B2)$ can be obtained by proceeding as in the proof of $(A2)$, and using \eqref{4.4} and \eqref{psidif}. Moreover, $(B1)$ and $(B2)$ lead to \eqref{4.2} in the same way as \eqref{3.3} was proved from $(A1)$ and $(A2)$.

\end{proof}

According to Theorem \ref{Th1} we deduce that for every $1\leq p<\infty$ and $w\in A_p(\mathbb{N})$ the operator $G_\lambda $ can be extended from $\ell ^2(\mathbb{N})\cap \ell ^p(\mathbb{N},w)$ to $\ell ^p(\mathbb{N},w)$ as a bounded operator from $\ell ^p(\mathbb{N},w)$ into $\ell ^p_\mathbb{B}(\mathbb{N},w)$, when $1<p<\infty$ and from $\ell ^1(\mathbb{N},w)$ into $\ell ^{1,\infty }_\mathbb{B}(\mathbb{N},w)$. Let us denote by $\mathbb{T}$ to this extension. Our objective now is to show that 
$$
[\mathbb{T}(f)(n)](t)=t\partial _tW_t(f)(n), \quad n\in \mathbb{N} \mbox{ and }t\in (0,\infty )\setminus E,
$$
for certain $E\subset (0,\infty )$ with $|E|=0$. 
Thus, we can conclude that the Littlewood-Paley function $g_{W^\lambda }^1$ is bounded from $\ell ^p(\mathbb{N},w)$ into itself, when $1<p<\infty$ and from $\ell ^1(\mathbb{N},w)$ into $\ell ^{1,\infty }(\mathbb{N},w)$.

Let $1<p<\infty$ and $w\in A_p(\mathbb{N})$. Suppose $f\in \ell ^p(\mathbb{N},w)$ and that $(f_\ell )_{\ell =0}^\infty $ is a sequence in $\mathbb{C}_0^\mathbb{N}$ for which $f_\ell \longrightarrow f$, as $\ell \rightarrow \infty$, in $\ell ^p(\mathbb{N},w)$. 
We have that
$$
G_\lambda (f_\ell )\longrightarrow \mathbb{T}(f),\quad \mbox{ as }\ell \rightarrow \infty ,\mbox{ in }\ell _\mathbb{B} ^p(\mathbb{N},w).
$$
Then, for every $n\in \mathbb{N}$,
$$
\int_0^\infty |t\partial _tW_t^\lambda (f_\ell )(n)-[\mathbb{T}(f)(n)](t)|^2\frac{dt}{t}\longrightarrow 0,\quad \mbox{ as }\ell \rightarrow \infty.
$$
Let $n\in \mathbb{N}$. There exists an increasing sequence $\phi:\mathbb{N}\longrightarrow \mathbb{N}$ and a measurable set $E\subset (0,\infty )$ with zero Lebesgue measure such that, for every $t\in (0,\infty )\setminus E$,
$$
t\partial _tW_t^\lambda (f_{\phi (\ell )})(n)\longrightarrow [\mathbb{T}(f)(n)](t),\quad \mbox{ as }\ell \rightarrow \infty.
$$

We consider the following maximal operators
$$
W_*^{\lambda ,-}(g)(\ell ):=\sup_{t>0}\left|\sum _{m=0}^\infty g(m)\sum_{r=1}^\infty c_\lambda (\ell ,m,r)h_t^\lambda (r-1)\right|,\quad \ell \in \mathbb{N},
$$
and
$$
W_*^{\lambda ,+}(g)(\ell ):=\sup_{t>0}\left|\sum _{m=0}^\infty g(m)\sum_{r=1}^\infty c_\lambda (\ell ,m,r)h_t^\lambda (r+1)\right|,\quad \ell \in \mathbb{N}.
$$
By proceeding as in Section \ref{S3} we can prove that these maximal operators are bounded from $\ell ^p(\mathbb{N},w)$ into itself.
Also, by \eqref{1.1}, \eqref{5prime} and \eqref{3.1} we have that
$$
|t\partial _tW_t^\lambda (g)(n)|\leq Ct[W_*^\lambda (|g|)(n)+W_*^{\lambda ,+}(|g|)(n)+W_*^{\lambda ,-}(|g|)(n)],\quad t>0.
$$

Let $t_0\in (0,\infty)$. Since $W_*^\lambda $, $W_*^{\lambda ,+}$ and $W_*^{\lambda , -}$ are bounded operators from $\ell ^p(\mathbb{N},w)$ into itself, it follows that
$$
[t\partial _tW_t^\lambda (f_{\phi (\ell )})(n)]_{|t=t_0}\longrightarrow [t\partial _tW_t^\lambda (f)(n)]_{|t=t_0},\quad \mbox{ as }\ell \rightarrow \infty.
$$
Hence, for every $t\in (0,\infty )\setminus E$, $[\mathbb{T}(f)(n)](t)=t\partial _tW_t^\lambda (f)(n)$.

In a similar way we can prove that $g_{W^\lambda }^1$ is bounded from $\ell ^1(\mathbb{N},w)$ into $\ell ^{1,\infty }(\mathbb{N},w)$.

\vspace*{0.2cm}

\subsection{ Proof of Theorem \ref{Th1} for $g_{W^\lambda }^k$, $k>1$}\label{s2}
In order to show that Littlewood-Paley functions $g_{W^\lambda }^k$, $k>1$, are bounded from $\ell ^p(\mathbb{N},w)$ into itself, when $1<p<\infty$ and $w\in A_p(\mathbb{N})$, we use a reduction argument that we learnt from Professor C. Segovia and J.L. Torrea.

We need the following form of a Krivine's result (\cite[Theorem 1.f.14, p. 93]{LT}).

\begin{teo}\label{Th4.1}
Let $\mathbb{H}_i$, $i=1,2$, be Hilbert spaces and let $(\Omega,\mathcal{A},\mu )$ be a measure space. Assume that $1<p<\infty$, $w\in A_p(\mathbb{N})$ and $T$ is a bounded operator from $\ell _{\mathbb{H}_1}^p(\mathbb{N}, w)$ into $\ell _{\mathbb{H}_2}^p(\mathbb{N},w)$. We define the operator $\widetilde{T}$ by
$$
\widetilde{T}(F)(n,\theta )=T(F(\cdot , \theta ))(n),\quad n\in \mathbb{N}\mbox{ and }\theta \in \Omega,
$$
for every $F:\mathbb{N}\times \Omega \longrightarrow \mathbb{H}_1$ such that $F(\cdot , \theta )\in \ell ^p_{\mathbb{H}_1}(\mathbb{N},w)$, for every $\theta \in \Omega $. Then, $\widetilde{T}$ is a bounded operator from $\ell _{L^2_{\mathbb{H}_1}(\Omega )}^p(\mathbb{N},w)$ into $\ell _{L^2_{\mathbb{H}_2}(\Omega )}^p(\mathbb{N},w)$.
\end{teo}

Let $k\in \mathbb{N}$ and suppose that $g_{W^\lambda }^k$ is bounded from $\ell ^p(\mathbb{N},w)$ into itself, where $1<p<\infty$ and $w\in A_p(\mathbb{N})$. We will show that $g_{W^\lambda}^{k+1}$ is also bounded from $\ell ^p(\mathbb{N},w)$ into itself.
 
For every $m\in \mathbb{N}$ we define the operator
\begin{align*}
G_\lambda ^m:\ell ^p(\mathbb{N},w)\longrightarrow \ell _B^p(\mathbb{N}, w)&\\
&\hspace{-3cm}f\longrightarrow G_\lambda ^m(f)(n;t):=t^m\partial _t^mW_t^\lambda (f)(n).
\end{align*}

We have that $G_\lambda ^1$ and $G_\lambda ^k$ are bounded from $\ell ^p(\mathbb{N},w)$ into $\ell _\mathbb{B}^p(\mathbb{N},w)$. We also consider the operator $\widetilde{G_\lambda^1}$ defined by
$$
\widetilde{G_\lambda ^1}(F)(n,\theta ;t)=G_\lambda ^1(F(\cdot ,\theta ))(n;t),\quad n\in \mathbb{N},\;\theta ,\;t\in (0,\infty ).
$$
According to Theorem \ref{Th4.1} for $\mathbb{H}_1=\mathbb{C}$ and $\mathbb{H}_2=\mathbb{B}$, the operator $\widetilde{G_\lambda ^1}$ is bounded from $\ell ^p_\mathbb{B}(\mathbb{N},w)$ into $\ell _{L^2_\mathbb{B}((0,\infty ),d\theta /\theta)}(\mathbb{N},w)$.

A straightforward manipulation (see, for instance, \cite[Proposition 2.5]{BFRTT} leads to
$$
\|G_\lambda ^{k+1}(f)(n;\cdot)\|_\mathbb{B}=\sqrt{2k(2k+1)}\|\widetilde{G_\lambda ^1}[G_\lambda ^k(f)](n,\cdot; \cdot )\|_{L^2_\mathbb{B}((0,\infty ),d\theta/\theta)},\quad f\in \mathbb{C}_0^\mathbb{N}.
$$
Hence, $G_\lambda ^{k+1}$ defines a bounded operator from $\ell ^p(\mathbb{N},w)$ into $\ell _\mathbb{B}^p(\mathbb{N},w)$.

Thus, we have established that $g_{W^\lambda }^k$ is bounded from $\ell ^p(\mathbb{N},w)$ into $\ell _\mathbb{B}^p(\mathbb{N},w)$, for every $k\in \mathbb{N}$.

\begin{remark}
The above argument does not allow us to obtain the weak (1,1) boundedness for the Littlewood-Paley function $g_{W^\lambda }^k$ when $k>1$. In order to get the $\ell ^1$-boundedness properties for $g_{W^\lambda }^k$  we need to obtain a treatable expression for $\partial _t^kh_t^\lambda $, $t\in (0,\infty )$, for every $k\in \mathbb{N}$. We can not get this general form for $\partial _t^kh_t^\lambda$, $t\in (0,\infty )$.
\end{remark}

\subsection{Proof of Theorem \ref{Th1} for $g_{P^\lambda }^k$}
Let $k\in \mathbb{N}\setminus\{0\}$. 
We consider the operator
\begin{align*}
Q^{\lambda ,k}:\ell ^2(\mathbb{N})\longrightarrow \ell _\mathbb{B}^2(\mathbb{N})&\\
&\hspace{-2cm}f\longrightarrow Q^{\lambda ,k}(f)(n;t):=t^k\partial _t^kP_t^\lambda (f)(n).
\end{align*}

According to (\ref{5prime}) and (\ref{1.5}) we have that, for every $f\in \mathbb{C}_0^\mathbb{N}$,
$$
P_t^\lambda (f)(n)=\frac{1}{\sqrt{\pi }}\int_0^\infty \frac{e^{-u}}{\sqrt{u}}W_{\frac{t^2}{4u}}^\lambda (f)(n)du=\sum_{m=0}^\infty f(m) M_t^\lambda (n,m),\quad n\in \mathbb{N}\mbox{ and }t>0,
$$
where
$$
M_t^\lambda(n,m):=\frac{1}{\sqrt{\pi }}\int_0^\infty \frac{e^{-u}}{\sqrt{u}}\;_\lambda \tau _n(h_{t^2/(4u)}^\lambda )(m)du,\quad n,m\in \mathbb{N}\mbox{ and }t>0.
$$

By making a change of variables we can write
$$
M_t^\lambda (n,m)=\frac{t}{2\sqrt{\pi }}\int_0^\infty \frac{e^{-t^2/(4u)}}{u^{3/2}}\;_\lambda \tau _n(h_u^\lambda )(m)du=\;_\lambda \tau _n\left[\frac{1}{2\sqrt{\pi }}\int_0^\infty \frac{te^{-t^2/(4u)}}{u^{3/2}}h_u^\lambda (k)du\right](m).
$$

For every $t>0$ we define $\Psi_t^{\lambda ,k}$ by 
$$
\Psi _t^{\lambda ,k}(\ell ):=\frac{t^k}{2\sqrt{\pi }}\int_0^\infty \frac{\partial _t^k[te^{-t^2/(4u)}]}{u^{3/2}}h_u^\lambda (\ell )du,\quad \ell \in \mathbb{N}.
$$
Then, for every $f\in \mathbb{C}_0^\mathbb{N}$, we have that
$$
t^k\partial _t^kP_t^\lambda (f)(n)=\sum_{m=0}^\infty f(m)Q_t^{\lambda , k}(n,m),\quad n\in \mathbb{N}\mbox{ and }t>0,
$$
where
$$
Q_t^{\lambda , k}(n,m):=\;_\lambda \tau _n(\Psi _t^{\lambda ,k})(m),\quad n,m\in \mathbb{N}.
$$

By using spectral theory we can see that $Q^{\lambda , k}$ is a bounded operator from $\ell ^2(\mathbb{N})$ into $\ell ^2_\mathbb{B}(\mathbb{N})$. Indeed, let $f\in \ell ^2(\mathbb{N})$. We have that
$$
\mathcal{F}_\lambda (P_t^\lambda (f))(x)=e^{-\sqrt{2(1-x)}t}\mathcal{F}_\lambda (f)(x),\quad x\in (-1,1),\;t>0.
$$

Plancherel equality leads to
\begin{align*}
\sum_{n=0}^\infty \|Q^{\lambda ,k}(f)(n;\cdot )\|_\mathbb{B}^2&=\int_0^\infty \sum_{n=0}^\infty |Q^{\lambda ,k}(f)(n;t)|^2\frac{dt}{t}\\
&=\int_0^\infty \int_{-1}^1(t\sqrt{2(1-x)})^{2k}e^{-2\sqrt{2(1-x)}t}|\mathcal{F}_\lambda (f)(x)|^2dx\frac{dt}{t}\\
&=\frac{\Gamma (2k)}{2^{2k}}\int_{-1}^1|\mathcal{F}_\lambda (f)(x)|^2dx=\frac{\Gamma (2k)}{2^{2k}}\sum_{n=0}^\infty |f(n)|^2.
\end{align*}

We are going to see that there exists $C>0$ such that, for every  $n,m,\ell \in \mathbb{N}, \;n\not=m$, 
\begin{equation}\label{4.161}
\|Q_t^{\lambda ,k}(n,m)\|_\mathbb{B}\leq \frac{C}{|n-m|};
\end{equation}
\begin{equation}\label{4.162}
\|Q_t^{\lambda ,k}(n,m)-Q_t^{\lambda ,k}(\ell ,m)\|_\mathbb{B}\leq C\frac{|n-\ell |}{|n-m|^2},\quad |n-m|>2|n-\ell|,\mbox{ and }\frac{m}{2}\leq n,\ell \leq \frac{3m}{2},
\end{equation}
and
\begin{equation}\label{4.163}
\|Q_t^{\lambda ,k}(m,n)-Q_t^{\lambda ,k}(m,\ell )\|_\mathbb{B}\leq C\frac{|n-\ell |}{|n-m|^2},\quad |n-m|>2|n-\ell|,\mbox{ and }\frac{m}{2}\leq n,\ell \leq \frac{3m}{2}.
\end{equation}

In order to obtain these estimations we can proceed as in the heat case in Section \ref{s1}. Thus, we only need to establish analogous properties to \eqref{4.4} and \eqref{psidif} for $\Psi_t^{\lambda ,k}$, that is,
\begin{equation}\label{acotPsi}
\|\Psi _t^{\lambda ,k}(\ell )\|_\mathbb{B}\leq \frac{C}{(\ell +1)^{\lambda +1}},\quad \ell \in \mathbb{N}\setminus \{0\},
\end{equation}
and
\begin{equation}\label{Psidif}
\left\|\Psi _t^{\lambda ,k}(\ell +1)-\frac{\sqrt{w_\lambda (\ell +1)}}{\sqrt{w_\lambda (\ell )}}\Psi_t^{\lambda , k}(\ell )\right\|_\mathbb{B}\leq \frac{C}{(\ell +1)^{\lambda +2}},\quad \ell \in \mathbb{N}\setminus \{0\}.
\end{equation}
\begin{remark}
Actually, since we assume $n,m\in \mathbb{N}$, $n\not=m$, and 
$$
\;_\lambda \tau _n(\Psi _t^{\lambda ,k})=\sum_{\ell=|n-m|}^{n+m}c_\lambda (n,m,\ell)\Psi _t^{\lambda ,k}(\ell ),
$$
we only need to establish properties \eqref{acotPsi} and \eqref{Psidif} for $\ell \in \mathbb{N}\setminus \{0\}$. 
\end{remark}

By using Fa di Bruno's formula (\cite[Lemma 4.3, (4.6)]{GLLNU}) we can obtain (see \cite[Lemma 4]{BCCFR})
$$
|\partial _t^k[te^{-t^2/(4u)}]|\leq Ce^{-t^2/(8u)}u^{(1-k)/2},\quad t,u\in (0,\infty ).
$$
It follows that
\begin{equation}\label{4.17}
\|t^k\partial _t^k[te^{-t^2/(4u)}]\|_\mathbb{B}\leq Cu^{(1-k)/2}\left(\int_0^\infty t^{2k-1}e^{-t^2/(4u)}dt\right)^{1/2}\leq C\sqrt{u},\quad u>0.
\end{equation}

Assume $\ell \in \mathbb{N}$, $\ell \geq 1$. Let us prove \eqref{acotPsi}. By the Minkowski integral's inequality and (\ref{3.5}) we can write
\begin{align*}
 \|\Psi _t^{\lambda ,k}(\ell )\|_\mathbb{B}&\leq C\int_0^\infty \|t^k\partial _t^k[te^{-t^2/(4u)}]\|_\mathbb{B}\frac{h_u^\lambda (\ell )}{u^{3/2}}du\leq C\sqrt{w_\lambda (\ell )}\int_0^\infty e^{-2u}u^{-\lambda -1}I_{\lambda +\ell }(2u)du\\
&\leq C\frac{\sqrt{w_\lambda (\ell )}}{\Gamma (\lambda +\ell +1/2)}\int_0^\infty u^{\ell -1}\int_{-1}^1e^{-2u(1+s)}(1-s^2)^{\lambda +\ell -1/2}dsdu\\
&\leq C\frac{\sqrt{w_\lambda (\ell )}\Gamma (\ell)}{2^\ell \Gamma (\lambda +\ell +1/2)}\int_{-1}^1(1-s^2)^{\lambda +\ell -1/2}(1+s)^{-\ell }ds=C\frac{\sqrt{w_\lambda (\ell )}\Gamma (\ell)}{\Gamma (2\lambda+ \ell +1)}\leq \frac{C}{(\ell +1)^{\lambda +1}}.
\end{align*}

In order to prove \eqref{Psidif} we use (\ref{3.5}) and (\ref{3.11})  and again (\ref{4.17}) to get
\begin{align*}
\left\|\Psi _t^{\lambda ,k}(\ell +1)-\frac{\sqrt{w_\lambda (\ell +1)}}{\sqrt{w_\lambda (\ell )}}\Psi_t^{\lambda , k}(\ell )\right\|_\mathbb{B}&\leq C\sqrt{w_\lambda (\ell +1)}\int_0^\infty  e^{-2u}u^{-\lambda -1}|I_{\lambda +\ell +1}(2u)-I_{\lambda +\ell }(2u)|du\\
&\hspace{-3cm}\leq C\frac{\sqrt{w_\lambda (\ell +1)}}{\Gamma (\lambda +\ell +1/2)}\int_0^\infty u^{\ell -1}\int_{-1}^1e^{-2u(1+s)}(1+s)(1-s^2)^{\lambda +\ell -1/2}dsdu\\
&\hspace{-3cm}\leq C\frac{\sqrt{w_\lambda (\ell +1)}\Gamma (\ell)}{2^\ell \Gamma (\lambda +\ell +1/2)}\int_{-1}^1(1-s^2)^{\lambda +\ell -1/2}(1+s)^{1-\ell }ds\\
&\hspace{-3cm}=C\frac{\sqrt{w_\lambda (\ell +1)}\Gamma (\ell)}{\Gamma (2\lambda +\ell +2)}\leq \frac{C}{(\ell +1)^{\lambda +2}}.
\end{align*}

By proceeding as in the proof of the corresponding property for the heat case in Section \ref{s1}, by using (\ref{acotPsi}) and (\ref{Psidif}) we can prove (\ref{4.161}), (\ref{4.162}) and (\ref{4.163}).

The proof of the $\ell ^p$-boundedness properties for $g_{P^\lambda }^k$ can be finished now similarly to the case of $g_{W^\lambda }^1$ (see Section \ref{s1}).

\section{Proof of Proposition \ref{Prop1}}\label{S5}

The right hand side inequality in (\ref{acotg}) was established in Theorem \ref{Th1}.
To prove the left hand side one we need to show the following polarization formulas.
\begin{lema}\label{Lem5.1a}
Let $\lambda >0$, $1<p<\infty$ and $w\in A_p(\mathbb{N})$. Then, for every $f\in \ell ^p(\mathbb{N},w)$ and $g\in \ell ^{p'}(\mathbb{N},v)$, where $1/p+1/p'=1$ and $v=w^{-p'/p}$,
\begin{equation}\label{5.1a}
\sum_{n=0}^\infty \int_0^\infty t^k\partial _t^kW_t^\lambda (f)(n)\overline{t^k\partial _t^kW_t^\lambda (g)(n)}\frac{dt}{t}=\frac{\Gamma (2k)}{2^{2k}}\sum_{n=0}^\infty f(n)\overline{g(n)},
\end{equation}
and
\begin{equation}\label{5.2b}
\sum_{n=0}^\infty \int_0^\infty t^k\partial _t^kP_t^\lambda (f)(n)\overline{t^k\partial _t^kP_t^\lambda (g)(n)}\frac{dt}{t}=\frac{\Gamma (2k)}{2^{2k}}\sum_{n=0}^\infty f(n)\overline{g(n)},
\end{equation}
\end{lema}

\begin{proof}
Let $f,g\in \mathbb{C}_0^\mathbb{N}$. We can write, by proceeding as in \eqref{eq:9},
\begin{align*}
\sum_{n=0}^\infty \int_0^\infty t^k\partial _t^kW_t^\lambda (f)(n)\overline{t^k\partial _t^kW_t^\lambda (g)(n)}\frac{dt}{t}&=\int_0^\infty t^{2k-1}\sum_{n=0}^\infty \partial _t^kW_t^\lambda (f)(n)\overline{\partial _t^kW_t^\lambda (g)(n)}dt\\
&\hspace{-1cm}=\int_0^\infty t^{2k-1}\int_{-1}^1\mathcal{F}_\lambda (f)(x)\overline{\mathcal{F}_\lambda (g)(x)}e^{-4t(1-x)}(-2(1-x))^{2k}dxdt\\
&\hspace{-1cm}=\frac{\Gamma (2k)}{2^{2k}}\int_{-1}^1\mathcal{F}_\lambda (f)(x)\overline{\mathcal{F}_\lambda (g)(x)}dx=\frac{\Gamma (2k)}{2^{2k}}\sum_{n=0}^\infty f(n)\overline{g(n)}.
\end{align*}
The interchange between the serie and the integral is legitimated because $f,g\in \mathbb{C}_0^\mathbb{N}$. Since $\mathbb{C}_0^\mathbb{N}$ is a dense subspace of $\ell ^p(\mathbb{N},w)$ and $\ell ^{p'}(\mathbb{N},v)$, by applying Theorem \ref{Th1} and a continuity argument, we can prove the equality (\ref{5.1a}) for every $f\in \ell ^p(\mathbb{N},w)$ and $g\in \ell ^{p'}(\mathbb{N},v)$.

Equality (\ref{5.2b}) can be established similarly.
\end{proof}

Let $1<p<\infty$ and $w\in A_p(\mathbb{N})$. By taking into account that $(\ell ^p(\mathbb{N},w))'=\ell ^{p'}(\mathbb{N},w^{-p'/p})$, (\ref{5.1a}) and Theorem \ref{Th1} we get
\begin{align*}
\|f\|_{\ell ^p(\mathbb{N},w)}&=\sup_{g\in \mathcal{G}}\left|\sum_{n=0}^\infty f(n)\overline{g(n)}\right|=\frac{2^{2k}}{\Gamma (2k)}\sup_{g\in \mathcal{G}}\left|\sum_{n=0}^\infty \int_0^\infty t^k\partial _t^k W_t^\lambda (f)(n)\overline{t^k\partial _t^kW_t^\lambda (g)(n)}\frac{dt}{t}\right|\\
&\leq C\sup_{g\in \mathcal{G}}\|t^k\partial _t^kW_t^\lambda (f)\|_{\ell _\mathbb{B}^p(\mathbb{N},w)}\|t^k\partial _t^kW_t^\lambda (g)\|_{\ell _\mathbb{B}^{p'}(\mathbb{N},w^{-p'/p})}\leq C\|t^k\partial _t^kW_t^\lambda (f)\|_{\ell _\mathbb{B}^p(\mathbb{N},w)},
\end{align*}
where
$\mathcal{G}$ represents the set of functions $g\in \ell ^{p'}(\mathbb{N},w^{-p'/p})$ such that $\|g\|_{\ell ^{p'}(\mathbb{N},w^{-p'/p})}\leq 1$. Thus, the left hand side inequality in (\ref{acotg}) in the heat case is proved. For the Poisson semigroup $\{P_t^\lambda \}_{t>0}$ we can proceed in a similar way.

\section{Proof of Theorem \ref{Th2}}

Let $\lambda ,\mu>1$. The transplantation operator $\mathcal{T}_{\lambda ,\mu }$ is defined by
$$
\mathcal{T}_{\lambda , \mu }(f)=\mathcal{F}_\mu ^{-1}(\mathcal{F}_\lambda f),\quad f\in \ell ^2(\mathbb{N}).
$$
Since $\mathcal{F}_\lambda $ is an isometric isomorphism from $\ell ^2(\mathbb{N})$ into $L^2(-1,1)$, $\mathcal{T}_{\lambda ,\mu }$ is an isometry from $\ell^2(\mathbb{N})$ into itself. If $f\in \mathbb{C}_0^\mathbb{N}$, we have that
$$
\mathcal{T}_{\lambda ,\mu }(f)(n)=\mathcal{F}_\mu ^{-1}\left(\sum_{m=0}^\infty f(m)\varphi _m^\lambda \right)(n)=\sum_{m=0}^\infty f(m)K_{\lambda , \mu }(n,m),\quad n\in \mathbb{N},
$$
where
$$
K_{\lambda , \mu }(n,m):=\int_{-1}^1\varphi _n^\mu (x)\varphi _m^\lambda (x)dx,\quad n,m\in \mathbb{N}.
$$

Moreover, if $f\in \ell ^2(\mathbb{N})$, then
$$
\mathcal{T}_{\lambda ,\mu }(f)(n)=\lim_{k\rightarrow\infty}\sum_{m=0}^kf(m)K_{\lambda ,\mu }(n,m),
$$
in the sense of convergence in $\ell ^2(\mathbb{N})$ and also pointwisely.

Since $\varphi _k^\lambda $ is an odd (even) function when $k$ is odd (even), we obtain
$$
K_{\lambda ,\mu }(n,m)=\left\{\begin{array}{ll}
					\displaystyle 2\int_0^1\varphi _n^\mu (x)\varphi _m^\lambda (x)dx,&\mbox{ if $n$ and $m$ are both odd or both even,}\\
					&\\
					0,&\mbox{ otherwise.}
				       \end{array}
\right.
$$

We can write, for every $f\in \ell ^2(\mathbb{N})$,
$$
\mathcal{T}_{\lambda ,\mu }(f)(n)=\left\{\begin{array}{ll}
					\displaystyle \sum_{m=0}^\infty f(2m)K_{\lambda ,\mu }(n,2m),&\mbox{ if $n$ is even,}\\
					&\\
					\displaystyle \sum_{m=0}^\infty f(2m+1)K_{\lambda ,\mu}(n,2m+1),&\mbox{ if $n$ is odd.}
				       \end{array}
\right.
$$

We now define the following two operators:
$$
\mathcal{T}_{\lambda ,\mu }^{\rm e}(f)(n):=\sum_{m=0}^\infty f(m)K_{\lambda , \mu}^{\rm e}(n,m),\quad f\in \ell ^2(\mathbb{N}),
$$
and
$$
\mathcal{T}_{\lambda ,\mu }^{\rm o}(f)(n):=\sum_{m=0}^\infty f(m)K_{\lambda ,\mu }^{\rm o}(n,m),\quad f\in \ell ^2(\mathbb{N}),
$$
where 
$$
K_{\lambda ,\mu }^{\rm e}(n,m):=K_{\lambda ,\mu }(2n,2m),\quad n,m\in \mathbb{N},$$
and
$$
K_{\lambda ,\mu }^{\rm o}(n,m):=K_{\lambda ,\mu }(2n+1,2m+1),\quad n,m\in \mathbb{N}.
$$

The superindices e and o refer to even and odd, respectively. Note that, for every $f\in \ell ^2(\mathbb{N})$,
$$
\mathcal{T}_{\lambda ,\mu }(f)(2n)=\mathcal{T}_{\lambda ,\mu }^{\rm e}(\widetilde{f})(n),\quad n\in \mathbb{N},
$$
and
$$
\mathcal{T}_{\lambda ,\mu }(f)(2n+1)=\mathcal{T}_{\lambda ,\mu }^{\rm o}(\widehat{f})(n),\quad n\in \mathbb{N},
$$
where $\widetilde{f}(n):=f(2n)$, and $\widehat{f}(n):=f(2n+1)$, $n\in \mathbb{N}$.

Since the operator $\mathcal{T}_{\lambda ,\mu}$ is bounded from $\ell ^2(\mathbb{N})$ into itself, the operators $\mathcal{T}_{\lambda ,\mu}^{\rm e}$ and $\mathcal{T}_{\lambda ,\mu }^{\rm o}$ are bounded from $\ell ^2(\mathbb{N})$ into itself.
It is sufficient to note that if $f\in \ell ^2(\mathbb{N})$, then
$$
\mathcal{T}_{\lambda ,\mu}^{\rm e}(f)(n)=\mathcal{T}_{\lambda ,\mu}(g)(2n)\quad \mbox{ and }\quad \mathcal{T}_{\lambda ,\mu}^{\rm o}(f)(n)=\mathcal{T}_{\lambda ,\mu}(h)(2n+1), \quad n\in \mathbb{N},
$$
where
$$
g(n)=f\Big(\frac{n}{2}\Big)\chi_{\{m\in \mathbb{N}:\;m { \rm \;is \;even }\}}(n)\quad \mbox{ and }\quad h(n)=f\Big(\frac{n-1}{2}\Big)\chi_{\{m\in \mathbb{N}:\;m {\rm \;is \; odd }\}}(n),\quad n\in \mathbb{N}.
$$

We are going to prove that there exists $C>0$ such that, for each $n,m,\ell \in \mathbb{N}$, $n\not=m$,
\begin{align}
\label{5.1}&|K_{\lambda ,\mu }(n,m)|\leq \frac{C}{|n-m|}, \; \mbox{ provided that }(\mu -1)/2<\lambda <\mu,
\\
&\label{5.2} |K_{\lambda ,\mu }^{\rm s}(n,m)-K_{\lambda ,\mu }^{\rm s}(\ell ,m)|\leq C\frac{|n-\ell |}{|n-m|^2},\quad \frac{m}{2}\leq n,\ell\leq \frac{3m}{2},\\
&\label{5.3}|K_{\lambda ,\mu }^{\rm s}(m,n)-K_{\lambda ,\mu }^{\rm s}(m,\ell)|\leq C\frac{|n-\ell |}{|n-m|^2},\quad \frac{m}{2}\leq n,\ell \leq \frac{3m}{2},
\end{align}
when ${\rm s}={\rm o}$ or $s={\rm e}$.

\begin{remark}\label{conditions}
Actually, we will also establish that \eqref{5.1} is satisfied for all $\lambda ,\mu >0$, when $n,m \in \mathbb{N}$, $n\not =m$ and $m/2\leq n\leq 3m/2$, that \eqref{5.2} holds for all $\lambda >0$ and $\mu >1$, and \eqref{5.3}, for all $\lambda >1$ and $\mu >0$.
\end{remark}
We will use the following two lemmas established in \cite[p. 49 and p. 59]{Sz2} (see also \cite[p. 400]{AW}), which refers to the ultraspherical polynomials $\mathcal{P}_k^\lambda$ in \eqref{ultraspherical}.

\begin{lema}\label{Lem5.1}
Assume that $\gamma >0$ is not an integer. Then, for every $k,r\in \mathbb{N}$, and $\theta \in (0,\pi )$,
\begin{equation}\label{5.7}
\mathcal{P}_k^\gamma (\cos \theta )=A_{k,r}^\gamma (\theta )+R_{k,r }^\gamma (\theta ),
\end{equation}
where
$$
A_{k,r}^\gamma (\theta ):=\sum_{\ell =0}^{r-1}b_{\ell}^\gamma \frac{k!}{\Gamma (k+ \ell +\gamma +1)}
\frac{\cos ((k+\ell +\gamma )\theta -(\ell +\gamma )\pi /2)}{(2\sin \theta )^{\ell +\gamma }},
$$
with
$$
b_{\ell }^\gamma :=\frac{2}{\pi }\sin (\gamma \pi )\frac{\Gamma (2\gamma )}{\Gamma (\gamma )}\frac{\Gamma (\ell +\gamma )\Gamma (\ell -\gamma +1)}{\ell !},\quad \ell=0,...,r-1,
$$
and
$$
|R_{k,r}^\gamma (\theta )|\leq C(k\sin \theta )^{-(r+\gamma) },\quad \theta \in (0,\pi ),
$$
being $C>0$ independent of $\theta \in (0,\pi )$ and $k,r\in \mathbb{N}$.
\end{lema}

\begin{lema}\label{Lem5.2}
Assume that $\gamma \in \mathbb{N}$, $\gamma \geq 1$. Then, for every $k\in \mathbb{N}$ and $ \theta \in (0,\pi )$,
$$
\mathcal{P}_k^\gamma (\cos \theta )=\sum_{\ell =0}^{\gamma -1}\widetilde{b}_\ell ^\gamma \frac{k!}{\Gamma (k+\ell +\gamma +1)}\frac{\cos ((k+\ell +\gamma )\theta -(\ell +\gamma )\pi /2)}{(2\sin \theta )^{\ell +\gamma }},
$$
where
$$
\widetilde{b}_\ell ^\gamma :=(-1)^\ell \frac{2\Gamma (2\gamma)\Gamma (\ell +\gamma)}{\Gamma (\gamma)\ell !\Gamma (\gamma -\ell)},\quad \ell=0,...,r-1.
$$
\end{lema}
Note that if $r,k\in \mathbb{N}$,  $r\leq \gamma -1$, and $\theta \in (0,\pi )$,
\begin{equation}\label{5.8}
\mathcal{P}_k^\gamma (\cos \theta )=\widetilde{A}_{k,r}^\gamma (\theta )+\widetilde{R}_{k,r }^\gamma (\theta ),
\end{equation}
where
$$
\widetilde{A}_{k,r}^\gamma :=\sum_{\ell =0}^{r-1}\widetilde{b}_\ell ^\gamma \frac{k!}{\Gamma (k+\ell +\gamma +1)}\frac{\cos ((k+\ell +\gamma )\theta -(\ell +\gamma )\pi /2)}{(2\sin \theta )^{\ell +\gamma }},
$$
and
$$
|\widetilde{R}_{k,r}^\gamma (\theta )|\leq C(k\sin \theta )^{-(r+\gamma)},\quad \theta \in (0,\pi ),
$$
where $C>0$ does not depend on $\theta \in (0,\pi )$ nor on $k,r\in \mathbb{N}$, $r\leq \gamma -1$.
\vspace*{0.2cm}

\begin{proof}[Proof of (\ref{5.1})] Let $(\mu -1)/2<\lambda <\mu$. 

When $n,m\in \mathbb{N}$ and $0\leq n\leq m/2$ or $3m/2\leq n$, according to \cite[pp. 400-401]{AW} and since $(\mu -1)/2<\lambda <\mu$, we have that
$$
|K_{\lambda ,\mu }(n,m)|\leq C\left\{\begin{array}{ll}
						\displaystyle \frac{1}{m},&\displaystyle 0\leq n\leq \frac{m}{2},\\
						&\\
						\displaystyle \frac{1}{n},&\displaystyle n\geq \frac{3m}{2},
            				      \end{array}
\right. 
$$
and \eqref{5.1} is established for every $n,m\in \mathbb{N}$, such that $0\leq n\leq m/2$ or $3m/2\leq n$.

We now assume that $n,m\in \mathbb{N}$, $m/2\leq n\leq 3m/2$, $n\not=m$. To analyze this case we use some of the ideas in \cite{AW}, and, as we can observe, in this case (\ref{5.1}) is satisfied for every $\lambda, \mu >0$.  

Suppose that $n$ and $m$ are both even or both odd. We have that
\begin{align}\label{5.5}
K_{\lambda ,\mu }(n,m)&=2\int_0^{\pi /2}\varphi _n^\mu (\cos \theta )\varphi _m^\lambda (\cos \theta)\sin \theta d\theta=2\left(\int_0^{1/n}+\int_{1/n}^{\pi /2}\right)\varphi _n^\mu (\cos \theta )\varphi _m^\lambda (\cos \theta )\sin \theta d\theta\nonumber\\
&=2\sqrt{w_\lambda (m)w_\mu (n)}\left(\int_0^{1/n}+\int_{1/n}^{\pi /2}\right)\mathcal{P}_n^\mu (\cos \theta )\mathcal{P} _m^\lambda (\cos \theta )(\sin \theta)^{\lambda +\mu} d\theta\nonumber\\
&=:K_0(n,m)+K_1(n,m).
\end{align}

By using (\ref{3.8}), and since $|\mathcal{P}_k^\gamma (x)|\leq 1$, $x\in (-1,1)$, $k\in \mathbb{N}$ and $\gamma >0$, (\cite[Theorem 7.33.1]{Sz}), we get
\begin{equation}\label{5.6}
|K_0(n,m)|\leq Cn^{\lambda +\mu }\int_0^{1/n}\theta ^{\lambda +\mu }d\theta \leq \frac{C}{n},
\end{equation}
where $C>0$ does not depend on $n$.

We use Lemmas \ref{Lem5.1} and \ref{Lem5.2} to estimate $K_1(n,m)$. 
Assume that $\lambda ,\mu $ are not integers. By using (\ref{5.7}) with $r=2$ we get
\begin{align*}
K_1(n,m)&=2\sqrt{w_\lambda (m)w_\mu (n)}\int_{1/n}^{\pi /2}\Big[A_{n,2}^\mu (\theta) A_{m,2}^\lambda (\theta )+A_{n,2}^\mu (\theta)R_{m,2}^\lambda (\theta )\\
&\quad +A_{m,2}^\lambda (\theta )R_{n,2}^\mu (\theta )+R_{n,2}^\mu (\theta )R_{m,2}^\lambda (\theta)\Big](\sin \theta)^{\lambda +\mu}d\theta .
\end{align*}
We observe that, for every $\gamma >0$, $k,r\in \mathbb{N}$, and $\theta \in (0,\pi /2)$,
\begin{equation}\label{acotA}
|A_{k,r}^\gamma ( \theta )|\leq C\frac{k!}{(\sin \theta )^\gamma }\sum_{\ell =0}^{r-1}\frac{1}{\Gamma (k+\ell +\gamma +1)(\sin \theta )^\ell}\leq \frac{C}{[(k+1)\theta]^\gamma}\sum_{\ell =0}^{r-1}\frac{1}{[(k+1)\theta]^\ell}.
\end{equation}
Then, for every $\theta \in (1/n,\pi /2)$, we have that
\begin{align*}
|A_{n,2}^\mu (\theta)R_{m,2}^\lambda (\theta )+A_{m,2}^\lambda (\theta )R_{n,2}^\mu (\theta )+R_{n,2}^\mu (\theta )R_{m,2}^\lambda (\theta)|&\leq C\left(\frac{1}{(n\theta )^{\lambda +\mu +2}}+\frac{1}{(n\theta )^{\lambda +\mu +4}}\right)\leq \frac{C}{(n\theta )^{\lambda +\mu +2}},
\end{align*}
which jointly with (\ref{3.8}), leads to
\begin{align}\label{K1}
|K_1(n,m)|&\leq C\left(n^{\lambda +\mu}\left|\int_{1/n}^{\pi /2}A_{n,2}^\mu (\theta) A_{m,2}^\lambda (\theta )(\sin \theta )^{\lambda +\mu}d\theta\right| +\frac{1}{n^2}\int_{1/n}^\infty \frac{d\theta }{\theta ^2}\right)\nonumber\\
&\leq C\left(n^{\lambda +\mu}\left|\int_{1/n}^{\pi /2}A_{n,2}^\mu (\theta) A_{m,2}^\lambda (\theta )(\sin \theta )^{\lambda +\mu }d\theta \right|+\frac{1}{n}\right).
\end{align}

Next we estimate the first summand in the last inequality. For every $k\in \mathbb{N}$ and $\gamma >0$, we consider $\alpha _k^\gamma (\theta ):=(k+\gamma )\theta -\gamma \pi /2$, $\theta \in (0,\pi /2)$. Then, for every $\theta \in (0,\pi /2)$,
\begin{align*}
A_{n,2}^\mu (\theta) A_{m,2}^\lambda (\theta )&=\frac{n!m!}{(\sin \theta )^{\lambda +\mu}}\left(b_0^\mu b_0^\lambda \frac{\cos (\alpha _n^\mu (\theta ))\cos (\alpha _m^\lambda (\theta ))}{\Gamma (n+\mu +1)\Gamma (m+\lambda +1)}
+b_1^\mu b_0^\lambda \frac{\sin (\alpha _n^\mu (\theta )+\theta )\cos (\alpha _m^\lambda (\theta ))}{\Gamma (n+\mu +2)\Gamma (m+\lambda +1)\sin \theta }\right.\\
&\quad \left.+b_0^\mu b_1^\lambda \frac{\cos (\alpha _n^\mu (\theta )) \sin(\alpha _m^\lambda (\theta )+\theta)}{\Gamma (n+\mu +1)\Gamma (m+\lambda +2)\sin \theta }+
b_1^\mu b_1^\lambda \frac{\sin (\alpha _n^\mu (\theta )+\theta )\sin(\alpha _m^\lambda (\theta )+\theta )}{\Gamma (n+\mu +2)\Gamma (m+\lambda +2)(\sin \theta)^2 }\right).
\end{align*}

Thus, we obtain
\begin{align}\label{K11}
n^{\lambda +\mu}\left|\int_{1/n}^{\pi /2}A_{n,2}^\mu (\theta) A_{m,2}^\lambda (\theta )(\sin \theta)^{\lambda +\mu}d\theta\right|&\leq C\left(\Big|\int_{1/n}^{\pi /2}\cos (\alpha _n^\mu (\theta ))\cos (\alpha _m^\lambda (\theta ))d\theta \Big|\right.\nonumber\\
&\hspace{-4cm}\quad +\Big|\int_{1/n}^{\pi /2}\frac{\sin (\alpha _n^\mu (\theta )+\theta )\cos (\alpha _m^\lambda (\theta ))}{n\sin \theta}d\theta\Big|+\Big|\int_{1/n}^{\pi /2}\frac{\cos (\alpha _n^\mu (\theta )) \sin (\alpha _m^\lambda (\theta )+\theta)}{n\sin \theta}d\theta\Big|\nonumber\\
&\hspace{-4cm}\left.\quad +\Big|\int_{1/n}^{\pi /2}\frac{\sin (\alpha _n^\mu (\theta )+\theta )\sin(\alpha _m^\lambda (\theta )+\theta )}{(n\sin \theta)^2}d\theta\Big|\right)=:C\sum_{j=1}^4I_j(n,m).
\end{align}
We note that
\begin{equation}\label{I4}
I_4(n,m)\leq \frac{C}{n^2}\int_{1/n}^\infty \frac{d\theta }{\theta ^2}\leq \frac{C}{n}.
\end{equation}

For $I_1(n,m)$, we have $I_1(n,m)\leq C$ and, also, when $n+\mu \not=m+\lambda$, 
\begin{align*}
I_1(n,m)&=\frac{1}{2}\left|\int_{1/n}^{\pi /2}\Big(\cos (\alpha _n^\mu (\theta )+\alpha _m^\lambda (\theta ))+\cos (\alpha _n^\mu (\theta )-\alpha _m^\lambda (\theta ))\Big)d\theta \right|\\
&\leq C\left(\frac{|\sin (\alpha _n^\mu (\theta )+\alpha _m^\lambda (\theta ))|}{n+m+\lambda +\mu }\Big|_{\theta =1/n}^{\pi /2}+\frac{|\sin (\alpha _n^\mu (\theta )-\alpha _m^\lambda (\theta ))|}{|n-m+\lambda -\mu |}\Big|_{\theta =1/n}^{\pi /2}\right)\\
&\leq C\left(\frac{1}{n+m+\lambda +\mu}+\frac{1}{|n-m+\mu -\lambda |}\right)\leq C\left(\frac{1}{n}+\frac{1}{|n-m+\mu -\lambda|}\right).
\end{align*}
Then, we get that
$$
I_1(n,m)\leq C\leq \frac{C}{|n-m|},\quad \mbox{ when }|n-m|\leq 2|\mu -\lambda|,
$$
and
$$
I_1(n,m)\leq C\left(\frac{1}{n}+\frac{1}{|n-m+\mu -\lambda|}\right)\leq C\left(\frac{1}{n}+\frac{1}{|n-m|}\right),\quad \mbox{ if }|n-m|>2|\mu -\lambda |,
$$
that is,
\begin{equation}\label{I1}
I_1(n,m)\leq \frac{C}{|n-m|}.
\end{equation}
On the other hand, we have that
\begin{align*}
2\sin (\alpha _n^\mu (\theta )+\theta )\cos (\alpha _m^\lambda (\theta ))&=\sin (\alpha _n^\mu (\theta )+\alpha _m^\lambda (\theta )+\theta)+\sin (\alpha _n^\mu (\theta )-\alpha _m^\lambda (\theta )+\theta)\\
&\hspace{-4cm}=\sin ((n+m+\mu +\lambda +1)\theta )\cos((\lambda +\mu)\pi /2)-\cos ((n+m+\mu +\lambda +1)\theta )\sin((\lambda +\mu)\pi /2)\\
&\hspace{-4cm}\quad +\sin ((n-m+\mu -\lambda +1)\theta )\cos((\mu-\lambda )\pi /2)-\cos ((n-m+\mu -\lambda +1)\theta )\sin((\mu-\lambda )\pi /2).
\end{align*}
Let $A:=n+m+\mu +\lambda +1$ or $A:=n-m+\mu -\lambda +1$. We are going to estimate 
$$
J_1:=\left|\int_{1/n}^{\pi /2}\frac{\sin (A\theta)}{n\sin \theta }d\theta \right|\quad \mbox{ and }\quad J_2:=\left|\int_{1/n}^{\pi /2}\frac{\cos (A\theta)}{n\sin \theta }d\theta \right|,
$$
when $A\not=0$. We can write
\begin{align*}
J_1&\leq \frac{1}{n}\left(\int_{1/n}^{\pi /2}\Big|\frac{1}{\sin \theta }-\frac{1}{\theta}\Big|d\theta +\left|\int_{1/n}^{\pi /2}\frac{\sin (|A|\theta )}{\theta }d\theta\right|\right)\leq \frac{1}{n}\left(\int_{1/n}^{\pi /2}\theta d\theta +\left|\int_{|A|/n}^{|A|\pi /2}\frac{\sin u}{u}du\right|\right)\\
&\leq \frac{C}{n}\left(1+\left|\frac{\cos u}{u}\Big]_{u=|A|/n}^{u=|A|\pi /2}+\int_{|A|/n}^{|A|\pi /2}\frac{\cos u}{u^2}du\right|\right)\leq \frac{C}{n}\left(1+\frac{n}{|A|}+\int_{|A|/n}^\infty \frac{du}{u^2}\right)\leq \frac{C}{n}\left(1+\frac{n}{|A|}\right).
\end{align*}
In a similar way,
$$
J_2\leq \frac{C}{n}\left(1+\left|\int_{|A|/n}^{|A|\pi /2}\frac{\cos u}{u}du\right|\right)\leq \frac{C}{n}\left(1+\left|\frac{\sin u}{u}\Big]_{u=|A|/n}^{u=|A|\pi /2}+\int_{|A|/n}^{|A|\pi /2}\frac{\sin u}{u^2}du\right|\right)
\leq \frac{C}{n}\left(1+\frac{n}{|A|}\right).
$$

Thus, we can deduce that, when $n+\mu\not=m+\lambda +1$, and $n+\mu\not=m+\lambda -1$,
$$
I_2(n,m)+I_3(n,m)\leq \frac{C}{n}\left(1+\frac{n}{n+m+\lambda +\mu+1}+\frac{n}{|n-m+\mu-\lambda +1|}+\frac{n}{|n-m+\mu-\lambda -1|}\right).
$$

Then, if $|n-m|> 2\max\{|\mu-\lambda +1|, |\mu-\lambda -1|\}$, 
$$
I_2(n,m)+I_3(n,m)\leq C\left(\frac{1}{n}+\frac{1}{|n-m|-|\mu-\lambda +1|}+\frac{1}{|n-m|-|\mu-\lambda -1|}\right)\leq C\left(\frac{1}{n}+\frac{1}{|n-m|}\right).
$$
In the case that  $|n-m|\leq 2\max\{|\mu-\lambda +1|, |\mu-\lambda -1|\}$, we write
$$
I_2(n,m)+I_3(n,m)\leq \frac{C}{n}\int_{1/n}^{\pi /2}\frac{d\theta }{\theta}= \frac{C}{n}\log \Big(\frac{\pi}{2}(n+1)\Big)\leq C\leq \frac{C}{|n-m|}.
$$
Hence,
\begin{equation}\label{I23}
I_2(n,m)+I_3(n,m)\leq \frac{C}{|n-m|}.
\end{equation}

By considering estimations \eqref{K1}-\eqref{I23} we get
$$
|K_1(n,m)|\leq C\left(\frac{1}{n}+\frac{1}{|n-m|}\right)\leq \frac{C}{|n-m|},
$$
which, jointly \eqref{5.5} and \eqref{5.6}, leads to
$$
K_{\lambda ,\mu}(n,m)\leq \frac{C}{|n-m|}, \quad \frac{m}{2}\leq n\leq \frac{3m}{2},\;n\not=m,
$$
provided that $m/2\leq n\leq 3m/2$ and $\lambda ,\mu$ are not integers.

When $\lambda $ or $\mu$ is an integer and $m/2\leq n\leq 3n/2$, (\ref{5.1}) can be proved in a similar way by using also (\ref{5.8}).
\end{proof}

\begin{proof}[Proof of (\ref{5.2}) for $K_{\lambda ,\mu }^{\rm e}$] Here we need assume $\mu >1$. Suppose that $\lambda $ and $\mu$ are not integers. In other cases we can proceed in a similar way.

It is sufficient to see that, for every $n,m\in \mathbb{N}$, $m/2\leq n\leq 3m/2$, $n\not=m$,
\begin{equation}\label{98}
|K_{\lambda ,\mu}^{\rm e}(n+1,m)-K_{\lambda ,\mu }^{\rm e}(n,m)|\leq \frac{C}{|n-m|^2}.
\end{equation}
Let $n,m\in \mathbb{N}$ such that $n\not=m$ and $m/2\leq n\leq 3m/2$. 
We can write
\begin{align}\label{descomK}
K_{\lambda ,\mu}^{\rm e}(n+1,m)-K_{\lambda ,\mu }^{\rm e}(n,m)&=2\int_0^1[\varphi _{2n+2}^\mu(x)-\varphi _{2n}^\mu (x)]\varphi _{2m}^\lambda(x)dx\nonumber\\
&\hspace{-4cm}=2\int_0^1\Big[\varphi_{2n+2}^\lambda (x)-\frac{\sqrt{w_\mu (2n+2)}}{\sqrt{w_\mu (2n)}}\varphi_{2n}^\lambda (x)\Big]\varphi _{2m}^\lambda (x)dx\nonumber\\
& \hspace{-4cm}\quad +2\Big[\frac{\sqrt{w_\mu (2n+2)}}{\sqrt{w_\mu (2n)}}-1\Big]\int_0^1\varphi _{2n}^\mu(x)\varphi _{2m}^\lambda (x)dx\nonumber\\
&\hspace{-4cm}=:H_1(n,m)+H_2(n,m).   
\end{align}
We have that
\begin{align*}
\left|\sqrt{\frac{w_\mu (2n+2)}{w_\mu (2n)}}-1\right|&=\left|\frac{w_\mu (2n+2)}{w_\mu (2n)}-1\right|\left|\sqrt{\frac{w_\mu (2n+2)}{w_\mu (2n)}}+1\right|^{-1}\leq \left|\frac{w_\mu (2n+2)}{w_\mu (2n)}-1\right|\\
&=\left|\frac{(2\mu +2n+1)(2\mu +2n)(2n+2+\mu)}{(2n+2)(2n+1)(2n+\mu )}-1\right|\leq \frac{C}{n}.
\end{align*}
Hence, by (\ref{5.1}) and Remark \ref{conditions} we deduce that
\begin{equation}\label{5.17}
|H_2(n,m)|\leq \frac{C}{n}|K_{\lambda ,\mu }(2n,2m)|\leq \frac{C}{n|n-m|}.
\end{equation}
On the other hand, according to \cite[(4.7.29)]{Sz}, for every $x\in (-1,1)$,
$$
\mathcal{P}_{2n+2}^\mu (x)-\mathcal{P}_{2n}^\mu (x)=\left(\frac{(2n+2)(2n+1)}{(2\mu +2n+1)(2\mu +2n)}-1\right)\mathcal{P}_{2n}^\mu (x)+\frac{2(2\mu -1)(2n+1+\mu)}{(2\mu +2n+1)(2\mu +2n)}\mathcal{P}_{2n+2}^{\mu -1}(x).
$$
Then, we can decompose $H_1(n,m)$ as follows:
\begin{align}\label{H1}
H_1(n,m)&=2\sqrt{w_\mu (2n+2)w_\lambda (2m)}\int_0^{\pi /2}[\mathcal{P}_{2n+2}^\mu (\cos \theta )-\mathcal{P}_{2n}^\mu (\cos \theta)]\mathcal{P}_{2m}^\lambda (\cos \theta )(\sin \theta )^{\lambda +\mu}d\theta\nonumber\\
&=2\sqrt{w_\mu (2n+2)w_\lambda (2m)}\nonumber\\
&\quad\times \left[\left(\frac{(2n+2)(2n+1)}{(2\mu +2n+1)(2\mu +2n)}-1\right)\int_0^{\pi /2}\mathcal{P}_{2n}^\mu (\cos\theta)\mathcal{P}_{2m}^\lambda (\cos \theta )(\sin \theta )^{\lambda +\mu}d\theta\right.\nonumber\\
&\quad +2\left.\frac{(2\mu -1)(2n+1+\mu)}{(2\mu +2n+1)(2\mu +2n)}\int_0^{\pi /2}\mathcal{P}_{2n+2}^{\mu -1}(\cos\theta)\mathcal{P}_{2m}^\lambda (\cos \theta )(\sin \theta )^{\lambda +\mu}d\theta\right]\nonumber\\
&=:H_{1,1}(n,m)+H_{1,2}(n,m).
\end{align}

It is not hard to see that
\begin{equation}\label{A3}
\mbox{ (a) }\left|\frac{(2n-2)(2n+1)}{(2\mu +2n+1)(2\mu +2n)}-1\right|\leq \frac{C}{n},\quad \mbox{ and }\quad \mbox{ (b) }\left|\frac{(2\mu -1)(2n+2+\mu)}{(2\mu +2n+1)(2\mu +2n)}\right|\leq \frac{C}{n}.
\end{equation}
Then, by using (\ref{3.8}), (\ref{5.1}) and Remark \ref{conditions} we obtain that
\begin{equation}\label{H11}
|H_{1,1}(n,m)|\leq \frac{C}{n}\frac{\sqrt{w_\mu (2n+2)}}{\sqrt{w_\mu (2n)}}|K_{\lambda ,\mu }(2n,2m)|\leq \frac{C}{n|n-m|}.
\end{equation}
Also, from (\ref{3.8}), (\ref{A3}) (b) and since $|\mathcal{P}_k^\gamma (x)|\leq 1$, $\gamma >0$, $k\in \mathbb{N}$, $x\in (-1,1)$, we have that
\begin{align}\label{H12}
|H_{1,2}(n,m)|&\leq Cn^{\lambda +\mu -1}\left|\left(\int_0^{1/n}+\int_{1/n}^{\pi /2}\right)\mathcal{P}_{2n+2}^{\mu -1}(\cos\theta)\mathcal{P}_{2m}^\lambda (\cos \theta )(\sin \theta )^{\lambda +\mu}d\theta\right|\nonumber\\
&\leq Cn^{\lambda+\mu  -1}\left(\int_0^{1/n}\theta ^{\lambda +\mu}d\theta +\left|\int_{1/n}^{\pi /2}\mathcal{P}_{2n+2}^{\mu -1}(\cos\theta)\mathcal{P}_{2m}^\lambda (\cos \theta )(\sin \theta )^{\lambda +\mu}d\theta\right|\right)\nonumber\\
&\leq C\left(\frac{1}{n^2}+n^{ \lambda+\mu -1}\left|\int_{1/n}^{\pi /2}\mathcal{P}_{2n+2}^{\mu -1}(\cos\theta)\mathcal{P}_{2m}^\lambda (\cos \theta )(\sin \theta )^{\lambda +\mu}d\theta\right|\right).
\end{align}

Lemma \ref{Lem5.1} and \eqref{acotA} with $r=3$ leads to
\begin{align}\label{H12a}
\left|\int_{1/n}^{\pi /2}\mathcal{P}_{2n+2}^{\mu -1}(\cos\theta)\mathcal{P}_{2m}^\lambda (\cos \theta )(\sin \theta )^{\lambda +\mu}d\theta\right|&\leq\left|\int_{1/n}^{\pi /2}A_{2n+2,3}^{\mu -1}(\theta) A_{2m,3}^\lambda (\theta )(\sin \theta)^{\lambda +\mu}d\theta \right|\nonumber\\
&\hspace{-6cm}\quad +\int_{1/n}^{\pi /2}|A_{2n+2,3}^{\mu -1}(\theta)R_{2m,3}^\lambda (\theta )+A_{2m,3}^\lambda (\theta )R_{2n+2,3}^{\mu -1}(\theta )+R_{2n+2,3}^{\mu -1}(\theta )R_{2m,3}^\lambda (\theta)\Big|(\sin \theta)^{\lambda +\mu}d\theta\nonumber\\
&\hspace{-6cm}\leq \left|\int_{1/n}^{\pi /2}A_{2n+2,3}^{\mu -1}(\theta) A_{2m,3}^\lambda (\theta )(\sin \theta)^{\lambda +\mu}d\theta \right| +C\int_{1/n}^\infty \frac{\theta ^{\mu +\lambda}}{(n\theta)^{\lambda +\mu +2}}d\theta \nonumber\\
&\hspace{-6cm}\leq \left|\int_{1/n}^{\pi /2}A_{2n+2,3}^{\mu -1}(\theta) A_{2m,3}^\lambda (\theta )(\sin \theta)^{\lambda +\mu}d\theta \right| +\frac{C}{n^{\lambda +\mu +1}}.
\end{align}

To analyze the integral term we consider, as before, $\alpha _k^\gamma (\theta ):=(k+\gamma )\theta -\gamma \pi /2$, $k\in \mathbb{N}$, $\gamma >0$ and $\theta \in (0,\pi /2)$. Then, for every $\theta \in (0,\pi /2)$,
\begin{align*}
A_{2n+2,3}^{\mu -1}(\theta) A_{2m,3}^\lambda (\theta )&=\frac{(2n+2)!(2m)!}{(\sin \theta )^{\lambda +\mu-1}}\left(b_0^{\mu-1} b_0^\lambda \frac{\cos (\alpha _{2n+2}^{\mu -1}(\theta ))\cos (\alpha _{2m}^\lambda (\theta ))}{\Gamma (2n+2+\mu )\Gamma (2m+\lambda +1)}\right.\\
&\hspace{-2cm}+b_0^{\mu -1}b_1^\lambda \frac{\cos (\alpha _{2n+2}^{\mu -1}(\theta ))\sin(\alpha _{2m}^\lambda (\theta )+\theta)}{\Gamma (2n+2+\mu )\Gamma (2m+\lambda +2)\sin \theta }-b_0^{\mu -1}b_2^\lambda \frac{\cos (\alpha _{2n+2}^{\mu -1}(\theta ))\cos(\alpha _{2m}^\lambda (\theta )+2\theta) }{\Gamma (2n+2+\mu )\Gamma (2m+\lambda +3)(\sin \theta )^2}\\
&\hspace{-2cm}+b_1^{\mu -1}b_0^\lambda \frac{\sin(\alpha _{2n+2}^{\mu -1}(\theta )+\theta)\cos(\alpha _{2m}^\lambda (\theta ))}{\Gamma (2n+3+\mu )\Gamma (2m+\lambda +1)\sin \theta}+b_1^{\mu -1}b_1^\lambda \frac{\sin (\alpha _{2n+2}^{\mu -1}(\theta )+\theta)\sin(\alpha _{2m}^\lambda (\theta )+\theta)}{\Gamma (2n+3+\mu )\Gamma (2m+\lambda +2)(\sin \theta )^2}\\
&\hspace{-2cm}-b_1^{\mu -1}b_2^\lambda \frac{\sin(\alpha _{2n+2}^{\mu -1}(\theta )+\theta)\cos(\alpha _{2m}^\lambda (\theta )+2\theta)}{\Gamma (2n+3+\mu )\Gamma (2m+\lambda +3)(\sin \theta)^3}-b_2^{\mu -1}b_0^\lambda \frac{\cos (\alpha _{2n+2}^{\mu -1}(\theta )+2\theta)\cos(\alpha _{2m}^\lambda (\theta ))}{\Gamma (2n+4+\mu )\Gamma (2m+\lambda +1)(\sin \theta )^2}\\
&\hspace{-2cm}\left.-b_2^{\mu -1}b_1^\lambda \frac{\cos(\alpha _{2n+2}^{\mu -1}(\theta )+2\theta)\sin(\alpha _{2m}^\lambda (\theta )+\theta)}{\Gamma (2n+4+\mu )\Gamma (2m+\lambda +2)(\sin \theta)^3}+b_2^{\mu -1}b_2^\lambda \frac{\cos (\alpha _{2n+2}^{\mu -1}(\theta )+2\theta)\cos(\alpha _{2m}^\lambda (\theta )+2\theta)}{\Gamma (2n+4+\mu )\Gamma (2m+\lambda +3)(\sin \theta )^4}\right).
\end{align*}
Then, 
\begin{align}\label{H12b}
n^{ \lambda+\mu -1}\left|\int_{1/n}^{\pi /2}A_{2n+2,3}^{\mu -1}(\theta) A_{2m,3}^\lambda (\theta )(\sin \theta)^{\lambda +\mu}d\theta \right| &\leq C\left(\left|\int_{1/n}^{\pi /2}\cos (\alpha _{2n+2}^{\mu -1}(\theta ))\cos (\alpha _{2m}^\lambda (\theta ))\sin \theta d\theta\right|\right.\nonumber\\
&\hspace{-6cm}+\frac{1}{n}\left|\int_{1/n}^{\pi /2}\cos (\alpha _{2n+2}^{\mu -1}(\theta ))\sin(\alpha _{2m}^\lambda (\theta )+\theta)d\theta\right|+\frac{1}{n^2}\left|\int_{1/n}^{\pi /2}\frac{\cos (\alpha _{2n+2}^{\mu -1}(\theta ))\cos(\alpha _{2m}^\lambda (\theta)+2\theta)}{\sin \theta }d\theta\right|\nonumber\\
&\hspace{-6cm}+\frac{1}{n}\left|\int_{1/n}^{\pi /2}\sin(\alpha _{2n+2}^{\mu -1}(\theta )+\theta)\cos(\alpha _{2m}^\lambda (\theta ))d\theta\right|+\frac{1}{n^2}\left|\int_{1/n}^{\pi /2}\frac{\sin (\alpha _{2n+2}^{\mu -1}(\theta )+\theta)\sin(\alpha _{2m}^\lambda (\theta )+\theta)}{\sin \theta }d\theta\right|\nonumber\\
&\hspace{-6cm}+\frac{1}{n^3}\left|\int_{1/n}^{\pi /2} \frac{\sin(\alpha _{2n+2}^{\mu -1}(\theta )+\theta)\cos(\alpha _{2m}^\lambda (\theta )+2\theta)}{(\sin \theta )^2}d\theta\right|+\frac{1}{n^2}\left|\int_{1/n}^{\pi /2}\frac{\cos (\alpha _{2n+2}^{\mu -1}(\theta )+2\theta)\cos(\alpha _{2m}^\lambda (\theta ))}{\sin \theta }d\theta\right|\nonumber\\
&
\hspace{-6cm}+\frac{1}{n^3}\left|\int_{1/n}^{\pi /2} \frac{\cos(\alpha _{2n+2}^{\mu -1}(\theta )+2\theta)\sin(\alpha _{2m}^\lambda (\theta )+\theta)}{(\sin \theta )^2}d\theta\right|+\frac{1}{n^4}\left|\int_{1/n}^{\pi /2}\frac{\cos (\alpha _{2n+2}^{\mu -1}(\theta )+2\theta)\cos(\alpha _{2m}^\lambda (\theta )+2\theta)}{(\sin \theta )^3}d\theta\right|\nonumber\\
&\hspace{-6.3cm}=:C\sum_{j=1}^9J_j(n,m).
\end{align}
By proceeding as in \eqref{I1} we get
\begin{equation}\label{J24}
J_2(n,m)+J_4(n,m)\leq \frac{C}{n|n-m|},
\end{equation}
and the method employed to estimate \eqref{I23} allows us also to obtain that
\begin{equation}\label{J357}
J_3(n,m)+J_5(n,m)+J_7(n,m)\leq \frac{C}{n|n-m|}.
\end{equation}
Also, we have that
\begin{equation}\label{J689}
J_6(n,m)+J_8(n,m)+J_9(n,m)\leq C\left(\frac{1}{n^3}\int_{1/n}^\infty \frac{d\theta}{\theta ^2}+\frac{1}{n^4}\int_{1/n}^\infty \frac{d\theta}{\theta ^3}\right)\leq \frac{C}{n^2}.
\end{equation}
Finally, it is not hard to see that
\begin{align*}
\cos (\alpha _{2n+2}^{\mu -1}(\theta ))\cos (\alpha _{2m}^\lambda (\theta ))\sin \theta &=-\sin (\alpha _{2n+1}^\mu (\theta ))\cos (\alpha _{2m}^\lambda (\theta ))\sin \theta\\
&\hspace{-2cm} =\frac{1}{4}\Big(\cos (\alpha _{2n+1}^\mu (\theta )+\alpha _{2m}^\lambda (\theta )+\theta)+\cos (\alpha _{2n+1}^\mu (\theta )-\alpha _{2m}^\lambda (\theta )+\theta)\\
&\hspace{-2cm}\quad -\cos (\alpha _{2n+1}^\mu (\theta )+\alpha _{2m}^\lambda (\theta )-\theta)-\cos(\alpha _{2n+1}^\mu (\theta )-\alpha _{2m}^\lambda (\theta )-\theta)\Big).
\end{align*} 
Then, we can write, when $2m+\lambda \not=2n+2+\mu$ and $2m+\lambda \not=2n+\mu$, 
\begin{align*}
J_1(n,m)&=\frac{1}{4}\left|\left[\frac{\sin(\alpha _{2n+1}^\mu (\theta )+\alpha _{2m}^\lambda (\theta )+\theta)}{2n+2+\mu+2m+\lambda }+\frac{\sin  (\alpha _{2n+1}^\mu (\theta )-\alpha _{2m}^\lambda (\theta )+\theta)}{2n+2+\mu-2m-\lambda }\right.\right.\\
&\quad \left.\left.-\frac{\sin (\alpha _{2n+1}^\mu (\theta )+\alpha _{2m}^\lambda (\theta )-\theta)}{2n+\mu+2m+\lambda}-\frac{\sin (\alpha _{2n+1}^\mu (\theta )-\alpha _{2m}^\lambda (\theta )-\theta)}{2n+\mu -2m-\lambda}\right]_{\theta =1/n}^{\theta =\pi /2}\right|\\
&=:\frac{1}{4}\left|\left[F_{n,m}(\theta )\right]_{\theta =1/n}^{\theta =\pi /2}\right|.
\end{align*}
We observe that $F_{n,m}(\pi/2)=0$ and also, we can write
\begin{align*}
F_{n,m}(\theta )&=\sin(\alpha _{2n+1}^\mu (\theta )+\alpha _{2m}^\lambda (\theta )+\theta)\left[\frac{1}{2n+2+\mu+2m+\lambda }-\frac{1}{2n+\mu+2m+\lambda }\right]\\
&+\frac{1}{2n+\mu+2m+\lambda }\Big[\sin(\alpha _{2n+1}^\mu (\theta )+\alpha _{2m}^\lambda (\theta )+\theta)-\sin (\alpha _{2n+1}^\mu (\theta )+\alpha _{2m}^\lambda (\theta )-\theta)\Big]\\
&+\sin(\alpha _{2n+1}^\mu (\theta )-\alpha _{2m}^\lambda (\theta )+\theta)\left[\frac{1}{2n+2+\mu -2m-\lambda}-\frac{1}{2n+\mu -2m-\lambda}\right]\\
&+\frac{1}{2n+\mu -2m-\lambda}\Big[\sin(\alpha _{2n+1}^\mu (\theta )-\alpha _{2m}^\lambda (\theta )+\theta)-\sin (\alpha _{2n+1}^\mu (\theta )-\alpha _{2m}^\lambda (\theta )-\theta)\Big].
\end{align*}
Hence, when $2m+\lambda \not=2n+2+\mu$ and $2m+\lambda \not=2n+\mu$, by using mean value theorem we obtain
\begin{align*}
J_1(n,m)&\leq C\left|F_{n,m}\Big(\frac{1}{m}\Big)\right|\\
&\leq C\left(\frac{1}{n^2}+\frac{1}{|2n+\mu+2-2m-\lambda||2n+\mu -2m-\lambda|}+\frac{1}{n|2n+\mu -2m-\lambda|}\right).
\end{align*}
Then, if $|n-m|>\max\{|\mu+2-\lambda |,|\mu -\lambda|\}$, we have
\begin{align*}
J_1(n,m)&\leq C\left(\frac{1}{n^2}+\frac{1}{(2|n-m|-|\mu +2-\lambda|)(2|n-m|-|\mu -\lambda|)}+\frac{1}{n(2|n-m|-|\mu -\lambda|)}\right)\\
&\leq\frac{C}{|n-m|^2}.
\end{align*}
On the other hand, if  $|n-m|\leq \max\{|\mu +2 -\lambda|,|\mu-\lambda |\}$,
$$
J_1(n,m)\leq C\int_{1/n}^{\pi /2}d\theta \leq C\leq \frac{C}{|n-m|^2},
$$
and we can deduce that
\begin{equation}\label{J1b}
J_1(n,m)\leq \frac{C}{|n-m|^2}.
\end{equation}
By joining estimations \eqref{H12}-\eqref{J1b} it follows that
\begin{equation}\label{H12d}
|H_{1,2}(n,m)|\leq \frac{C}{|n-m|^2}.
\end{equation}

Then, by taking into account \eqref{descomK}-\eqref{H1}, \eqref{H11} and \eqref{H12d} we conclude property \eqref{98}. The same procedure allows us to show (\ref{5.2}) for $K^{\rm o}_{\lambda,\mu}$.
\end{proof}

\begin{proof}[Proof of (\ref{5.3}) for $K^{\rm s}_{\lambda,\mu}$] We observe that
$$
K_{\lambda ,\mu }(n,m)=K_{\mu ,\lambda }(m,n),\quad n,m\in \mathbb{N}.
$$
Then, since $\lambda >1$, from \eqref{5.2} we obtain \eqref{5.3}. 
\end{proof}

Suppose that $(\mu -1)/2<\lambda <\mu$. According to Theorem \ref{Th2.1}, the operators $\mathcal{T}^e_{\lambda,\mu}$ and $\mathcal{T}^{\rm o}_{\lambda,\mu}$ are bounded

(a) from $\ell^1(\N,\omega)$ into $\ell^{1,\infty}(\N,\omega)$, for every $\omega \in A_1(\N)$;

(b) from $\ell^p(\N,\omega)$ into itself, for every $1<p<\infty$ and $\omega \in A_p(\N)$.
\vspace*{0.2cm}

Let $1\leq p< \infty$ and $\omega\in A_p(\N)$. We denote by $\mathbb{T}_{\lambda ,\mu }^e$ and $\mathbb{T}_{\lambda ,\mu} ^{\rm o}$ the bounded extensions to $\ell ^p(\mathbb{N}, w)$ of $\mathcal{T}_{\lambda ,\mu }^e$ and $\mathcal{T}_{\lambda ,\mu} ^{\rm o}$, respectively, and define $\widetilde w(k):=\omega(2k)$ and $\widehat\omega(k):=\omega(2k+1)$, $k\in \N$. Then, $\widetilde\omega,\widehat\omega\in A_p(\N)$. We consider the operator  $\mathbb{T}_{\lambda ,\mu }$ as follows:
$$\mathbb{T}_{\lambda ,\mu }(f)(n):=\left\{\begin{array}{ll}
			\mathbb{T}_{\lambda ,\mu }^e(\widetilde{f})(m),&\mbox{ if }n=2m,\\
			\mathbb{T}_{\lambda ,\mu }^{\rm o}(\widehat{f})(m),&\mbox{ if }n=2m+1,
			\end{array}
\right.,\quad f\in \ell ^p(\mathbb{N},w).
$$
Note that $\mathbb{T}_{\lambda ,\mu }(g)=\mathcal{T}_{\lambda ,\mu }(g)$, $g\in \ell ^2(\mathbb{N})$.

We have that, if $1<p<\infty$ and $f\in \ell^p(\N,\omega)$,
\begin{align*}
\sum_{n=0}^\infty \left|\mathbb{T}_{\lambda,\mu}(f)(n)\right|^p\omega(n) &=\sum_{n=0}^\infty \left|\mathbb{T}_{\lambda,\mu}(f)(2n)\right|^p\omega(2n)+\sum_{n=0}^\infty \left|\mathbb{T}_{\lambda,\mu}(f)(2n+1)\right|^p\omega(2n+1)\\
&= \sum_{n=0}^\infty \left|\mathbb{T}_{\lambda,\mu}^e(\widetilde f)(n)\right|^p\widetilde{\omega}(n) + \sum_{n=0}^\infty \left|\mathbb{T}_{\lambda,\mu}^{\rm o}(\widehat f)(n)\right|^p\widehat{\omega}(n)\\
&\leq C\left( \sum_{n=0}^\infty \left|\widetilde{f}(n)\right|^p\omega(2n)+\sum_{n=0}^\infty \left|\widehat {f}(n)\right|^p\omega(2n+1)\right)=C \sum_{n=0}^\infty|f(n)|^p\omega(n).
\end{align*}
Hence, $\mathbb{T}_{\lambda,\mu}$ is bounded from $\ell^p(\N,\omega)$ into itself, when $1<p<\infty$. In a similar way we can see that $\mathbb{T}_{\lambda,\mu}$ is bounded from $\ell^1(\N,\omega)$ into $\ell^{1,\infty}(\N,\omega)$.

Moreover, for every $f\in \ell ^p(\mathbb{N}, w)$,
$$
\mathbb{T}_{\lambda , \mu }(f)(n)=\lim_{k\rightarrow \infty }\sum_{m=0}^kK_{\lambda ,\mu }(n,m)f(m),
$$
where the convergence of the series can be understood in $\ell ^p(\mathbb{N},w)$, when $1<p<\infty$; in $\ell ^{1,\infty }(\mathbb{N},w)$, and also, pointwisely. This property justifies to write $\mathcal{T}_{\lambda ,\mu }=\mathbb{T}_{\lambda ,\mu }$ in $\ell ^p(\mathbb{N},w)$, $1\leq p<\infty$.

Suppose now that $1<\lambda<\mu$ and consider $r\in \N$ such that $\mu\in (\lambda+r,\lambda+r+1]$. We have that
$$
\mathcal{T}_{\lambda,\mu}=\mathcal{T}_{\lambda+r,\mu}\circ \mathcal{T}_{\lambda+r-1,\lambda +r}\circ \ldots \circ \mathcal{T}_{\lambda+1,\lambda+2}\circ \mathcal{T}_{\lambda,\lambda+1},\;\;\mbox{on}\;\ell^2(\N).
$$
Hence, if $1<p<\infty$ and $\omega \in A_p(\N)$, the operator $\mathcal{T}_{\lambda,\mu}$ is bounded from $\ell^p(\N,\omega)$ into itself. 

According to Plancherel equality we can write, for every $f,g\in \ell^2(\N)$,
$$
\sum_{n\in \N}(\mathcal{T}_{\mu,\lambda}f)(n)g(n)=  \int_{-1}^1\mathcal{F}_\mu(f)(x)\mathcal{F}_\lambda(g)(x)dx=\sum_{n\in \N}f(n)\mathcal{T}_{\lambda,\mu}(g)(n).
$$
Then, $\mathcal{T}_{\mu,\lambda}$ is bounded from $\ell^p(\N,\omega)$ into itself, for every $1<p<\infty$ and $\omega \in A_p(\N)$.

Let now $\lambda,\mu >1$, $1<p<\infty$ and $\omega \in A_p(\N)$. Since $\mathcal{T}_{\mu,\lambda}\mathcal{T}_{\lambda,\mu} f=f$, $f\in \ell^2(\N)$, we get
$$
\|f\|_{\ell^p(\N,\omega)} \leq C\|\mathcal{T}_{\lambda,\mu}f\|_{\ell^p(\N,\omega)},\;\;f\in \ell^p(\N,\omega),
$$
because $\mathcal{T}_{\mu,\lambda}$ and $\mathcal{T}_{\lambda,\mu}$ are bounded operators from $\ell^p(\N,\omega)$ into itself.

Thus the proof of this theorem is finished.

\def\cprime{$'$}



\begin{thebibliography}{10}

\bibitem{AM} H. Aimar and R. Mac\'{\i}as.
\newblock Weighted norm inequalities for the Hardy-Littlewood maximal operator on spaces of homogeneous type.
\newblock{\em Proc. Amer. Math. Soc.}, 91 (1984), 213--216.

\bibitem{ABCSS} V. Almeida, J.J. Betancor, A.J. Castro, A. Sanabria, and R. Scotto.
\newblock Variable exponent Sobolev spaces associated with Jacobi expansions.
arXiv: 1410.3642, 2014.

\bibitem{An1}
K.~F. Andersen.
\newblock Discrete {H}ilbert transforms and rearrangement invariant sequence
  spaces.
\newblock {\em Applicable Anal.}, 5(3) (1975/76), 193--200.

\bibitem{An2}
K.~F. Andersen.
\newblock Inequalities with weights for discrete {H}ilbert transforms.
\newblock {\em Canad. Math. Bull.}, 20(1) (1977), 9--16.

\bibitem{AH1}
R. Askey and I. Hirschman, Jr.
\newblock Weighted quadratic norms and ultraspherical polynomials. {I}.
\newblock {\em Trans. Amer. Math. Soc.}, 91 (1959), 294--313.

\bibitem{AW}
R. Askey and S. Wainger.
\newblock A transplantation theorem for ultraspherical coefficients.
\newblock {\em Pacific J. Math.}, 16 (1966), 393--405.

\bibitem{AC}
P.~Auscher and M.~J. Carro.
\newblock On relations between operators on {${\bf R}^N,\;{\bf T}^N$} and
  {${\bf Z}^N$}.
\newblock {\em Studia Math.}, 101(2) (1992), 165--182.

\bibitem{BM}
N. Badr and J.~M. Martell.
\newblock Weighted norm inequalities on graphs.
\newblock {\em J. Geom. Anal.}, 22(4) (2012), 1173--1210.

\bibitem{BR}
N. Badr and E. Russ.
\newblock Interpolation of {S}obolev spaces, {L}ittlewood-{P}aley inequalities
  and {R}iesz transforms on graphs.
\newblock {\em Publ. Mat.}, 53(2) (2009), 273--328.

\bibitem{BCP}
A.~Benedek, A.-P. Calder{\'o}n, and R.~Panzone.
\newblock Convolution operators on {B}anach space valued functions.
\newblock {\em Proc. Nat. Acad. Sci. U.S.A.}, 48 (1962), 356--365.

\bibitem{BCCFR}
J.~J. Betancor, A.~J. Castro, J.~Curbelo, J.~C. Fari{\~n}a, and
  L.~Rodr{\'{\i}}guez-Mesa.
\newblock Square functions in the {H}ermite setting for functions with values
  in {UMD} spaces.
\newblock {\em Ann. Mat. Pura Appl. (4)}, 193(5) (2014), 1397--1430.

\bibitem{BFRTT}
J.~J. Betancor, J.~C. Fari{\~n}a, L. Rodr{\'{\i}}guez-Mesa, R.
  Testoni, and J.~L. Torrea.
\newblock Fractional square functions and potential spaces.
\newblock {\em J. Math. Anal. Appl.}, 386 (2012), 487--504.

\bibitem{BH}
W.~R. Bloom and H. Heyer.
\newblock {\em Harmonic analysis of probability measures on hypergroups},
  volume~20 of {\em de Gruyter Studies in Mathematics}.
\newblock Walter de Gruyter \& Co., Berlin, 1995.

\bibitem{B1}
J.~Bourgain.
\newblock On the maximal ergodic theorem for certain subsets of the integers.
\newblock {\em Israel J. Math.}, 61(1) (1988), 39--72.

\bibitem{B2}
J.~Bourgain.
\newblock On the pointwise ergodic theorem on {$L^p$} for arithmetic sets.
\newblock {\em Israel J. Math.}, 61(1) (1988), 73--84.

\bibitem{CZ}
A.~P. Calder\'on and A.~Zygmund.
\newblock On the existence of certain singular integrals.
\newblock {\em Acta Math.}, 88 (1952), 85--139.

\bibitem{CGRTV}
{O}.~{Ciaurri}, T.~A. {Gillespie}, L.~{Roncal}, J.~L. {Torrea}, and J.~L.
  {Varona}.
\newblock {Harmonic Analysis associated with a discrete Laplacian}. To appear in Journal d'Analyse Mathematique, 
arXiv: 1401.2091v2, 2014.

\bibitem{DP}
K. Domelevo and S. Petermichl.
\newblock Sharp {$L^p$} estimates for discrete second-order {R}iesz transforms.
\newblock {\em C. R. Math. Acad. Sci. Paris}, 352(6) (2014), 503--506.

\bibitem{D}
N. Dungey.
\newblock Riesz transforms on a discrete group of polynomial growth.
\newblock {\em Bull. London Math. Soc.}, 36(6) (2004), 833--840.

\bibitem{Duo}
J. Duoandikoetxea.
\newblock {\em Fourier analysis}, volume~29 of {\em Graduate Studies in
  Mathematics}.
\newblock American Mathematical Society, Providence, RI, 2001.

\bibitem{EOT}
J.~S. Ellenberg, R. Oberlin, and T. Tao.
\newblock The {K}akeya set and maximal conjectures for algebraic varieties over
  finite fields.
\newblock {\em Mathematika}, 56(1) (2010), 1--25.

\bibitem{EMOT}
A. Erd{\'e}lyi, W. Magnus, F. Oberhettinger, and F.~G.
  Tricomi.
\newblock {\em Higher transcendental functions. {V}ols. {I}, {II}}.
\newblock McGraw-Hill Book Company, Inc., New York-Toronto-London, 1953.

\bibitem{GT}
T.~A. Gillespie and J.~L. Torrea.
\newblock Weighted ergodic theory and dimension free estimates.
\newblock {\em The Quaterly J. Math. Oxford}, 54 (2003), 257--280.

\bibitem{GLLNU}
P. Graczyk, J.~J. Loeb, I.~A. L{\'o}pez~P., A. Nowak, and W.~O.
  Urbina.
\newblock Higher order {R}iesz transforms, fractional derivatives, and
  {S}obolev spaces for {L}aguerre expansions.
\newblock {\em J. Math. Pures Appl. (9)}, 84(3) (2005), 375--405.

\bibitem{Gra}
L. Grafakos.
\newblock An elementary proof of the square summability of the discrete
  {H}ilbert transform.
\newblock {\em Amer. Math. Monthly}, 101(5) (1994), 456--458.

\bibitem{GLY}
L. Grafakos, L. Liu, and D. Yang.
\newblock Vector-valued singular integrals and maximal functions on spaces of
  homogeneous type.
\newblock {\em Math. Scand.}, 104(2) (2009), 296--310.

\bibitem{Hi1}
I.~I. Hirschman, Jr.
\newblock Variation diminishing transformations and ultraspherical polynomials.
\newblock {\em J. Analyse Math.}, 8 (1960/1961), 337--360.

\bibitem{Hsu} H.~Y. Hs\"u, \newblock Certain integrals and infinite series involving ultraspherical polynomials and Bessel functions, \newblock{\em Duke Math. J.}, 4 (1938), 374--383.

\bibitem{HMW}
R. Hunt, B. Muckenhoupt, and R. Wheeden.
\newblock Weighted norm inequalities for the conjugate function and {H}ilbert
  transform.
\newblock {\em Trans. Amer. Math. Soc.}, 176 (1973), 227--251.

\bibitem{IMSW}
A.~D. Ionescu, E.~M. Stein, A. Magyar, and S. Wainger.
\newblock Discrete {R}adon transforms and applications to ergodic theory.
\newblock {\em Acta Math.}, 198(2) (2007), 231--298.

\bibitem{IW}
A.~D. Ionescu and S. Wainger.
\newblock {$L^p$} boundedness of discrete singular {R}adon transforms.
\newblock {\em J. Amer. Math. Soc.}, 19(2) (2006), 357--383.

\bibitem{Ke}
R.~A. Kerman.
\newblock Strong and weak weighted convergence of {J}acobi series.
\newblock {\em J. Approx. Theory}, 88(1) (1997), 1--27.

\bibitem{K}
P.~H. Koester.
\newblock Harmonic analysis on finite abelian groups. 

\newblock {\em Avaible in www.ms.uky.edu/~pkoester/research/finiteabeliansanalysis.pdf}.

\bibitem{KM}
I. Kyrezi and M. Marias.
\newblock {$H^1$}-bounds for spectral multipliers on graphs.
\newblock {\em Proc. Amer. Math. Soc.}, 132(5) (2004), 1311--1320.

\bibitem{LT}
J. Lindenstrauss and L. Tzafriri.
\newblock {\em Classical Banach spaces II}.
\newblock Springer Verlag, Berlin, 1979.

\bibitem{La}
E. Laeng.
\newblock Remarks on the {H}ilbert transform and on some families of multiplier
  operators related to it.
\newblock {\em Collect. Math.}, 58(1) (2007), 25--44.

\bibitem{Leb}
N.~N. Lebedev.
\newblock {\em Special functions and their applications}.
\newblock Dover Publications Inc., New York, 1972.

\bibitem{L2}
F. Lust-Piquard.
\newblock Riesz transforms associated with the number operator on the {W}alsh
  system and the fermions.
\newblock {\em J. Funct. Anal.}, 155(1) (1998), 263--285.

\bibitem{L1}
F. Lust-Piquard.
\newblock Dimension free estimates for discrete {R}iesz transforms on products
  of abelian groups.
\newblock {\em Adv. Math.}, 185(2) (2004), 289--327.

\bibitem{NS}
A. Nowak and K. Stempak.
\newblock Weighted estimates for the {H}ankel transform transplantation
  operator.
\newblock {\em Tohoku Math. J. (2)}, 58(2) (2006), 277--301.

\bibitem{Ri}
M. Riesz.
\newblock Sur les fonctions conjugu\'ees.
\newblock {\em Math. Z.}, 27(1) (1928), 218--244.

\bibitem{RRT}
J.~L. Rubio~de Francia, F.~J. Ruiz, and J.~L. Torrea.
\newblock Calder\'on-{Z}ygmund theory for operator-valued kernels.
\newblock {\em Adv. in Math.}, 62(1) (1986), 7--48.

\bibitem{R1}
E. Russ.
\newblock Riesz transforms on graphs for {$1\leq p\leq 2$}.
\newblock {\em Math. Scand.}, 87(1) (2000), 133--160.

\bibitem{R2}
E. Russ.
\newblock {{$H^1$}-{$L^1$} boundedness of {R}iesz transforms on {R}iemannian
  manifolds and on graphs.}
\newblock {\em Potential Anal.}, 14(3) (2001), 301--330.

\bibitem{Stein} 
E.~M. Stein.
\newblock {On the maximal ergodic theorem.}
\newblock {\em Pure Nat. Acad. Sci.}, 47(12) (1961), 1894--1897.

\bibitem{StLP}
E.~M. Stein.
\newblock {\em Topics in harmonic analysis related to the {L}ittlewood-{P}aley
  theory}.
\newblock Annals of Mathematics Studies, No. 63. Princeton University Press,
  Princeton, N.J., 1970.

\bibitem{SW1}
E.~M. Stein and S.~Wainger.
\newblock Discrete analogues of singular {R}adon transforms.
\newblock {\em Bull. Amer. Math. Soc. (N.S.)}, 23(2) (1990), 537--544.

\bibitem{SW2}
E.~M. Stein and S. Wainger.
\newblock Discrete analogues in harmonic analysis. {I}. {$l^2$} estimates for
  singular {R}adon transforms.
\newblock {\em Amer. J. Math.}, 121(6) (1999), 1291--1336.

\bibitem{SW3}
E.~M. Stein and S.~Wainger.
\newblock Discrete analogues in harmonic analysis. {II}. {F}ractional
  integration.
\newblock {\em J. Anal. Math.}, 80 (2000), 335--355.

\bibitem{Stem}
K. Stempak.
\newblock Jacobi conjugate expansions.
\newblock {\em Studia Sci. Math. Hungar.}, 44(1) (2007), 117--130.

\bibitem{Sz}
G. Szeg{\"o}.
\newblock {\em Orthogonal polynomials}.
\newblock American Mathematical Society, Providence, R.I., fourth edition,
  1975.
\newblock American Mathematical Society, Colloquium Publications, Vol. XXIII.

\bibitem{Sz2}
G. {Szeg\"o}.
\newblock {Asymptotische Entwicklungen der Jacobischen Polynome.}
\newblock {Schr. K\"onigsberg. Gel. Ges. Jahr 10, Naturw. Klasse}, Heft 3 (1933), 35--111.

\bibitem{Ti}
E.~C. Titchmarsh.
\newblock {Reciprocal formulae involving series and integrals}, \newblock{\em Math. Z.}, 25 (1926), 321--347.

\end{thebibliography}
\end{document}